\documentclass[11pt,a4paper,reqno]{amsart}
\usepackage{vmargin,color}
\usepackage[latin1]{inputenc}
\usepackage{amsmath,amsfonts,amsthm,epsfig,graphicx}
\usepackage{pstricks,pst-plot,multido}
\usepackage{caption}
\usepackage{esint} 
\usepackage{pgfplots}
\usepackage{graphicx,tikz,subfigure,color,calc}
\usetikzlibrary{decorations.pathreplacing}

\pgfplotsset{compat=1.10}
\usepgfplotslibrary{fillbetween}
\usetikzlibrary{patterns}

\definecolor{darkgreen}{rgb}{0.0,0.5,0.0}
\definecolor{darkblue}{rgb}{0.0,0.0,0.3}
\definecolor{nicosred}{rgb}{0.65,0.1,0.1}
\definecolor{light-gray}{gray}{0.6}
\definecolor{really-light-gray}{gray}{0.8}

\def\H{\mathcal{H}}

\def\N{\mathbb N}

\def\R{\mathbb R}

\def\e{\varepsilon}

\def\om{\omega}
\def\vphi{\varphi}

\def\pa{\partial}

\def\pae{\partial^{{\rm e}}}

\newcommand{\supp}{\mathrm{supp}} 

\newcommand{\diver}{\mathrm{div}} 
\newcommand{\mres}{\mathbin{\vrule height 1.6ex depth 0pt width  
0.13ex\vrule height 0.13ex depth 0pt width 1.3ex}}    

\DeclareMathOperator*{\aplim}{ap\,lim}

\def\big{\bigskip}

\newtheorem{theorem}{Theorem}[section]
\newtheorem{remark}[theorem]{Remark}

\newtheorem{proposition}[theorem]{Proposition}
\newtheorem{lemma}[theorem]{Lemma}
\newtheorem{corollary}[theorem]{Corollary}

\numberwithin{equation}{section}
\numberwithin{figure}{section}

\begin{document}

\title{Rigidity for perimeter inequality \\
under spherical symmetrisation}

\author{F. Cagnetti}
\address{Department of Mathematics, University of Sussex, Pevensey 2, BN1 9QH, Brighton, UK}%
\email{f.cagnetti@sussex.ac.uk}

\author{M. Perugini}
\address{Department of Mathematics, University of Sussex, Pevensey 2, BN1 9QH, Brighton, UK}%
\email{m.perugini@sussex.ac.uk}

\author{D. St\"oger}
\address{Technische Universit\"at M\"unchen, Zentrum Mathematik - M15, 
Boltzmannstrasse 3, 85747 Garching, Germany}
\email{dominik.stoeger@ma.tum.de}

\maketitle

\begin{abstract}
{\rm Necessary and sufficient conditions for
rigidity of the perimeter inequality under spherical symmetrisation are given.
That is, a characterisation for the uniqueness (up to orthogonal transformations) 
of the extremals is provided.
This is obtained through a careful analysis of the equality cases, 
and studying fine properties of the circular symmetrisation, 
which was firstly introduced by P\'olya in 1950.}
\end{abstract}

\section{Introduction}
In this paper we study the perimeter inequality under spherical symmetrisation, 
giving necessary and sufficient conditions for the uniqueness, up to orthogonal transformations, of the extremals. 
Perimeter inequalities under symmetrisation have been studied by many authors, 
see for instance \cite{kawohl_book_85, kawohl86} and the references therein.
In general, we say that rigidity holds true for one of these inequalities if the set of extremals is trivial.
The study of rigidity can have important applications to show that minimisers
of  variational problems (or solutions of  PDEs) are symmetric.

For instance, a crucial step in the proof of the Isoperimetric Inequality given by Ennio De Giorgi
consists in showing rigidity of Steiner's inequality (see, for instance, \cite[Theorem~14.4]{maggiBOOK})
for convex sets (see the proof of Theorem I in Section~4 in 
\cite{DeGiorgi58ISOP, DeGiorgiSelected}).
After De Giorgi, an important contribution in the understanding of rigidity for Steiner's 
inequality was given by Chleb{\'{\i}}k, Cianchi, and Fusco.
In the seminal paper \cite{ChlebikCianchiFuscoAnnals05}, 
the authors give sufficient conditions for rigidity which are much more general than convexity.
After that, this result was extended to the case of higher codimensions
in \cite{barchiesicagnettifusco}, where a quantitative version of Steiner's inequality  was  also given.
 
Then, necessary and sufficient conditions for rigidity (in codimension $1$)
were given in \cite{CagnettiColomboDePhilippisMaggiSteiner},
in the case where the distribution function is a Special Function of Bounded Variation
with locally finite jump set \cite[Theorem~1.29]{CagnettiColomboDePhilippisMaggiSteiner}.
The anisotropic case has recently been considered in \cite{Perugini}, 
where rigidity for Steiner's inequality in the isotropic and anisotropic
setting are shown to be equivalent, under suitable conditions.
In the Gaussian setting, where the  role of Steiner's inequality
is played  by Ehrhard's inequality (see \cite[Section~4.1]{cianchifuscomaggipratelliGAUSS}), 
necessary and sufficient conditions for rigidity are given in \cite{ccdpmGAUSS}, by making use of the notion of essential connectedness \cite[Theorem~1.3]{ccdpmGAUSS}.
 Finally, in the smooth case, sufficient conditions for rigidity are given in 
\cite[Proposition~5]{MorganHoweHarman2011},
for a general class of symmetrisations in warped products.

\medskip

The main motivation for the study of the spherical symmetrisation 
is that it can be used to understand the symmetry properties 
of the solutions of  PDEs and variational problems, 
when the radial symmetry has been ruled out.
Moreover, some well established methods 
(as for instance the moving plane method, see \cite{Serrin71, GNN79})  
rely on convexity properties of the domain which fail, for instance, 
when one deals with annuli. 

In particular, in many applications minimisers of variational problems 
and solutions of PDEs turn out to be \textit{foliated Schwarz symmetric}.
Roughly speaking, a function $u: \R^n \to \R$ is foliated Schwarz symmetric
if one can find a direction $p \in \mathbb{S}^{n-1}$
such that $u$ only depends on $|x|$ and on the polar angle
$\alpha = \arccos (\hat{x} \cdot p)$, and $u$ is non increasing with respect to $\alpha$
(here $\hat{x}:= x/|x|$, and $| \cdot |$ denotes the Euclidean norm in $\R^n$). 
We direct the interested reader to \cite{BWW05, Baernstein95, BirindelliLeoniPacella17, SmetsWillem03} 
and the references therein for more information.

%
%
%
%
%
%
%
%
%

\subsection{Spherical Symmetrisation}
To the best of our knowledge, the spherical symmetrisation was first introduced 
by P\'olya in \cite{polya50}, in the case  $n=2$ and in the smooth setting. 
Let $n \in \N$ with $n \geq 2$. For each $r > 0$ and $x \in \R^n$, 
we denote by  $B (x, r )$  the open ball of $\R^n$ of radius $r$ centred at
$x$, by $\omega_n$ the ($n$-dimensional) volume of the unit ball,
and we write $B (r)$ for $B (0, r )$. 
Moreover, $e_1, \ldots, e_n$ stand for the vectors of the canonical basis of $\R^n$.
Given a set $E \subset \R^n$ and $r > 0$, we define the \textit{spherical slice} $E_r$ 
of $E$ with respect to $\partial B (r)$ as
$$
E_r := E \cap \partial B (r) = \{ x \in E : |x| = r \}.
$$
%
Let $v : (0, \infty) \to [0, \infty)$ be a measurable function.
We say that $E$ is \textit{spherically $v$-distributed} if 
\begin{equation} \label{def v}
v (r) = \mathcal{H}^{n-1} (E_r), \qquad \text{ for } \mathcal{H}^1\text{-a.e. } r \in (0, \infty),
\end{equation}
where $\mathcal{H}^{k}$ denotes the $k$-dimensional 
Hausdorff measure of $\R^n$, $1 \leq k \leq n$. 
Note that, in order $v$ to be an admissible distribution, one needs 
\begin{equation} \label{bound on v}
v (r) \leq \mathcal{H}^{n-1} (\partial B (r)) = n \omega_n r^{n-1} 
\qquad \text{ for $\mathcal{H}^1$-a.e. } r \in (0, \infty).
\end{equation}
In the following, as usual, we set $\mathbb{S}^{n-1} = \partial B (1)$.
For every $x, y \in \mathbb{S}^{n-1}$, the \textit{geodesic distance} 
between $x$ and $y$ is given by 
$$
\text{dist}_{\mathbb{S}^{n-1}} (x, y) := \arccos (x \cdot y).
$$
Let $r > 0$, $p \in \mathbb{S}^{n-1}$, and $\beta \in [0, \pi]$ be fixed. 
The \textit{open geodesic ball} (or \textit{spherical cap}) 
of centre $r p$ and radius $\beta$ is the set 
$$
\mathbf{B}_{\beta} (r p) := \{ x \in \partial B (r) : \text{dist}_{\mathbb{S}^{n-1}} (\hat{x}, p) < \beta \}.
$$
The $(n-1)$-dimensional Hausdorff measure of $\mathbf{B}_{\beta} (r p)$
can be explicitly calculated, and is given by
$$
\mathcal{H}^{n-1} (\mathbf{B}_{\beta} (r p)) = (n-1) \omega_{n-1} r^{n-1}
\int_0^{\beta} (\sin \tau)^{n-2} \, d \tau.
$$
The expression above shows that the function 
$\beta \mapsto \mathcal{H}^{n-1} (\mathbf{B}_{\beta} (r p))$
is strictly increasing from $[0, \pi]$ to $[0, n \omega_n r^{n-1}]$.
Therefore, if $v : (0, \infty) \to [0, \infty)$ is a measurable
function satisfying \eqref{bound on v}, and $E \subset \R^n$ is a spherically $v$-distributed set, 
there exists only one (defined up to a subset of zero $\mathcal{H}^1$-measure) measurable function $\alpha_v: (0,\infty) \to [0, \pi]$
satisfying
\begin{equation} \label{this is the def of alphav}
v (r) = \mathcal{H}^{n-1} (\mathbf{B}_{\alpha_v (r)} (r e_1)) 
\qquad \text{ for $\mathcal{H}^1$-a.e. } r \in (0, \infty).
\end{equation}
Among all the spherically $v$-distributed sets of $\R^n$, we denote by $F_v$ the one whose spherical slices
are open geodesic balls centred at the positive $e_1$ axis., i.e.
\begin{equation*} 
F_v := \{ x \in \R^n \setminus \{ 0 \} : \text{dist}_{\mathbb{S}^{n-1}} (\hat{x}, e_1) < \alpha_v (|x|) \},
\end{equation*}
see Figure~\ref{figure 3D}.
\begin{figure}[!htb]
\centering 
\def
\svgwidth{13cm} 
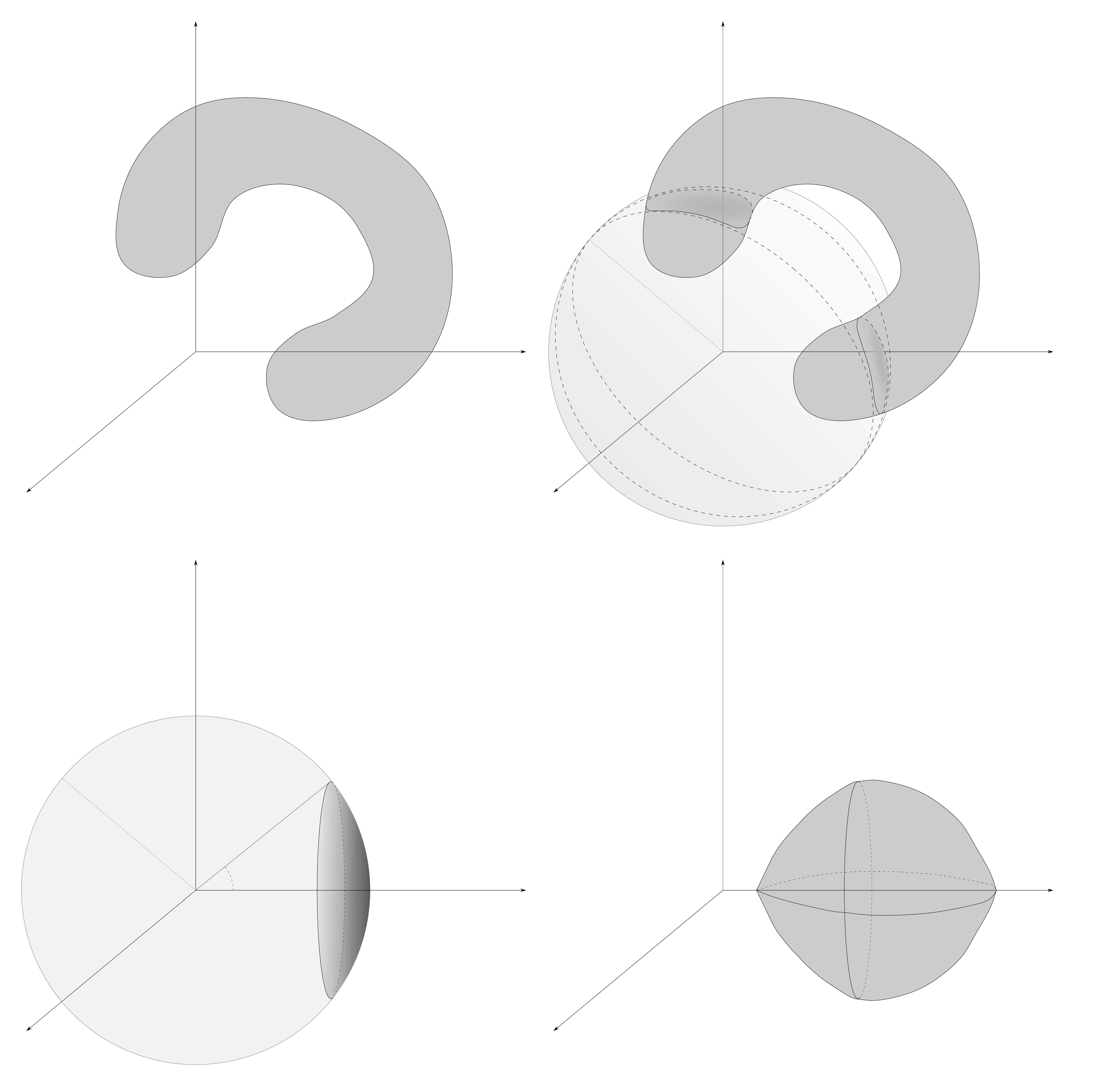 
\caption{A pictorial idea of the spherical symmetral $F_v$ of a $v$-distributed set $E$, 
in the case $n =3$.} 
\label{figure 3D}
\end{figure}
Before stating our results, it will be convenient to recall some basic notions
about sets of finite perimeter.

\subsection{Basic notions on sets of finite perimeter}

Let $E \subset \R^n$ be a measurable set, and let $t \in [0, 1]$.
We denote by $E^{(t)}$ the set of points of density $t$ of $E$, given by 
$$
E^{(t)} : = \left\{ x \in \R^n : \lim_{\rho \to 0^+ } 
\frac{\mathcal{H}^n (E \cap B (x, \rho))}{\omega_n \rho^n} = t \right\}.
$$
The essential boundary of $E$ is then defined as 
$$
\partial^{\textnormal{e}} E:= E \setminus (E^{(1)} \cup E^{(0)}).
$$
Moreover, if $A \subset \R^n$ is any Borel set, we define the perimeter of $E$
relative to $A$ as the extended real number given by
$$
P (E; A) := \mathcal{H}^{n-1} (\partial^{\textnormal{e}} E \cap A), 
$$
and we set $P(E):= P (E; \R^n)$.
When $E$ is a set with smooth boundary, it turns out that 
$\partial^{\textnormal{e}} E = \partial E$, and the perimeter of $E$
agrees with the usual notion of $(n-1)$-dimensional surface measure of $\partial E$. 

If $P(E) < \infty$, it is possible to define the reduced boundary $\partial^*E$ of $E$. 
This has  the  property that $\partial^*E \subset \partial^{\textnormal{e}} E$, 
$\mathcal{H}^{n-1} (\partial^{\textnormal{e}} E \setminus \partial^*E) = 0$, 
and is such that for every $x \in \partial^* E$ there exists the \textit{measure theoretic outer unit normal} 
$\nu^E (x)$ of $\partial^* E$ at $x$, see Section~\ref{preliminaries}.
If $x \in \partial^* E$, it will be convenient to decompose $\nu^{E} (x)$ as 
$$
\nu^{E} (x) = \nu^{E}_{\perp} (x) + \nu^{E}_{\parallel} (x),
$$
where $\nu^{E}_{\perp} (x) := (\nu^{E} (x) \cdot \hat{x} ) \hat{x}$ 
and $\nu^{E}_{\parallel} (x)$ are the radial and tangential component 
of $\nu^{E} (x)$ along $\partial B (|x|)$, respectively. 
In the following, we will use the diffeomorphism
$\Phi : (0, \infty) \times \mathbb{S}^{n-1} \to \mathbb{R}^n \setminus \{ 0 \}$ 
defined as
$$
\Phi (r, \omega) := r \omega \qquad \text{ for every } (r, \omega) \in (0, \infty) \times \mathbb{S}^{n-1}.
$$

\subsection{Perimeter Inequality under spherical symmetrisation}

Our first result shows that the spherical symmetrisation does not increase the perimeter, 
and gives some necessary conditions for equality cases.
In our analysis we require the set $F_v$ (or, equivalently, any spherically $v$-distributed set)
to have finite volume. This is not restrictive.
Indeed, if $F_v$ has finite perimeter but infinite volume, 
we can consider the complement $\R^n \setminus F_v$
which, by the relative isoperimetric inequality, has finite volume.
This change corresponds to considering the complementary distribution 
function $r \mapsto  n \omega_n r^{n-1}  - v (r)$, and the spherical symmetrisation 
with respect to the axis $-e_1$. 

\begin{theorem} \label{fv locally finite perimeter}
Let $v: (0, \infty) \to [0, \infty)$ be a measurable function
satisfying \eqref{bound on v}, and let $E \subset \R^n$ be a 
spherically $v$-distributed set of finite perimeter and finite volume.
Then, $v \in BV (0, \infty)$. 
Moreover, $F_v$ is a set of finite perimeter and 
\begin{equation} \label{per ineq}
P (F_v; \Phi (B \times \mathbb{S}^{n-1})) \leq P (E; \Phi (B \times \mathbb{S}^{n-1})),
\end{equation}
for every Borel set $B \subset (0, \infty)$. 

Finally, if $P (E) = P(F_v)$, 
then for $\mathcal{H}^1$-a.e. $r \in \{ 0 < \alpha_v <  \pi \}$:
\begin{itemize}

\item[(a)] $E_r$ is $\mathcal{H}^{n-1}$-equivalent 
to a spherical cap and $\mathcal{H}^{n-2} (\partial^* (E_r) \Delta (\partial^* E)_r) = 0$;

\vspace{.1cm}

\item[(b)] the functions $x \mapsto \nu^E (x) \cdot \hat{x}$ and $x \mapsto | \nu^E_{\parallel}| (x)$
are constant $\mathcal{H}^{n-2}$-a.e. in $(\partial^* E )_r$.

\end{itemize}
\end{theorem} 
The result above shows that the perimeter inequality holds on a local level, 
provided one considers sets of the type $\Phi (B \times \mathbb{S}^{n-1})$, 
with $B \subset (0,\infty)$ Borel.
Inequality \eqref{per ineq} is very well known in the literature.
In the special case $n=2$,
a short proof was given by P\'olya in \cite{polya50}.
In the general $n$-dimensional case with $B = (0,\infty)$ the result is stated
 in  \cite[Theorem~6.2]{MorganPratelli2013}), but the proof is only sketched
 (see also \cite{McGillivray18} and \cite[Proposition~3 and Remark~4]{MorganHoweHarman2011}). 
As mentioned by Morgan and Pratelli in \cite{MorganPratelli2013}, 
certain parts of the proof of \eqref{per ineq} follow the general lines of analogous results in the context of Steiner
symmetrisation (see, for instance, \cite[Lemma~3.4]{ChlebikCianchiFuscoAnnals05} and
\cite[Theorem~1.1]{barchiesicagnettifusco}).
There are, however, non trivial technical difficulties that arise when one deals 
with the spherical symmetrisation.
For this reason, we give a detailed proof of Theorem~\ref{fv locally finite perimeter}.

\medskip

We start by introducing radial and tangential components 
of a Radon measure, see Section~\ref{subsection tangential}. 
These turn out to be useful tools which allow to prove several preliminary results.
Moreover, since we are dealing with a symmetrisation of codimension $n-1$, 
we need to pay attention to some delicate effects that are not usually observed 
when the codimension is $1$ (as, for instance, in \cite{ChlebikCianchiFuscoAnnals05}). 
Indeed, a crucial role is played by the measure $\lambda_E$ given by:
\begin{equation} \label{measure mu abs cont}
\lambda_E (B) :=\int_{\partial^* E \cap \Phi(B\times \mathbb{S}^{n-1})\cap\{\nu^{E}_{\|}=0 \}}
\hat{x} \cdot \nu^E(x)  \, d\mathcal{H}^{n-1}(x), 
\end{equation}
for every Borel set $B \subset (0, \infty)$.
When $n = 2$, it turns out that $\lambda_E$ is singular with respect to the Lebesgue measure
in $(0,\infty)$. However, for $n > 2$ it may happen that $\lambda_E$ contains 
a non trivial absolutely continuous part, see Remark~\ref{remark gmt}. 
This requires some extra care while proving inequality \eqref{per ineq}. 
A similar phenomenon has already been observed 
in \cite{barchiesicagnettifusco}, in the study of the Steiner symmetrisation 
of codimension higher than $1$.
Higher codimension effects play an important role 
also in the study of rigidity, as explained below.

\subsection{Rigidity of the Perimeter Inequality}
Given $v: (0,\infty) \to [0, \infty)$ measurable, 
satisfying \eqref{bound on v}, and such that $F_v$ 
is a set of finite perimeter and finite volume, 
we define $\mathcal{N} (v)$ as the class of extremals of \eqref{per ineq}:
$$
\mathcal{N} (v):= \{ E \subset \R^n :  E \text{ is spherically $v$-distributed and }
P(E) = P(F_v)  \}. 
$$
Note that, by definition of $F_v$, and by the invariance of the perimeter under
rigid transformations, every time we apply an orthogonal transformation to $F_v$
we obtain a set that belongs to $\mathcal{N} (v)$, i.e.:
\begin{equation*}
 \mathcal{N} (v) \supset \{ E \subset \R^n :  
\mathcal{H}^n ( E \Delta (R \, F_v) ) =  0 \text{ for some } R \in O (n) \},
\end{equation*}
where $\Delta$ denotes the symmetric difference of sets and
$O (n)$ is the set of orthogonal transformations in $\R^n$. 
We would like to understand when also the opposite inclusion is satisfied, 
that is, when the class of extremals of \eqref{per ineq}
is just given by rotated copies of $F_v$.
We will say that \textit{rigidity} holds true for inequality \eqref{per ineq} if 
\begin{equation}
\mathcal{N} (v) = \{ E \subset \R^n :  
\mathcal{H}^n ( E \Delta (R \, F_v) ) =  0 \text{ for some } R \in O (n) \}.
\tag{$\mathcal{R}$}
\end{equation}
In order to explain which conditions we should expect 
in order ($\mathcal{R}$) to be true, let us first give some examples.

\medskip

Figure~\ref{case iii counterexample true} shows a set $E \in \mathcal{N}(v)$
that cannot be obtained by applying a single orthogonal transformation to $F_v$.
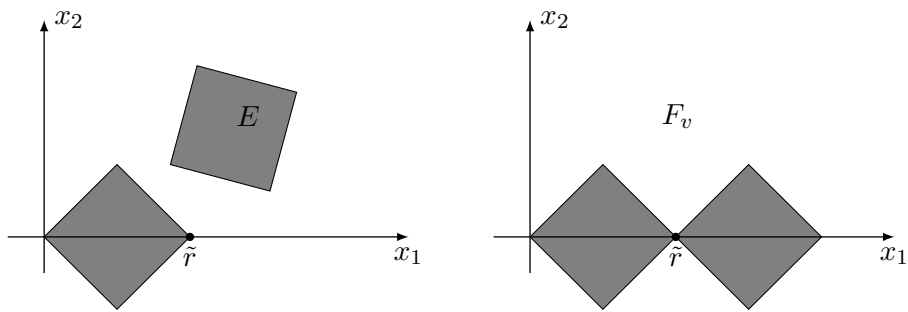
\begin{figure}[!htb]
    \centering
\begin{tikzpicture}[>=latex, scale=1.6]

 \draw[fill=gray] (0,0)--(.6,.6)--(1.2,0)--(.6,-.6)--(0,0);
  
     \draw(1.2,0)node[below]{$\tilde{r}$};
\draw[fill=black] (1.2,0) circle (.03cm);
 
 \begin{scope}[rotate around={30:(0,0)}]
 \draw[fill=gray] (0+1.2,0)--(.6+1.2,.6)--(1.2+1.2,0)--(.6+1.2,-.6)--(0+1.2,0);
\end{scope}
   
     \draw[->, line width = .5pt] (-.3,0)--(3,0); 
  \draw(3,0)node[below]{$x_1$};
 \draw[->, line width = .5pt] (0,-.3)--(0,1.8); 
  \draw(0,1.8)node[right]{$x_2$};
  
 \draw(1.5,1)node[right]{$E$};

  \draw[fill=gray] (0+4,0)--(.6+4,.6)--(1.2+4,0)--(.6+4,-.6)--(0+4,0);
  \draw[fill=gray] (0+4+1.2,0)--(.6+4+1.2,.6)--(1.2+4+1.2,0)--(.6+4+1.2,-.6)--(0+4+1.2,0);
 
     \draw[->, line width = .5pt] (-.3+4,0)--(3+4,0); 
  \draw(3+4,0)node[below]{$x_1$};
 \draw[->, line width = .5pt] (0+4,-.3)--(0+4,1.8); 
  \draw(0+4,1.8)node[right]{$x_2$}; 
   
      \draw(4+1.2,0)node[below]{$\tilde{r}$};
\draw[fill=black] (4+1.2,0) circle (.03cm);

 \draw(5,1)node[right]{$F_v$};

\end{tikzpicture}
   \caption{Rigidity ($\mathcal{R}$) fails, since the set $\{ 0 < \alpha_v < \pi \}$ is disconnected
   by a point $\tilde{r} \in (0, \infty)$ such that $\alpha_v (\tilde{r}) = 0$.}
	\label{case iii counterexample true}
\end{figure}
This is due to the fact that the set $\{ 0 < \alpha_v < \pi \}$ 
is disconnected by a point $\tilde{r}$ satisfying $\alpha_v (\tilde{r}) = 0$.
A similar situation happens when $\{ 0 < \alpha_v < \pi \}$ 
is disconnected by points belonging to the set $\{ \alpha_v = \pi \}$, 
see Figure~\ref{case alpha = pi}.


\begin{figure}[!htb]
    \centering
\begin{tikzpicture}[>=latex, scale=1.6]
 
  \begin{scope}[rotate around={25:(0,0)}]
 \draw[fill=gray, line width = .5pt] (1,0) ellipse (1.7 and 1.1);
  \end{scope}

  \draw[dashed, line width = .5pt] (0,0) circle (.7);

 \draw[fill=white, line width = .5pt] (-.35,0) ellipse (.35 and .2);

\draw[fill=black] (.7,0) circle (.03cm);
 \draw(.8,0)node[below]{$\hat{r}$};

%
%
%
     \draw[->, line width = .5pt] (-1,0)--(3,0); 
  \draw(3,0)node[below]{$x_1$};
 \draw[->, line width = .5pt] (0,-.3)--(0,1.8); 
  \draw(0,1.8)node[right]{$x_2$};
 \draw(1.5,.6)node[right]{$E$};

 \draw[fill=gray, line width = .5pt] (1+4.5,0) ellipse (1.7 and 1.1);

   \draw[dashed, line width = .5pt] (+4.5,0) circle (.7);

 \draw[fill=white, line width = .5pt] (-.35+4.5,0) ellipse (.35 and .2);

\draw[fill=black] (.7+4.5,0) circle (.03cm);
 \draw(.8+4.5,0)node[below]{$\hat{r}$}; 
 
      \draw[->, line width = .5pt] (-1+4.5,0)--(3+4.5,0); 
  \draw(3+4.5,0)node[below]{$x_1$};
 \draw[->, line width = .5pt] (0+4.5,-.3)--(0+4.5,1.8); 
  \draw(0+4.5,1.8)node[right]{$x_2$};
 \draw(1.5+4.5,.6)node[right]{$F_v$};

\end{tikzpicture}
   \caption{The set $E$  above
   cannot be obtained by applying an orthogonal transformation around the origin to the set $F_v$ shown in the right,
   therefore rigidity ($\mathcal{R}$) fails.
   This happens because the set $\{ 0 < \alpha_v < \pi \}$ is disconnected
   by a point $\hat{r} \in (0, \infty)$ such that $\alpha_v (\hat{r}) = \pi$.}
	\label{case alpha = pi}
\end{figure}
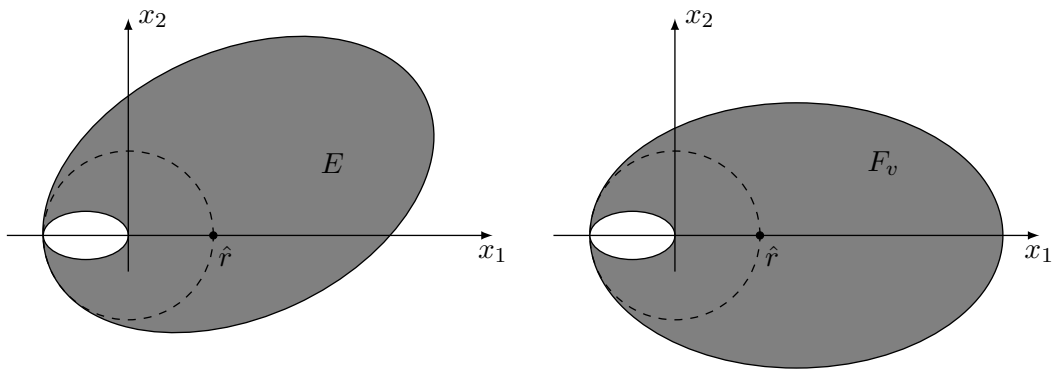

One possibility to avoid such a situation could be 
to request the set $\{ 0 < \alpha_v < \pi \}$ to be an interval.
However, this condition depends on the representative chosen for $\alpha_v$, 
while the perimeters of the sets $E$ and $F_v$ don't.
Indeed, in Figure~\ref{case iii counterexample true} one could modify $\alpha_v$
just at the point $\tilde{r}$, in such a way that 
$\{ 0 < \alpha_v < \pi \}$ becomes an interval.
Nevertheless, rigidity still fails, see Figure~\ref{case iii bis counterexample true}. 

\medskip

To formulate a condition which is independent 
on the chosen representative, we consider the approximate liminf 
and the approximate limsup of $\alpha_v$, which we denote by $\alpha_v^{\wedge}$ and $\alpha_v^{\vee}$,
respectively (see Section~\ref{preliminaries}).
These two functions 
are defined \textit{at every point} $r \in (0,\infty)$ and satisfy $\alpha_v^{\wedge} \leq \alpha_v^{\vee}$.
In addition, they do not depend on the representative chosen for $\alpha_v$, 
and $\alpha_v^{\wedge} = \alpha_v^{\vee} = \alpha_v$
$\mathcal{H}^1$-a.e. in $(0, \infty)$.
The condition that we will impose is then the following:
\begin{equation} \label{condition 1}
\textnormal{$\{ 0 < \alpha^{\wedge}_v \leq \alpha^{\vee}_v < \pi \}$ is a (possibly unbounded) interval.}
\end{equation}
One can check that, in the example given in Figure~\ref{case iii bis counterexample true}
this condition fails, since $\alpha^{\wedge}_v (\tilde{r}) = \alpha^{\vee}_v (\tilde{r}) = 0$.
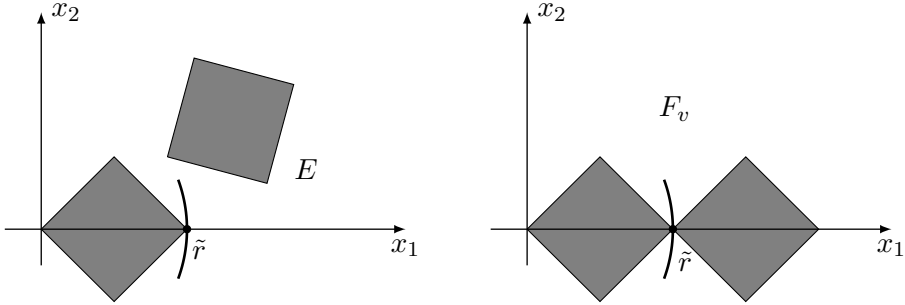
\begin{figure}[!htb]
    \centering
\begin{tikzpicture}[>=latex, scale=1.6]
 
 \draw [line width = 1pt,domain=-20:20] plot ({1.2*cos(\x)}, {1.2*sin(\x)});
 
 \draw[fill=gray] (0,0)--(.6,.6)--(1.2,0)--(.6,-.6)--(0,0);
  
     \draw(1.3,0)node[below]{$\tilde{r}$};
\draw[fill=black] (1.2,0) circle (.03cm);
 
 \begin{scope}[rotate around={30:(0,0)}]
 \draw[fill=gray] (0+1.2,0)--(.6+1.2,.6)--(1.2+1.2,0)--(.6+1.2,-.6)--(0+1.2,0);
\end{scope}
   
     \draw[->, line width = .5pt] (-.3,0)--(3,0); 
  \draw(3,0)node[below]{$x_1$};
 \draw[->, line width = .5pt] (0,-.3)--(0,1.8); 
  \draw(0,1.8)node[right]{$x_2$};
  
 \draw(2,.5)node[right]{$E$};

  \draw[fill=gray] (0+4,0)--(.6+4,.6)--(1.2+4,0)--(.6+4,-.6)--(0+4,0);
  \draw[fill=gray] (0+4+1.2,0)--(.6+4+1.2,.6)--(1.2+4+1.2,0)--(.6+4+1.2,-.6)--(0+4+1.2,0);
 
     \draw[->, line width = .5pt] (-.3+4,0)--(3+4,0); 
  \draw(3+4,0)node[below]{$x_1$};
 \draw[->, line width = .5pt] (0+4,-.3)--(0+4,1.8); 
  \draw(0+4,1.8)node[right]{$x_2$}; 
  
   \draw [line width = 1pt,domain=-20:20] plot ({4+1.2*cos(\x)}, {1.2*sin(\x)});
   
      \draw(4+1.3,-.1)node[below]{$\tilde{r}$};
\draw[fill=black] (4+1.2,0) circle (.03cm);

 \draw(5,1)node[right]{$F_v$};
 
\end{tikzpicture}
   \caption{Modifying the function $\alpha_v$ 
   given in Figure~\ref{case iii counterexample true} at the point $\tilde{r}$, 
we can make sure that $\{ 0 < \alpha_v < \pi\}$ is an open connected interval. However, rigitidy still fails.}
	\label{case iii bis counterexample true}
\end{figure}

\medskip

Let us show that, even imposing \eqref{condition 1}, rigidity can still be violated.
In the example given in Figure~\ref{case ii counterexample}, there is some radius 
$\overline{r} \in \{ 0 < \alpha^{\wedge}_v \leq \alpha^{\vee}_v < \pi \}$
such that the boundary of $F_v$ contains a non trivial subset of $\partial B (\overline{r})$.
In this way, it is possible to rotate a proper subset of $F_v$ around the origin, 
without affecting the perimeter.
Note that at each point of the set $\partial^* F_v \cap \partial B (\overline{r} )$
the exterior normal $\nu^{F_v}$ is parallel to the radial direction.
To rule out the situation described in Figure~\ref{case ii counterexample}, we will impose the following condition:
\begin{equation} \label{condition 2}
\mathcal{H}^{n-1} ( \{ x \in \partial^* F_v : \nu^{F_v}_{\parallel} (x) =  0 \text{ and } 
|x| \in \{ 0 < \alpha^{\wedge}_v \leq \alpha^{\vee}_v  < \pi \}) = 0.
\end{equation}
Note that, from Theorem~\ref{fv locally finite perimeter}
and identity \eqref{this is the def of alphav}, it follows that in general   
we only have $\alpha_v \in BV_{\text{loc}} (0,\infty)$.
However, it turns out that \eqref{condition 2} is equivalent to ask
that $\alpha_v$ is $W^{1,1}_{\text{loc}}$ in the interior of $\{ 0 < \alpha^{\wedge}_v \leq \alpha^{\vee}_v  < \pi \}$, see Proposition~\ref{prop:5.3dominik}.
\begin{figure}[!htb]
    \centering
  \begin{tikzpicture}[>=latex, scale=1.6]


    \draw[fill=gray]
(0,0) -- +(90:1) arc(90:0:1) -- cycle;
    
 \draw[fill=gray]
(0,0) -- +(60:1.3) arc(60:0:1.3) -- cycle;

\draw[color=gray, line width = 1.5pt] (0,0) -- (60:1);

    \draw(1,1)node[right]{$E$};
  \draw[->, line width = .5pt] (-1,0)--(1.7,0); 
  \draw(1.7,0)node[below]{$x_1$};
    \draw(1,0)node[below]{$\overline{r}$};
\draw[fill=black] (1,0) circle (.03cm);

 \draw[->, line width = .5pt] (0,-1)--(0,1.8); 
  \draw(0,1.8)node[right]{$x_2$};

 \draw[->, line width = .5pt] (3.5-1,0)--(3.5+1.7,0); 
  \draw(3.5+1.7,0)node[below]{$x_1$};
 \draw[->, line width = .5pt] (3.5,-1)--(3.5,1.8); 
  \draw(3.5,1.8)node[right]{$x_2$};
  

  \draw[fill=gray]
(3.5,0) -- +(45:1) arc(45:-45:1) -- cycle;
    
 \draw[fill=gray]
(3.5,0) -- +(30:1.3) arc(30:-30:1.3) -- cycle;    

\draw[color=gray, line width = 1.5pt] (3.5,0) -- (3.5+.85, .5);
\draw[color=gray, line width = 1.5pt] (3.5,0) -- (3.5+.85, -.5);

 \draw(3.5+1,1)node[right]{$F_v$};

 \draw[->, line width = .5pt] (3.5-1,0)--(3.5+1.7,0); 
  
 \draw[->, line width = .5pt] (3.5,-1)--(3.5,1.8); 
  
        \draw(3.5+1,0)node[below]{$\overline{r}$};
\draw[fill=black] (3.5+1,0) circle (.03cm);

\end{tikzpicture}
   \caption{An example in which rigidity fails. 
   In this case, the tangential part of $\partial^* F_v$ gives a non trivial contribution to $P(F_v)$.  
   This allows to slide a proper subset of $F_v$ around the origin, 
   without modifying the perimeter.}
	\label{case ii counterexample}
\end{figure}
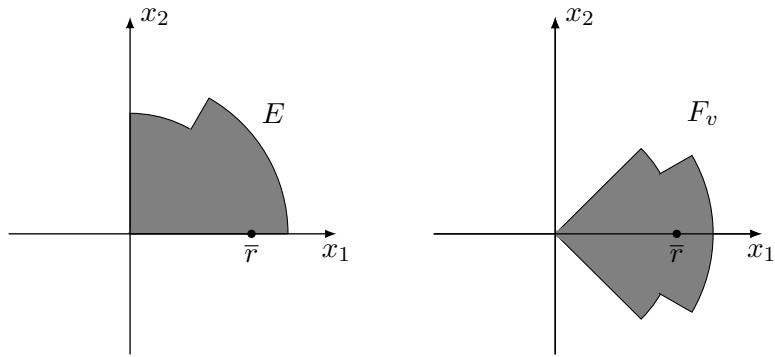
%
%
%
%
%
%
%
%
%
%
%
%
%
%

\medskip

Our main result shows that the two conditions above give a complete 
characterisation of rigidity for inequality~\eqref{per ineq}
(below, $\mathcal{\mathring I}$ stands for the interior of the set $\mathcal{I}$).
\begin{theorem} \label{rigidity theorem}
Let $v: (0,\infty) \to [0, \infty)$ be a measurable function 
satisfying \eqref{bound on v} such that $F_v$ 
is a set of finite perimeter and finite volume, 
and let $\alpha_v$ be defined by \eqref{this is the def of alphav}.
Then, the following two statements are equivalent:

\begin{itemize}

\item[(i)] $(\mathcal{R})$ holds true;

\vspace{.1cm}

\item[(ii)] $\{ 0 < \alpha^{\wedge}_v \leq \alpha^{\vee}_v < \pi \}$ is a (possibly unbounded) interval 
$\mathcal{I}$, 
and $\alpha_v \in W^{1, 1}_{\textnormal{loc}} (\mathcal{\mathring I})$. 

%
 
\end{itemize}
\end{theorem}
Let us point out that, although similar results in the context of Steiner and Ehrhard's inequalities
already appeared in \cite{CagnettiColomboDePhilippisMaggiSteiner, ccdpmGAUSS}, 
the proof of Theorem~\ref{rigidity theorem} cannot simply use previous ideas, 
especially in the implication (i) $\Longrightarrow$ (ii).
We cannot rely, as in \cite{CagnettiColomboDePhilippisMaggiSteiner}, 
on a general formula for the perimeter of sets $E$ satisfying equality in \eqref{per ineq}.
Instead, we exhibit explicit counterexamples to rigidity, whenever 
one of the assumptions in (ii) fails. This requires a careful 
analysis of the transformations that one can apply to the set $F_v$, 
without modifying its perimeter. 
This turns out to be non trivial, especially if one assumes $\alpha_v$
to have a non zero Cantor part (see Proposition~\ref{no Cantor}).

Also the proof of the implication (ii) $\Longrightarrow$ (i)
presents some difficulties.
In the context of Steiner symmetrisation,  
 this  has been proved in 
\cite[Theorem~1.3]{ChlebikCianchiFuscoAnnals05}
and \cite[Theorem~1.2]{barchiesicagnettifusco}, for
codimension $1$ and  for every  codimension, respectively.
 In the smooth case, 
a proof is given in \cite[Proposition~5]{MorganHoweHarman2011},
for the general class of symmetrisations in warped products. 
For the spherical setting  without any smoothness assumption,  this implication has already been stated in \cite[Theorem~6.2]{MorganPratelli2013},  but the proof is only sketched.
A rigorous proof of this fact turns out to be more delicate than one would expect, 
and relies on the following result. 
\begin{lemma} \label{lem:5.4dominik}
Let $v: (0,\infty) \to [0, \infty)$ be a measurable function 
satisfying \eqref{bound on v} such that $F_v$ 
is a set of finite perimeter and finite volume.
Let $E \subset \R^n$ be a spherically $v$-distributed set, 
and let $I\subset (0,+\infty)$ be a Borel set. Assume that
\begin{equation}\label{eq:5.6dominik}
\mathcal{H}^{n-1}\left(\left\{x\in \partial^* E\cap \Phi(I\times \mathbb{S}^{n-1}): \nu^{E}_{\|}(x)=0  \right\}  \right)=0.
\end{equation}
Then,
\begin{equation}\label{eq:5.7dominik}
\mathcal{H}^{n-1}\left(\left\{x\in \partial^* F_v \cap \Phi(I\times \mathbb{S}^{n-1}): \nu^{F_v}_{\|}(x)=0  \right\}  \right)=0.
\end{equation}
Viceversa, let \eqref{eq:5.7dominik} be satisfied, and suppose that
$P(E;\Phi(I\times \mathbb{S}^{n-1}))=P(F_v;\Phi(I\times \mathbb{S}^{n-1}))$.
Then, \eqref{eq:5.6dominik} holds true.
\end{lemma} 
A direct proof of Lemma~\ref{lem:5.4dominik} does not seem to be obvious, 
due to the fact that, as pointed out above, 
the measure $\lambda_E$ defined in \eqref{measure mu abs cont}
can have an absolutely continuous part when $n > 2$.
In the context of Steiner symmetrisation of higher codimension,
 a result playing the role of Lemma~\ref{lem:5.4dominik} (see \cite[Proposition~3.6]{barchiesicagnettifusco})
is proved using the fact that the statement  
holds true in codimension $1$, see \cite[Proposition~4.2]{ChlebikCianchiFuscoAnnals05}.
For this reason, we are led to consider 
the \textit{circular symmetrisation}, 
which is the codimension $1$ version of the spherical symmetrisation,
and  was originally introduced by P\'olya in the case $n = 3$ (see \cite{polya50}). 
Note that, when $n=2$, spherical and circular symmetrisation coincide.

\subsection{Circular Symmetrisation}
In order to introduce the circular symmetrisation, let us first observe 
how the spherical symmetrisation operates on a given set $E$, in the special case $n=2$.
In this situation, for each $r > 0$ one intersects $E$ with the circle $\partial B (r)$ of radius $r$
centred at the origin. 
Then, the symmetric set $F_v$ is obtained by centring, for each $r > 0$, 
an open circumference arc of length $\mathcal{H}^1 (E \cap \partial B (r))$ at the point $r e_1$.
When $n > 2$ one can proceed in a similar way, by first slicing the set $E$ with parallel planes, 
and then by symmetrising it (in each plane) with the procedure just described.
Note that, in this case, one needs to specify both the direction along which the open arcs are centred, 
and the direction along which the slicing through planes is performed.

Let us then choose an ordered pair of orthogonal directions in $\R^n$, 
which we will assume to be $(e_1, e_2)$
(we will be centring open circumference arcs along $e_1$, 
while we will be slicing the set $E$ with parallel planes that are orthogonal to $e_2$).
In the following, for each $x = (x_1, \ldots, x_n) \in \R^n$, we will write 
$x = (x_{12}, x')$, where $x_{12} = (x_1, x_2) \in \R^2$
and $x' = (x_3, \ldots, x_n) \in \R^{n-2}$.
When $x_{12} \neq 0$, we set $\hat{x}_{12} : = x_{12}/|x_{12}|$. 
For each given $z' \in \R^{n-2}$, we denote by $\Pi_{z'}$
the two-dimensional plane defined by
$$
\Pi_{z'} : = \{ x = (x_{12}, x' )\in \R^2 \times \R^{n-2} : x' = z' \}.
$$
Given a set $E \subset \R^n$ and $(r, z') \in (0, \infty) \times \R^{n-2}$, 
we define the  \textit{circular slice} $E_{(r, z')}$ of $E$ 
with respect to $\partial B ( (0,z') , r) \cap \Pi_{z'}$ as
$$
E_{(r, z')} := E \cap \partial B ( (0,z') , r) \cap \Pi_{z'}  = \{ x = (x_{12}, x' ) \in E : x' = z' \text{ and } |x_{12}| = r \}.
$$
%
Let $\ell  : (0, \infty) \times \R^{n-2} \to [0, \infty)$ be a measurable function.
We say that $E$ is \textit{circularly $\ell$-distributed} if 
\begin{equation*} 
\ell (r, x') = \mathcal{H}^{1} (E_{(r, x')}), 
\qquad \text{ for } \mathcal{H}^{n-1}\text{-a.e. } (r, x') \in (0, \infty) \times \R^{n-2}.
\end{equation*}
If $\ell$ is a circular distribution, then we have  
\begin{equation} \label{bound on ell}
\ell (r, x') \leq \mathcal{H}^{1} (\partial B ((0, x'), r) \cap \Pi_{x'}) = 2 \pi r 
\qquad \text{ for } \mathcal{H}^{n-1}\text{-a.e. } (r, x') \in (0, \infty) \times \R^{n-2}.
\end{equation}
%
%
%
Among all the sets in $\R^n$ that are circularly $\ell$-distributed, 
we denote by $F^{\ell}$ the one whose circular slices
are open circumference arcs centred at the positive $e_1$ axis.
That is, we set
\begin{equation*} 
F^{\ell} := \left\{ (x_{12}, x') \in 
\R^n \setminus \{ x_{12} = 0 \} 
 : \text{dist}_{\mathbb{S}^{1}} (\hat{x}_{12}, e_1) < \frac{1}{2 r} \ell (r, x') \right\}.
\end{equation*}
In the following, we introduce the diffeomorphism 
$\Phi_{12} : (0, \infty) \times \R^{n-2} \times \mathbb{S}^1
\to \R^n \setminus \{ \hat{x}_{12} = 0 \}$ given by
$$
\Phi_{12} (r, x', \omega) := (r \omega, x' ) \qquad \text{ for every }
(r, x', \omega) \in (0, \infty) \times \R^{n-2} \times \mathbb{S}^1.
$$
Moreover, for every $x \in \partial^* E$ 
we write $\nu^E (x) = ( \nu^E_{12} (x), \nu^E_{x'} (x))$,
where $\nu^E_{12} (x) = (\nu^E_1 (x), \nu^E_2 (x))$ and 
$\nu^E_{x'} (x) = (\nu^E_3 (x), \ldots, \nu^E_n (x))$.
Then, we further decompose $\nu^E_{12} (x)$ as 
$$
\nu^E_{12} (x) = \nu^E_{12 \perp} (x) + \nu^E_{12 \parallel} (x),
$$
where $\nu^E_{12 \perp} (x):= ( \nu^E (x) \cdot \hat{x}_{12} ) \hat{x}_{12}$
and $\nu^E_{12 \parallel} (x) := \nu^E_{12} (x) - \nu^E_{12 \perp} (x)$.
We can now state  a result that plays the role of 
Theorem~\ref{fv locally finite perimeter} for the circular symmetrisation. 
\begin{theorem} \label{theorem ineq cylindrical}
Let $\ell  : (0, \infty) \times \R^{n-2} \to [0, \infty)$ be a measurable function
satisfying \eqref{bound on ell}, and let $E \subset \R^n$ be a 
circularly $\ell$-distributed set of finite perimeter and finite volume.
Then, $\ell \in BV_{\textnormal{loc}} ((0, \infty) \times \R^{n-2})$. 
Moreover, $F^{\ell}$ is a set of finite perimeter and 
\begin{equation} \label{per ineq cyl}
P (F^{\ell}; \Phi_{12} (B \times \mathbb{S}^{1})) \leq P (E; \Phi_{12}  (B \times \mathbb{S}^{1})),
\end{equation}
for every Borel set $B \subset (0, \infty) \times \R^{n-2}$. 

Finally, if $P (E) = P(F^{\ell})$, 
then for $\mathcal{H}^{n-1}$-a.e. $(r, x') \in (0, \infty) \times \R^{n-2}$:
\begin{itemize}

\item[(a)] $E_{(r, x')}$ is $\mathcal{H}^1$-equivalent 
to a circular arc and $\partial^* (E_{(r, x')}) = (\partial^* E)_{(r, x')} $;

\vspace{.1cm}

\item[(b)] the three functions 
$$
x \longmapsto \nu^E (x) \cdot \hat{x}_{12}, \qquad 
x \longmapsto | \nu^E_{12 \parallel}| (x), \qquad 
x \longmapsto \nu^E_{x'} (x),
$$
are constant in $(\partial^* E )_{(r, x')}$.

\end{itemize}
\end{theorem} 
In the smooth setting and in the case $n=3$,
inequality \eqref{per ineq cyl} was proved by P\'olya.  
The following result is the counterpart of Lemma~\ref{lem:5.4dominik}
in the context of circular symmetrisation. 
\begin{lemma} \label{lem:5.4dominik cyl}
Let $\ell  : (0, \infty) \times \R^{n-2} \to [0, \infty)$
be a measurable function 
satisfying  \eqref{bound on ell} such that $F^{\ell}$ 
is a set of finite perimeter and finite volume.
Let $E \subset \R^n$ be a circularly $\ell$-distributed set, 
and let $I \subset (0, \infty) \times \R^{n-2}$ be a Borel set. 
Assume that
\begin{equation} \label{eq:5.6dominik cyl}
\mathcal{H}^{n-1}\left(\left\{x\in \partial^* E \cap \Phi(I\times \mathbb{S}^{1}): 
\nu^E_{12 \parallel} (x)=0  \right\}  \right)=0.
\end{equation}
Then,
\begin{equation}\label{eq:5.7dominik cyl}
\mathcal{H}^{n-1}\left(\left\{x\in \partial^* F^{\ell} \cap \Phi(I\times \mathbb{S}^{1}): 
\nu^{F^{\ell}}_{12 \parallel} (x)=0  \right\}  \right)=0.
\end{equation}
Viceversa, let \eqref{eq:5.7dominik cyl} be satisfied, and suppose that
$P(E;\Phi(I\times \mathbb{S}^{1}))=P(F^{\ell};\Phi(I\times \mathbb{S}^{1}))$.
Then, \eqref{eq:5.6dominik cyl} holds true.
\end{lemma} 
Once Lemma~\ref{lem:5.4dominik cyl} is established, 
we can show Lemma~\ref{lem:5.4dominik} through a slicing argument.
Finally, the proof of (ii) $\Longrightarrow$ (i) is concluded by showing 
that, if $E$ satisfies equality in \eqref{per ineq}, 
the function associating to every $r \in (0, \infty)$
the center of $E_r$ (see \eqref{def dE}) is $W^{1, 1}_{\textnormal{loc}}$
and, ultimately, constant (see Section~\ref{section ii implies i}).

The paper is divided as follows.
Section~\ref{preliminaries} contains basic results of Geometric Measure Theory that are extensively 
used in the following.
In Section~\ref{preliminary spherical} we give the setting 
of the problem and introduce useful tools to deal with the spherical framework. 
Section~\ref{section properties v and xi} is devoted to the study 
of the properties of the functions $v$ and $\xi_v$,   
while Theorem~\ref{fv locally finite perimeter} is proven in 
Section~\ref{section spherical proof perimeter inequality}.  
Important properties of the circular symmetrisation are discussed 
in Section~\ref{cylindrical section}, where we also give the proof of Lemma~\ref{lem:5.4dominik}.
 The implications (ii) $\Longrightarrow$ (i) and (i) $\Longrightarrow$ (ii)
of Theorem~\ref{rigidity theorem} are proven in Section~\ref{section ii implies i}
and Section~\ref{section i implies ii}, respectively.

\section{Basic notions of Geometric Measure Theory} \label{preliminaries}


In this section we introduce some tools from Geometric Measure Theory.
The interested reader can find more details in the monographs 
\cite{AFP, GMSbook1,maggiBOOK, Simon83}.
For $n \in \mathbb{N}$, we denote with $\mathbb{S}^{n-1}$
the unit sphere of $\mathbb{R}^n$, i.e. 
$$
\mathbb{S}^{n-1} = \{ x \in \R^n :  |x| = 1  \},
$$
where $| \cdot |$ stands for the Euclidean norm,
and we set $\R^n_0:= \R^n \setminus \{ 0 \}$.
For every $x \in \R^n_0$, we write $\hat{x} : = x/ |x|$ for the radial versor of $x$. 
We denote by $e_1, \ldots, e_n$ the canonical basis in $\R^n$, and 
for every $x, y \in \R^n$, $x \cdot y$ stands for the standard scalar product in $\R^n$ between $x$ and $y$.
For every $r > 0$ and $x \in \R^n$, we denote by $B (x, r)$ the open ball of $\R^n$
with radius $r$ centred at $x$.
In the special case $x = 0$, we set $B(r):= B(0, r)$.
In the following, we will often make use of the diffeomorphism 
$\Phi : (0, \infty) \times \mathbb{S}^{n-1} \to \mathbb{R}^n_0$ 
defined as
$$
\Phi (r, \omega) := r \omega \qquad \text{ for every } (r, \omega) \in (0, \infty) \times \mathbb{S}^{n-1}.
$$
For $x\in\R^n$ and $\nu\in \mathbb{S}^{n-1}$, we will denote by $H_{x,\nu}^+$ and $H_{x,\nu}^-$ the closed half-spaces
whose boundaries are orthogonal to $\nu$:
\begin{eqnarray}\label{Hxnu+}
  H_{x,\nu}^+&:=&\Big\{y\in\R^n:(y-x)\cdot\nu\ge 0\Big\}\,,
  \\\nonumber
  H_{x,\nu}^-&:=&\Big\{y\in\R^n:(y-x)\cdot\nu\le 0\Big\}\,.
\end{eqnarray}
If $1 \leq k \leq n$, we denote by $\mathcal{H}^k$ the $k$-dimensional Hausdorff measure in $\R^n$.
If $\{E_h\}_{h\in\N}$ is a sequence of Lebesgue measurable sets in $\R^n$
with finite volume, and $E \subset \R^n$ is also measurable with finite volume, 
we say that $\{E_h\}_{h\in\N}$ converges to $E$ as $h\to\infty$, and write $E_h\to E$, 
if $\H^n(E_h\Delta E)\to 0$ as $h\to\infty$. In the following, we will denote by $\chi_E$
the characteristic function of a measurable set $E \subset \R^n$.

\subsection{Density points} 
Let $E \subset \R^n$ be a Lebesgue measurable set and let $x\in\R^n$. 
The upper and lower $n$-dimensional densities of $E$ at $x$ are defined as
\begin{eqnarray*}
  \theta^*(E,x) :=\limsup_{r\to 0^+}\frac{\H^n(E\cap B(x,r))}{\om_n\,r^n}\,,
  \qquad
  \theta_*(E,x) :=\liminf_{r\to 0^+}\frac{\H^n(E\cap B(x,r))}{\om_n\,r^n}\,,
\end{eqnarray*}
respectively. 
It turns out that $x \mapsto  \theta^*(E,x)$ and $x \mapsto  \theta_*(E,x)$
are Borel functions that agree $\mathcal{H}^n$-a.e. on $\R^n$. 
Therefore, the $n$-dimensional density of $E$ at $x$
\[
\theta(E,x) := \lim_{r\to 0^+}\frac{\H^n(E\cap B(x,r))}{\om_n\,r^n}\,,
\]
is defined for $\H^n$-a.e. $x\in\R^n$, and $x \mapsto  \theta (E,x)$
is a Borel function on $\R^n$.
Given $t \in [0,1]$, we set
$$
E^{(t)} :=\{x\in\R^n:\theta(E,x)=t\}.
$$
By the Lebesgue differentiation theorem, the pair $\{E^{(0)},E^{(1)}\}$ 
is a partition of $\R^n$, up to a $\H^n$-negligible set. 
The set $\pae E :=\R^n\setminus(E^{(0)}\cup E^{(1)})$ is called the \textit{essential boundary} of $E$. 

\medskip

\subsection{Rectifiable sets}\label{section sofp} 
Let $1\le k\le n$, $k\in\N$. 
If $A, B \subset \R^n$ are Borel sets we say that 
$A \subset_{\H^k} B$ if $\H^k (B \setminus A) = 0$, and $A =_{\H^k} B$ if $\H^k (A \Delta B) = 0$,
where $\Delta$ denotes the symmetric difference of sets.
Let $M\subset\R^n$ be a Borel set.
We say that $M$ is {\it countably $\H^k$-rectifiable} if there exist 
Lipschitz functions $f_h:\R^k\to\R^n$ ($h\in\N$) such that $M\subset_{\H^k}\bigcup_{h\in\N}f_h(\R^k)$.
Moreover, we say that $M$ is {\it locally $\H^k$-rectifiable} if $\H^k(M\cap K)<\infty$ for every compact set $K\subset\R^n$, or, equivalently, if $\H^k\llcorner M$ is a Radon measure on $\R^n$. 

A Lebesgue measurable set $E\subset\R^n$ is said of {\it locally finite perimeter} in $\R^n$ if there exists a $\R^n$-valued Radon measure $\mu_E$, called the {\it Gauss--Green measure} of $E$, such that
\[
\int_E\nabla\vphi(x)\,dx=\int_{\R^n}\vphi(x)\,d\mu_E(x)\,,\qquad\forall \vphi\in C^1_c(\R^n)\,,
\]
where $C^1_c (\R^n)$ denotes the class of $C^1$ functions in $\R^n$
with compact support. 
The relative perimeter of $E$ in $A\subset\R^n$ is then defined by setting $P(E;A):=|\mu_E|(A)$
for any Borel set $A \subset \R^n$.
The perimeter of $E$ is then defined as $P(E):=P(E;\R^n)$.
If $P(E) < \infty$, we say that $E$ is a set of {\it finite perimeter} in $\R^n$.
The {\it reduced boundary} of $E$ is the set $\pa^*E$ of those $x\in\R^n$ such that
\[
\nu^E(x)=\lim_{r\to 0^+}\,\frac{\mu_E(B(x,r))}{|\mu_E|(B(x,r))}\qquad\mbox{exists and belongs to $\mathbb{S}^{n-1}$}\,.
\]
The Borel function $\nu^E:\pa^*E\to \mathbb{S}^{n-1}$ 
is called the {\it measure-theoretic outer unit normal} to $E$. 
If $E$ is a set of locally finite perimeter, it is possible to show
that $\pa^*E$ is a locally $\H^{n-1}$-rectifiable set in  $\R^n$ 
\cite[Corollary 16.1]{maggiBOOK}, with $\mu_E=\nu^E\,\H^{n-1}\mres\pa^*E$, and
\[
\int_E\nabla\vphi(x)\,dx=\int_{\pa^*E}\vphi(x)\,\nu^E(x)\,d\H^{n-1}(x)\,,\qquad\forall \vphi\in C^1_c(\R^n)\,.
\]
Thus, $P(E;A)=\H^{n-1}(A\cap\pa^*E)$ for every Borel set $A\subset\R^n$. 
If $E$ is a set of locally finite perimeter, it turns out that 
\begin{equation*}
  \label{inclusioni frontiere}
  \pa^*E \subset \subset E^{(1/2)} \subset \pae E\,.
\end{equation*}
Moreover, {\it Federer's theorem} holds true (see \cite[Theorem 3.61]{AFP} and \cite[Theorem 16.2]{maggiBOOK}):
\[
\H^{n-1}(\pae E\setminus\pa^*E)=0\,,
\]
thus implying that the essential boundary $\pae E$ of $E$ is locally $\H^{n-1}$-rectifiable 
in  $\R^n$.  

\subsection{General facts about measurable functions}
Let $f:\R^n\to\R$ be a Lebesgue measurable function.
We define the {\it approximate upper limit} $f^\vee(x)$ and the {\it approximate lower limit} 
$f^\wedge(x)$ of $f$ at $x\in\R^n$ as 
\begin{eqnarray}
  \label{def fvee}
  f^\vee(x)=\inf\Big\{t\in\R: x\in \{f>t\}^{(0)}\Big\}\,,
  \\
  \label{def fwedge}
  f^\wedge(x)=\sup\Big\{t\in\R: x\in \{f<t\}^{(0)}\Big\}\,.
\end{eqnarray}
We observe that $f^\vee$ and $f^\wedge$ are Borel functions that are defined at 
{\it every} point of $\R^n$, with values in $\R\cup\{\pm\infty\}$.
Moreover, if $f_1: \R^n \to \R$ and $f_2: \R^n \to \R$
are measurable functions satisfying $f_1=f_2$  $\H^n$-a.e.  on $\R^n$, 
then $f_1^\vee=f_2^\vee$ and $f_1^\wedge=f_2^\wedge$ {\it everywhere} on $\R^n$. 
We define the {\it approximate discontinuity} set $S_f$ of $f$ as
$$
S_f: =\{f^\wedge<f^\vee\}.
$$
Note that, by the above considerations, it follows that  $\H^n (S_f)=0$. 
Although $f^\wedge$ and $f^\vee$ may take infinite values on $S_f$, 
the difference $f^\vee(x)-f^\wedge(x)$ is well defined in $\R\cup\{\pm\infty\}$ for every $x\in S_f$.
Then, we can define the {\it approximate jump} $[f]$ of $f$ as the Borel function $[f]:\R^n\to[0,\infty]$ given by
  \begin{eqnarray*}
    [f](x):=\left\{\begin{array}{l l}
      f^\vee(x)-f^\wedge(x)\,,&\mbox{if $x\in S_f$}\,,
      \vspace{.2cm} \\
      0\,,&\mbox{if $x\in \R^n\setminus S_f$}\,.
    \end{array}
    \right .
  \end{eqnarray*}
Let $A\subset\R^n$ be a Lebesgue measurable set. 
We say that $t\in\R\cup\{\pm\infty\}$ is the approximate limit of $f$ at $x$ with respect to $A$, 
and write $t=\aplim (f,A,x)$, if
\begin{eqnarray*}
  &&\theta\Big(\{|f-t|>\e\}\cap A;x\Big)=0\,,\qquad\forall\e>0\,,\hspace{0.3cm}\qquad (t\in\R)\,,
  \\
  &&\theta\Big(\{f<M\}\cap A;x\Big)=0\,,\qquad\hspace{0.6cm}\forall M>0\,,\qquad (t=+\infty)\,,
  \\
  &&\theta\Big(\{f>-M\}\cap A;x\Big)=0\,,\qquad\hspace{0.3cm}\forall M>0\,,\qquad (t=-\infty)\,.
\end{eqnarray*}
We say that $x\in S_f$ is a \textit{jump point} of $f$ if there exists $\nu\in \mathbb{S}^{n-1}$ such that
\[
f^\vee(x)=\aplim(f,H_{x,\nu}^+,x)\,,\qquad f^\wedge(x)=\aplim(f,H_{x,\nu}^-,x)\,.
\]
If this is the case, we say that $\nu_f(x):= \nu$ is the approximate jump direction of $f$ at $x$.
If we denote by $J_f$ the set of approximate jump points of $f$, we have that $J_f\subset S_f$ 
and $\nu_f:J_f\to \mathbb{S}^{n-1}$ is a Borel function. 


\subsection{Functions of bounded variation} 
Let $f:\R^n\to\R$ be a Lebesgue measurable function, 
and let $\Omega\subset\R^n$ be open. 
We define the {\it total variation of $f$ in $\Omega$} as
\[
|Df|(\Omega)=\sup\Big\{\int_\Omega\,f(x)\,\diver\,T(x)\,dx:T\in C^1_c(\Omega;\R^n)\,,|T|\le 1\Big\}\,,
\]
where $C^1_c(\Omega;\R^n)$ is the set of $C^1$ functions from 
$\Omega$ to $\R^n$ with compact support.
We also denote by $C_c (\Omega ; \R^n)$ the class of all
continuous functions from $\Omega$ to $\R^n$.
Analogously, for any $k \in \mathbb{N}$, the class of $k$ times continuously differentiable functions from $\Omega$ to $\R^n$ is denoted by $C^k_c (\Omega ; \R^n)$. We say that $f$ belongs to the space of functions of bounded variations, $f \in BV(\Omega)$, if $|Df|(\Omega)<\infty$ and $f\in L^1(\Omega)$. Moreover, we say that $f\in BV_{\textnormal{loc}}(\Omega)$ if $f\in BV(\Omega')$ for every open set $\Omega'$ compactly contained in $\Omega$. 
Therefore, if $f\in BV_{\textnormal{loc}}(\R^n)$ the distributional derivative $Df$ of $f$ is an $\R^n$-valued Radon measure. 
In particular, $E$ is a set of locally finite perimeter if and only if $\chi_E\in BV_{\textnormal{loc}}(\R^{n})$. 
If $f\in BV_{loc}(\R^n)$, one can write the  Radon--Nykodim decomposition of $Df$ with respect to $\mathcal{H}^n$
as $Df=D^af+D^sf$, where $D^sf$ and $\mathcal{H}^n$ are mutually singular, and where $D^af\ll\mathcal{H}^n$. 
We denote the density of $D^af$ with respect to $\mathcal{H}^n$ by $\nabla f$, 
so that $\nabla\,f\in L^1(\Omega;\R^n)$ with $D^af=\nabla f\,d\mathcal{H}^n$. Moreover, for $\mathcal{H}^n$-a.e. $x\in\R^n$, $\nabla f(x)$ is the approximate differential of $f$ at $x$. 
If $f\in BV_{\textnormal{loc}}(\R^n)$, then $S_f$ is countably $\mathcal{H}^{n-1}$-rectifiable.
Moreover,  we have  $\mathcal{H}^{n-1}(S_f\setminus J_f)=0$, $[f]\in L^1_{loc}(\mathcal{H}^{n-1}\llcorner J_f)$, 
and the $\R^n$-valued Radon measure $D^jf$ defined as
\[
D^jf=[f]\,\nu_f\,d\mathcal{H}^{n-1}\llcorner J_f\,,
\]
is called the {\it jump part of $Df$}. 
If we set $D^cf=D^sf-D^jf$, we have that 
$D f = D^af+D^jf+D^cf$.
The $\R^n$-valued Radon measure $D^cf$ is called the {\it Cantorian part} of $Df$, 
and it is such that $|D^cf|(M)=0$ for every $M \subset \R^n$ 
which is $\sigma$-finite with respect to $\mathcal{H}^{n-1}$.

In the special case $n = 1$, if $(a, b) \subset \R$ is an open (possibly unbounded) interval, 
 \textit{every} $f \in BV ((a, b))$ can be written as 
\begin{equation} \label{1D decomposition BV}
 f = f^a + f^j + f^c,
\end{equation} 
where $f \in W^{1, 1} ((a, b))$, $f^j$ is a jump function (i.e. $D f = D^j f$)
and $f^c$ is a Cantor function (i.e. $Df = D^c f$), 
see \cite[Corollary~3.33]{AFP}. 
Moreover, if $f^j = 0$ (or, more in general, if $f$ is a \textit{good representative}, 
see \cite[Theorem~3.28]{AFP}), the total variation of $D f$ can be obtained as   
\begin{equation} \label{formula total variations in 1D}
| D f | (a, b) = 
\sup \left\{ \sum_{i=1}^N | f (x_{i+1})  -   f (x_i) | :
a < x_1 < x_2 < \ldots < x_N < b  \right\},
\end{equation}
where the supremum runs over all $N \in \mathbb{N}$, 
and over all the possible partitions of $(a, b)$ with $a<  x_1 < x_2 < \ldots < x_N < b$.
When $n=1$, we will often write $f'$ instead of $\nabla f$.
\section{Setting of the problem and preliminary results}
\label{preliminary spherical}
In this section we give the notation for the chapter, and we introduce some 
 results that will be extensively used later. 
For every $x, y \in \mathbb{S}^{n-1}$, the \textit{geodesic distance} between $x$ and $y$ is given by 
$$
\text{dist}_{\mathbb{S}^{n-1}} (x, y) := \arccos (x \cdot y).
$$
We recall that the geodesic distance satisfies the triangle inequality:
$$
\text{dist}_{\mathbb{S}^{n-1}} (x, y) \leq \text{dist}_{\mathbb{S}^{n-1}} (x, z) 
+ \text{dist}_{\mathbb{S}^{n-1}} (z, y)  \qquad \text{ for every } x, y, z \in \mathbb{S}^{n-1}.
$$
 Let $r > 0$, $p \in \mathbb{S}^{n-1}$ and $\beta \in [0, \pi]$ be fixed. 
The \textit{open geodesic ball} (or \textit{spherical cap}) of centre $r p$ and radius $\beta$ is the set 
$$
\mathbf{B}_{\beta} (r p) := \{ x \in \partial B (r) : \text{dist}_{\mathbb{S}^{n-1}} (\hat{x}, p) < \beta \}.
$$
Note in the extreme cases $\beta = 0$ and $\beta = \pi$
we have $\mathbf{B}_{0} (r p) = \emptyset$ and $\mathbf{B}_{\pi} (r p) = \partial B (r) \setminus \{- r p \}$, respectively.
Accordingly, the \textit{geodesic sphere} of centre $r p$ and radius $\beta$
is the boundary of $\mathbf{B}_{\beta} (r p)$, which is given by
$$
\mathbf{S}_{\beta} (r p) := \{ x \in \partial B (r) : \text{dist}_{\mathbb{S}^{n-1}} (\hat{x}, p) = \beta \}.
$$ 
The $(n-1)$-dimensional Hausdorff measure of a geodesic ball 
and the  $(n-2)$-dimensional Hausdorff measure
of a geodesic sphere are given by
\begin{align}
\mathcal{H}^{n-1} (\mathbf{B}_{\beta} (r p)) &= (n-1) \omega_{n-1} r^{n-1}
\int_0^{\beta} (\sin \tau)^{n-2} \, d \tau, \label{measure geodesic ball}\\
\mathcal{H}^{n-2} (\mathbf{S}_{\beta} (r p)) &= (n-1) \omega_{n-1} r^{n-2}
(\sin \beta)^{n-2}. \label{measure geodesic sphere}
\end{align}
Let $E \subset \R^n$ be a measurable set. For every $r > 0$, 
we define the \textit{spherical slice of radius $r$ of $E$} as the set
$$
E_r := E \cap \partial B (r) = \{ x \in \partial B (r) : x \in E \}.
$$
Let $v : (0, \infty) \to [0, \infty)$ be a Lebesgue measurable function, and let 
$E \subset \R^n$ be a measurable set in $\R^n$. We say that $E$ is spherically $v$-distributed if 
$$
v (r) = \mathcal{H}^{n-1} (E_r), \qquad \text{ for } \mathcal{H}^1\text{-a.e. } r \in (0, \infty). 
$$
If $E$ is spherically $v$-distributed, we can define the function
\begin{equation} \label{xi def}
\xi_v (r) := \frac{v (r)}{r^{n-1}} = \frac{\mathcal{H}^{n-1} (E_r)}{r^{n-1}}, \qquad \text{ for every } r \in (0, \infty). 
\end{equation}
Note that $\mathcal{H}^{n-1} (\mathbf{B}_{\pi}) = \mathcal{H}^{n-1} (\mathbb{S}^{n-1}) = n \omega_n$,
so that
\begin{equation} \label{xi}
0 \leq \xi_v (r) \leq n \omega_n, \qquad \text{ for every } r \in (0, \infty).
\end{equation}
From \eqref{measure geodesic ball}, it follows that the function 
$\mathcal{F}: [0, \pi] \to [0, n \omega_n]$ given by 
\begin{equation} \label{def mathcal F}
\mathcal{F} (\beta) := \mathcal{H}^{n-1} (\mathbf{B}_{\beta} (e_1))
\text{ is strictly increasing and smoothly invertible in } (0, n \omega_n).
\end{equation}
Therefore, if $v : (0, \infty) \to [0, \infty)$ is measurable, 
thanks to \eqref{xi}, there exists a unique function $\alpha_v: (0, \infty) \to [0, \pi]$ such that 
\begin{equation} \label{alpha}
\xi_v (r) = \mathcal{H}^{n-1} (\mathbf{B}_{\alpha_v (r)}(e_1) ) \qquad \text{ for every } r \in (0, \infty).
\end{equation}
Among all the spherically $v$-distributed sets of $\R^n$, we denote by $F_v$ the one whose spherical slices
are  open geodesic  balls centred at the positive $e_1$ axis., i.e.
\begin{equation} \label{def F}
F_v := \{ x \in \R^n_0 : \text{dist}_{\mathbb{S}^{n-1}} (\hat{x}, e_1) < \alpha_v (|x|) \}, 
\end{equation}
where $\alpha_v$ is defined by \eqref{xi def} and \eqref{alpha}.
The next result (see \cite[Lemma~2.35]{AFP}) will be used 
in the proof of Theorem~\ref{fv locally finite perimeter}. 
\begin{lemma}\label{lem:cagnetti 2.6}
Let $B\subset \R^n$ be a Borel set and let $\varphi_h,\varphi:B\rightarrow \R,$ $h\in \mathbb{N}$ be summable Borel functions such that $|\varphi_h|\leq |\varphi|$ for every $h$. Then
$$ \int_B\sup_{h}\varphi_h dx =\sup_{H}\left\{\sum_{h\in H}\int_{A_h}\varphi_h dx \right\}, $$
where the supremum ranges over all finite sets $H\subset \mathbb{N}$ and all finite partitions ${A_h},\, h\in H$ of $B$ in Borel sets.
\end{lemma}

\subsection{Normal and tangential components of functions and measures} \label{subsection tangential}
For every $\varphi \in  C_c (\R^n_0 ; \R^n)$, we decompose $\varphi$ as 
$\varphi = \varphi_{\perp} + \varphi_{\parallel}$, where 
$$
\varphi_{\perp} (x):= \left( \varphi (x) \cdot \hat{x} \right) \hat{x}
\qquad  \quad \text{ and } \quad \qquad \varphi_{\parallel} (x) := \varphi (x) - \varphi_{\perp} (x)
$$
are the radial and tangential components of $\varphi$, respectively.
%
If $\varphi \in  C^1_c (\R^n_0 ; \R^n)$, $\diver_{\parallel} \varphi (x)$ stands for the tangential divergence of $\varphi$ 
at $x$ along the sphere $\partial B(|x|)$:
\begin{equation} \label{defin div parallel}
\diver_{\parallel} \varphi (x):= \diver \varphi (x) - \big( \nabla \varphi (x) \hat{x} \big) \cdot  \hat{x}.
\end{equation}
The following lemma gives some  useful identities that will be needed later.
\begin{lemma}
Let $\varphi \in  C^1_c (\R^n_0 ; \R^n)$. Then, for every $x \in \R^n_0$ one has
\begin{align} 
\diver \varphi_{\perp} (x) &= \big( \nabla \varphi (x) \hat{x} \big) \cdot  \hat{x} 
+ \left( \varphi (x) \cdot \hat{x} \right) \frac{n -1}{|x|}, \label{n-1 term} \\ 
\diver \varphi_{\parallel} (x) &= \diver_{\parallel} \varphi_{\parallel} (x).  \label{identity 1}
\end{align}
\end{lemma}

\begin{remark} \label{rem divergence}
Let $\varphi \in  C^1_c (\R^n_0 ; \R^n)$. Recalling that $\varphi = \varphi_{\perp} + \varphi_{\parallel}$, combining \eqref{n-1 term} and \eqref{identity 1} it follows that 
$$
\diver \varphi (x) = \big( \nabla \varphi (x) \hat{x} \big) \cdot  \hat{x} 
+ \left( \varphi (x) \cdot \hat{x} \right) \frac{n -1}{|x|} + \diver_{\parallel} \varphi_{\parallel} (x) 
\qquad \forall \,  x \in \R^n_0.
$$
\end{remark}
\begin{proof}
First of all, note that 
\begin{equation} \label{nabla phi}
\nabla \left( \varphi (x) \cdot \hat{x} \right) = ( \nabla \varphi (x) )^T \hat{x}  + \frac{1}{| x | } \varphi_{\parallel} (x).
\end{equation}
Indeed, 
\begin{align*}
\nabla \left( \varphi (x) \cdot \hat{x} \right)
& = ( \nabla \varphi (x) )^T \hat{x} 
+ \frac{I - \hat{x} \otimes \hat{x}}{|x|} \varphi (x) 
= ( \nabla \varphi (x) )^T \hat{x} + \frac{1}{| x | } \varphi_{\parallel} (x),
\end{align*}
where $I$ represents the identity map in $\R^n$, and $\hat{x} \otimes \hat{x}$
is the usual tensor product of $\hat{x}$ with itself (so that $I - \hat{x} \otimes \hat{x}$
is the orthogonal projection on the tangent plane to $\mathbb{S}^{n-1}$ at $\hat{x}$).
Thanks to \eqref{nabla phi}, we have
\begin{align*}
\diver \varphi_{\perp} (x) 
&= \diver \left( ( \varphi (x) \cdot \hat{x} )  \hat{x} \right) 
= \nabla \left( \varphi (x) \cdot \hat{x} \right) \cdot \hat{x} +  \left( \varphi (x) \cdot \hat{x} \right) \diver \hat{x} \\
&= \left[ ( \nabla \varphi (x) )^T \hat{x} 
+ \frac{1}{|x|} \varphi_{\parallel} (x) \right] \cdot \hat{x} + \left( \varphi (x) \cdot \hat{x} \right) \frac{n -1}{|x|} \\
&= \big( \nabla \varphi (x) \hat{x} \big) \cdot  \hat{x} + \left( \varphi (x) \cdot \hat{x} \right) \frac{n -1}{|x|}, 
\end{align*}
which proves \eqref{n-1 term}.
Note now that, by definition \eqref{defin div parallel}, it follows that 
\begin{align} \label{intermediate}
\diver \varphi (x) = \diver_{\parallel} \varphi (x) + \big( \nabla \varphi (x) \hat{x} \big) \cdot  \hat{x}.
\end{align}
On the other hand, from \eqref{n-1 term} 
\begin{align*}
\diver \varphi (x) &= \diver \varphi_{\parallel} (x)  + \diver \varphi_{\perp} (x) \\
&= \diver \varphi_{\parallel} (x)  + \big( \nabla \varphi (x) \hat{x} \big) \cdot  \hat{x} + \left( \varphi (x) \cdot \hat{x} \right) \frac{n -1}{|x|}.
\end{align*}
Comparing last identity with \eqref{intermediate} we obtain that for every $\varphi \in C^1_c (\R^n_0 ; \R^n)$
$$
\diver_{\parallel} \varphi (x)  = \diver \varphi_{\parallel} (x) + \left( \varphi (x) \cdot \hat{x} \right) \frac{n -1}{|x|}.
$$
Applying the last identity to the function $\varphi_{\parallel}$ we obtain \eqref{identity 1}. 
\end{proof}

If $\mu$ is an $\R^n$-valued Radon measure on $\R^n_0$, 
we will write $\mu = \mu_{\perp} + \mu_{\parallel}$, where $\mu_{\perp}$ and $\mu_{\parallel}$
are the $\R^n$-valued Radon measures on $\R^n_0$ such that
\begin{align*}
\int_{\mathbb{R}^n_0} \varphi  \cdot  d \mu_{\perp} 
= \int_{\mathbb{R}^n_0} \varphi_{\perp}  \cdot d \mu, 
\qquad \text{ and } \qquad 
\int_{\mathbb{R}^n_0} \varphi  \cdot d \mu_{\parallel} 
= \int_{\mathbb{R}^n_0} \varphi_{\parallel}  \cdot d \mu,
\end{align*}
for every $\varphi \in C_c (\mathbb{R}^n_0; \R^n)$.
Note that $\mu_{\perp}$ and $\mu_{\parallel}$ are well defined by Riesz~Theorem (see, for instance, \cite[Theorem 1.54]{AFP}).
In the special case $\mu = D f$, with $f \in BV_{\text{loc}} (\mathbb{R}^n_0)$, we will shorten the notation writing
$D_{\parallel} f$ and $D_{\perp} f$ in place of $(D f)_{\parallel}$ and $(D f)_{\perp} $, respectively.
In particular, if $f = \chi_E$ and $E \subset \R^n$ is a set of finite perimeter, by De Giorgi structure theorem 
we have 
\begin{equation} \label{de giorgi structure and decomposition}
D_{\perp} \chi_E = \nu_{\perp}^E d \mathcal{H}^{n-1} \mres \partial^* E \qquad \text{ and } \qquad
D_{\parallel} \chi_E = \nu_{\parallel}^E d \mathcal{H}^{n-1} \mres \partial^* E.
\end{equation}

Next lemma gives some useful identities concerning the radial and tangential components
of the gradient of a $BV_{\text{loc}}$ function.
\begin{lemma} \label{lem:4.2domink}
Let $f \in BV_{\textnormal{loc}} (\R^n_0)$. Then, 
\begin{align}
\int_{\R^n_0} \varphi (x) \cdot d D_{\parallel} f 
&= - \int_{\R^n_0}  f(x) \, \diver_{\parallel} \varphi_{\parallel}  (x)   \, d x, \label{eq:4.3dominik} \\
\int_{\R^n_0} \varphi (x) \cdot d D_{\perp} f 
&= - \int_{\R^n_0}  f(x) \left( \nabla \varphi (x) \, \hat{x} \right) \cdot \hat{x} \, dx
- \int_{\R^n_0} f(x)  \frac{n-1}{|x|} \left( \varphi (x) \cdot \hat{x} \right) \, d x, \label{eq:4.4dominik}
\end{align}
for every $\varphi \in C^1_c (\R^n_0 ; \R^n)$.
\end{lemma}
\begin{proof}
Let $\varphi \in C^1_c (\R^n_0 ; \R^n)$.
By definition of $D_{\parallel} f$ and thanks to \eqref{identity 1} we have 
\begin{align*}
&\int_{\R^n_0} \varphi (x) \cdot d D_{\parallel} f 
= \int_{\R^n_0} \varphi_{\parallel} (x) \cdot d D f \\
&\hspace{.2cm}= - \int_{\R^n_0} \diver \varphi_{\parallel}  (x) f (x) \, d x 
=  - \int_{\R^n_0} \diver_{\parallel} \varphi_{\parallel}  (x) f (x) \, d x, 
\end{align*}
and this shows \eqref{eq:4.3dominik}.
Similarly, by definition of $D_{\perp} f$
\begin{align*}
&\int_{\R^n_0} \varphi (x) \cdot d D_{\perp} f 
= \int_{\R^n_0} \varphi_{\perp} (x) \cdot d D f = - \int_{\R^n_0} \diver \varphi_{\perp}  (x) f (x) \, d x.
\end{align*}
Thanks to \eqref{n-1 term}, identity \eqref{eq:4.4dominik} follows.
\end{proof}
An immediate consequence of identity \eqref{eq:4.3dominik} is the following.
\begin{corollary} \label{parallel total variation}
Let $f \in BV _{\textnormal{loc}} (\R^n_0)$ and let $\Omega \subset \subset \R^n_0$ be open and bounded.
Then,
\begin{equation*}
\left| D_{\|}f \right|(\Omega)= \sup \left\{ \int_{\R^n}f(x) \, \diver_{\parallel}\varphi_{\parallel}(x)dx:
\; \varphi  \in C^1_{c}(\Omega;\R^n),\, \|\varphi \|_{L^\infty (\Omega ; \R^n)} \leq 1\right\}.
\end{equation*}
\end{corollary}
We conclude this subsection with an important proposition, 
that is a special case of the Coarea Formula  (see \cite[Theorem~2.93]{AFP}).
\begin{proposition} \label{coarea}
Let E be a set of finite perimeter in $\R^n$ and let $g:\R^n \rightarrow [0,\infty]$ be a Borel function. 
Then, 
$$ 
\int_{\partial^* E} g(x) |\nu^{E}_{\parallel}(x)| d\mathcal{H}^{n-1}(x) 
= \int_{0}^{\infty}dr\int_{(\partial^* E)_r }g(x) \, d\mathcal{H}^{n-2}(x).
$$
\end{proposition}

\begin{proof}
The result follows by applying \cite[Remark~2.94]{AFP} with $N = n-1$, 
$M = n$, $k = 1$, and $f (x) = |x|$. 
\end{proof}
In the next subsection we show how 
the notion of set of finite perimeter can be given 
in a natural way also for subsets of the sphere $\mathbb{S}^{n-1}$
(and, more in general, of $\partial B (r)$, for any $r > 0$).

\subsection{Sets of finite perimeter on $\mathbb{S}^{n-1}$}
We now give a very brief introduction to sets of finite perimeter on $\mathbb{S}^{n-1}$,
by using the notion of integer multiplicity rectifiable currents, 
see \cite[Chapter~6]{Simon83} for more details (see also \cite{bogelainduzaarfusco2017}). 
 Let $k \in \mathbb{N}$ with $1 \leq k \leq n-1$.
We denote by $\Lambda_k (\R^n)$
and $\Lambda^k (\R^n)$ the linear spaces of $k$-vectors and  
$k$-covectors in $\R^n$, respectively,
while $\mathcal{D}^k (\R^n)$ stands for
the set of smooth $k$-forms with compact support in $\R^n$.

A \textit{$k$-dimensional current} in $\R^n$ 
is a continuous linear functional on $\mathcal{D}^k (\R^n)$.
The family of $k$-dimensional currents in $\R^n$
is denoted by $\mathcal{D}_k (\R^n)$.
We say that $T \in \mathcal{D}_k (\R^n)$
is an \textit{integer multiplicity rectifiable $k$-current} 
if it can be represented as 
\[
T (\omega) = \int_{M} \langle \omega (x) , \eta (x) \rangle \, \theta (x) \, d \mathcal{H}^{k} (x)
\quad \text{ for every } \omega \in \mathcal{D}^k (\R^n),
\]
where $M$ is an $\mathcal{H}^k$-measurable countably 
$k$-rectifiable subset of $\R^n$, 
$\theta$ is an $\mathcal{H}^k$-measurable positive integer-valued 
function, and $\eta: M \to \Lambda_k (\R^n)$ is an $\mathcal{H}^k$-measurable 
function such that for $\mathcal{H}^k$-a.e. $x \in M$
one has $\eta (x) = \tau_1 (x) \wedge \ldots \wedge \tau_k (x)$, 
with $\tau_1 (x), \ldots, \tau_k (x)$
an orthonormal basis for the approximate tangent space of $M$ at $x$,
and $\langle \cdot , \cdot \rangle$ denotes the usual pairing between  
$\Lambda^k (\R^n)$ and $\Lambda_k (\R^n)$.
In the special case when 
\[
T (\omega) = \int_{M} \langle \omega (x) , \eta (x) \rangle \, d \mathcal{H}^{k} (x)
\quad \text{ for every } \omega \in \mathcal{D}^k (\R^n),
\]
we write $T = [\![ M ]\!]$.
The boundary $\partial T$ of $T$ is then defined as  
the element of $\mathcal{D}_{k-1} (\R^n)$ such that
\[
\partial T ( \omega ) 
= T (d \omega) \quad \text{ for every } \omega \in \mathcal{D}^k (\R^n),
\]
while the mass $\mathbf{M} (T)$ of $T$ is given by
\[
\mathbf{M} (T) := \sup \left\{ T (\omega) :  \omega \in \mathcal{D}^k (\R^n), \, | \omega | \leq 1  \right\}. 
\]
More in general, for any open set $U \subset \R^n$, we set
\[
\mathbf{M}_U (T) := \sup \left\{ T (\omega) :  \omega \in \mathcal{D}^k (\R^n), \,
 | \omega | \leq 1, \, \text{supp} \, \omega \in U  \right\}. 
\]

Let $A \subset \mathbb{S}^{n-1}$ be 
an $\mathcal{H}^{n-1}$-measurable set.
We will say that $A$ is a set of finite perimeter 
 on $\mathbb{S}^{n-1}$ 
if there exists $Q \in  \mathcal{D}_{n-2}   (\R^n)$ 
with $\text{supp} \,  Q   \subset \mathbb{S}^{n-1}$ and 
$$
 Q   = \partial   [\![ A ]\!],
$$
with the property that $\mathbf{M}_U ( Q )  < \infty$
for every $U \subset \subset \R^n$. 
 By the Riesz representation theorem it follows that there exists a 
Radon measure $\mu_Q$ and a 
 $\mu_Q$-measurable  function $\nu: \mathbb{S}^{n-1} \to T_x \mathbb{S}^{n-1}$
such that $| \nu (x)| = 1$ for $\mu_T$-a.e. $x$ and 
$$
\int_A \diver_{\parallel} \varphi (x) \, d \mathcal{H}^{n-1} (x)
= \int_{\mathbb{S}^{n-1}} \varphi (x) \cdot \nu (x) \, d  \mu_Q  (x),
$$
for every smooth vector field with $\varphi = \varphi_{\parallel}$.
If $A \subset \mathbb{S}^{n-1}$ is a set of finite perimeter on the sphere, 
the reduced boundary $\partial^* A$ is the set of points $x \in \mathbb{S}^{n-1}$
such that the limit
$$
\nu^A (x) := \lim_{\rho \to 0} \frac{1}{ \mu_Q (B (x, \rho))} \int_{B (x, \rho)} \nu (y) \, d  \mu_Q (y)
$$
exists, $\nu^A (x) \in T_x \mathbb{S}^{n-1}$, and  $\nu^A (x) = 1$.
The De Giorgi structure theorem holds true also for sets of finite perimeter on the sphere.
In particular, $\partial^* A$ is countably $(n-2)$-rectifiable,
$ \mu_Q  = \mathcal{H}^{n-2} \mres \partial^* A$, and 
\begin{equation} \label{div theorem hypersurfaces}
\int_A \diver_{\parallel} \varphi (x) \, d \mathcal{H}^{n-1} (x)
= \int_{\partial^* A} \varphi (x) \cdot \nu^A (x) \, d \mathcal{H}^{n-2} (x),
\end{equation}
for every smooth vector field with $\varphi = \varphi_{\parallel}$.
 The isoperimetric inequality on the sphere states that, 
if $\beta \in (0, \pi)$ and $A \subset \mathbb{S}^{n-1}$ is a set of finite perimeter on $\mathbb{S}^{n-1}$
with $\mathcal{H}^{n-1} (A) = \mathcal{H}^{n-1} (\mathbf{B}_{\beta} (e_1))$, 
then (see \cite{schmidt})
\begin{equation} \label{isop ineq}
\mathcal{H}^{n-2} (\partial^* \mathbf{B}_{\beta} (e_1)) \leq 
\mathcal{H}^{n-2} (\partial^* A).
\end{equation}
 The next theorem is a version of a result by Vol'pert (see \cite{Volpert}).
\begin{theorem} \label{thm:volpert}
Let $v: (0, \infty) \to [0, \infty)$ be a measurable function
satisfying \eqref{bound on v}, and let $E \subset \R^n$ be a spherically
$v$-distributed set of finite perimeter and finite volume.
Then, there exists a Borel set $G_E \subset \{ \alpha_v > 0 \}$
with $\mathcal{H}^1 (\{ \alpha_v > 0 \} \setminus G_E) = 0$,
such that 

\begin{itemize}

\item[(i)] for every $r \in G_E$:

\vspace{.1cm}

\begin{itemize}

\item[(ia)] $E_r$ is a set of finite perimeter in $\partial B(r)$;

\vspace{.1cm}

\item [(ib)] $\mathcal{H}^{n-2}(\partial^{*}(E_r)\Delta (\partial^* E)_r)=0$;

\end{itemize}

\vspace{.1cm}

\item[(ii)] for every $r \in G_E \cap \{ 0 < \alpha_v  < \pi \}$:
 
 \vspace{.1cm}

 \begin{itemize}
\item[(iia)] $| \nu^{E}_{\|}( r \omega)|  > 0$,

\vspace{.1cm}

\item[(iib)]$\nu^{E}_{\|}( r \omega )= \nu^{E_r}( r \omega )|\nu^{E}_{\|}( r \omega )|$,   

\vspace{.1cm}

\end{itemize}
for $\mathcal{H}^{n-2}$-a.e. $\omega \in \mathbb{S}^{n-1}$ such that $r \omega \in \partial^{*}(E_r) \cap (\partial^* E)_r$.
\vspace{.1cm}

\end{itemize}
%
\end{theorem}
\begin{proof}
The result follows applying \cite[Theorem~28.5]{Simon83} with $f (x) = |x|$, 
and recalling the definition of slicing of a current (see \cite[Definition~28.4]{Simon83}).
\end{proof}
We now make some important remarks about Theorem~\ref{thm:volpert}.
\begin{remark}
Thanks to property (ib), we have 
$$
\partial^{*}(E_r) =_{\mathcal{H}^{n-2}} (\partial^* E)_r \qquad \text{ for every } r \in G_{E}.
$$
Therefore, whenever $r \in G_{E}$ we will often write
$\partial^* E_r$ instead of $\partial^{*}(E_r)$ or $(\partial^* E)_r$, without any risk of ambiguity.
Moreover, for every $r \in G_{E}$ we will also use the notation
$$
p_E (r) : = \mathcal{H}^{n-2} (\partial^* E_r). 
$$
\end{remark}


\begin{remark} \label{remark gmt}
In dimension $n = 2$, the theorem above 
implies that, if $r \in G_{E} \cap \{ 0 < \theta < \pi \}$, then $\partial^* (E_r) = (\partial^* E)_r$ and 
\begin{equation} \label{lucky n=2}
| \nu^{E}_{\parallel}( r \omega )| > 0 \quad \textnormal{ \textbf{for every} 
$\omega \in \mathbb{S}^{1}$ such that } r \omega \in (\partial^* E)_r.
\end{equation}
Let now $\lambda_E$ be the measure defined in \eqref{measure mu abs cont}:
$$
\lambda_E (B) =\int_{\partial^* E \cap \Phi(B\times \mathbb{S}^{1})\cap\{\nu^{E}_{\|}=0 \}}
\hat{x} \cdot \nu^E(x)  \, d\mathcal{H}^{1}(x) \quad \text{ for every Borel set } B \subset (0, \infty).
$$
If $B \subset G_E$, then by \eqref{lucky n=2}
$$
| \lambda_E (B) | 
\leq \mathcal{H}^1 (\partial^* E \cap \Phi(G_E \times \mathbb{S}^{1})\cap\{\nu^{E}_{\|}=0 \}) = 0,
$$
so that $\lambda_E (B) = 0$.
As a consequence, $\lambda_E$ is singular with respect 
to the Lebesgue measure in $(0,\infty)$.
If $n > 2$ this conclusion is in general false (unless one chooses $E = F_v$, 
see Remark~\ref{rem 2 after volpert} below), 
and it may happen that $\lambda_E$ 
has a non trivial absolutely continuous part.
\end{remark}

\begin{remark} \label{rem 2 after volpert}
If $n \geq 2$, but we consider the special case $E = F_v$, Theorem~\ref{thm:volpert}
gives much more information than the one we can obtain for a generic set of finite perimeter. 
Indeed, let $R \in O(n)$ be any orthogonal transformation that keeps fixed the $e_1$ axis.
By definition of $F_v$, and thanks to \cite[Exercise 15.10]{maggiBOOK}, 
we have that if $x \in \partial^* F_v$, then $R x \in \partial^* F_v$ and 
$$
\nu^{F_v}_{\parallel}( R x ) = R \, \nu^{F_v}_{\parallel}( x ) \qquad \text{ and }
 \qquad  \nu^{F_v}_{\perp}( R x ) = R \, \nu^{F_v}_{\perp}( x ).
$$
Therefore, applying Theorem~\ref{thm:volpert} to $F_v$ we infer that 

\begin{itemize}

\item[(j)] for every $r \in G_{F_v}$:

\vspace{.1cm}

\begin{itemize}

\item[(ja)] $(F_v)_r$ is a spherical cap;

\vspace{.1cm}

\item [(jb)] $\partial^{*} (F_v)_r = (\partial^* F_v)_r$;

\end{itemize}

\vspace{.1cm}

\item[(jj)] for every $r \in G_{F_v} \cap \{ 0 < \alpha_v < \pi \}$:
 
 \vspace{.1cm}

 \begin{itemize}
\item[(jja)] $| \nu^{F_v}_{\|}( r \omega)|  > 0$,

\vspace{.1cm}

\item[(jjb)]$\nu^{F_v}_{\|}( r \omega )= \nu^{(F_v)_r}( r \omega )|\nu^{F_v}_{\|}( r \omega )|$,   

\vspace{.1cm}

\end{itemize}
\textnormal{ \textbf{for every} } 
$\omega \in \mathbb{S}^{n-1}$ such that $r \omega \in (\partial^* F_v)_r \cap \partial^{*} ( F_v)_r$.
\vspace{.1cm}

\end{itemize}
Therefore,
\begin{equation} \label{measure of B0 is 0}
\mathcal{H}^1(B_0)=0, 
\end{equation}
where 
$$
B_0:=\left\{r\in (0,+\infty): \exists  \, \omega \in \mathbb{S}^{n-1} 
\textnormal{ such that } r \omega \in \partial^{*}F_v \textnormal{ and } \nu^{F_v}_{\|}(r \omega)=0 \right\}.
$$
Moreover, repeating the argument used in Remark~\ref{remark gmt} one obtains that
$$
\mathcal{H}^{n-1} (\partial^* F_v \cap \Phi(G_{F_v} \times \mathbb{S}^{n-1})\cap\{\nu^{F_v}_{\|}=0 \}) = 0.
$$
Thus, the measure $\lambda_{F_v}$ defined in \eqref{measure mu abs cont}
is purely singular with respect to the Lebesgue measure in $(0, \infty)$.
\end{remark}

\section{Properties of $v$ and $\xi_v$}\label{section properties v and xi}

In this section we discuss several properties of the functions $v$ and $\xi_v$.
These are the natural counterpart in the spherical setting 
of analogous results proven in \cite{ChlebikCianchiFuscoAnnals05} and  \cite{barchiesicagnettifusco}.
We start by showing that, if $E \subset \R^n$ is a set of finite perimeter and volume, 
then $v \in BV(0, \infty)$.
\begin{lemma}  \label{lem:4.12Dominik}
Let $v$ be as in Theorem~\ref{fv locally finite perimeter}, 
and let $E \subset \R^n$ be a spherically $v$-distributed set of finite perimeter and finite volume.
Then, $v \in BV (0, \infty)$. 
Moreover,  $\xi_v \in BV_{\textnormal{loc}} (0, \infty)$ and 
\begin{equation} \label{formula for D xi}
\int_0^{\infty}  \psi (r) r^{n-1} d D \xi_v (r)
= \int_{\R^n_0} \psi (|x|) \, \hat{x} \cdot d D_{\perp} \chi_E (x), 
\end{equation}
for every bounded Borel function $\psi: (0, \infty) \to \R$.
As a consequence, 
\begin{equation} \label{bound for D xi}
| r^{n-1}  D \xi_v | (B) \leq | D_{\perp} \chi_E | (\Phi (B \times \mathbb{S}^{n-1})), 
\end{equation}
for every Borel set $B \subset (0, \infty)$. 
In particular, $r^{n-1} D \xi_v$ is a bounded Radon measure on $(0, \infty)$.
\end{lemma}

\begin{proof}
We divide the proof into steps.

\vspace{.2cm}

\noindent
\textbf{Step 1:} We show that $v \in BV (0, \infty)$.
First of all, note that $v \in L^1 (0, \infty)$, since
$$
\| v \|_{L^1 (0, \infty)} = \int_{0}^{\infty} v (r) \, dr
= \int_{0}^{\infty} \, dr \int_{\partial B (r)} \chi_E (x) \, d \mathcal{H}^{n-1} (x)
= \mathcal{H}^n ( E) < \infty. 
$$
Let now $\psi \in C^1_c (0, \infty)$ with $|\psi | \leq 1$.
Applying formula \eqref{n-1 term} to the radial function $\psi (|x|) \hat{x}$,
we obtain that for every $x \in \R^n_0$ 
\begin{align}
&\diver \left( \psi (|x|) \hat{x}  \right)
= \left[ \nabla \left(  \psi (|x|) \hat{x} \right) \hat{x} \right] \cdot \hat{x} 
+  \left[ \psi (|x|) \hat{x} \cdot \hat{x} \right] \frac{n-1}{|x|} \nonumber \\
&= \left[  \left( \psi' (|x|) \hat{x} \otimes \hat{x} + \psi (|x|) \frac{I - \hat{x} \otimes \hat{x}}{|x|} \right) \hat{x} \right] \cdot \hat{x}  
+ \psi (|x|) \frac{n-1}{|x|} \nonumber \\
&= \psi' (|x|) + \psi (|x|) \frac{n-1}{|x|}. \label{formula div}
\end{align}
Thus, 
\begin{align*}
&\int_{\R^n} \left[   \psi' (|x|) + \psi (|x|) \frac{n-1}{|x|} \right] \chi_E (x) \, dx
= \int_{\R^n} \diver \left( \psi (|x|) \, \hat{x}  \right) \chi_E (x) \, dx \\
&= - \int_{\R^n} \psi (|x|) \, \hat{x} \cdot d D \chi_E (x) 
= - \int_{\R^n} \psi (|x|) \, \hat{x} \cdot d D_{\perp} \chi_E (x),
\end{align*}
so that 
\begin{align}
&\int_{\R^n} \psi' (|x|) \chi_E (x) \, dx \label{integral one} \\
&= - \int_{\R^n} \psi (|x|) \frac{n-1}{|x|} \chi_E (x) \, dx
-\int_{\R^n} \psi (|x|) \, \hat{x} \cdot d D_{\perp} \chi_E (x). \nonumber
\end{align}
By Coarea formula, the integral in the left hand side can be written as
\begin{align}
&\int_{\R^n} \psi' (|x|) \chi_E (x) \, dx  
= \int_{0}^{\infty} \, dr \, \psi' (r) \int_{\partial B (r)} \chi_E (x) \, d \mathcal{H}^{n-1} (x) 
= \int_{0}^{\infty}  \psi' (r) v (r) \, dr.  \label{integral two}
\end{align}
Combining \eqref{integral one} and \eqref{integral two} we find that
%
\begin{align} 
&\int_{0}^{\infty}  \psi (r) \, d D v (r)   \nonumber \\
&= \int_{\R^n} \psi (|x|) \frac{n-1}{|x|} \chi_E (x) \, dx
+ \int_{\R^n}\psi (|x|) \, \hat{x}  \cdot  d D_{\perp} \chi_E (x). 
 \label{distr derivative} \\
&\leq \int_{B (1)} \psi (|x|) \frac{n-1}{|x|} \chi_E (x) \, dx
+ \int_{\R^n \setminus B (1)} \psi (|x|) \frac{n-1}{|x|} \chi_E (x) \, dx + P (E) \nonumber \\
& \leq n (n-1) \omega_n \int_{0}^1 \rho^{n-2} \, d \rho
+ (n-1) | E | + P (E)  \nonumber \\
&= n  \omega_n + (n-1) | E | + P (E) < \infty. \nonumber
\end{align}
Taking the supremum over $\psi$ we obtain that 
$$
| D v| (0, \infty) < \infty,
$$
so that $v \in BV (0, \infty)$.

\vspace{.2cm}

\noindent
\textbf{Step 2:} We conclude the proof.
Since the function $r \mapsto 1/(r^{n-1})$ is smooth and locally bounded in $(0, \infty)$, 
we also have that $\xi_v (r) \in BV_{\text{loc}} (0, \infty)$.
Moreover, recalling that $v (r) = r^{n-1} \xi_v (r)$, by the chain rule in $BV$ (see \cite[Example~3.97]{AFP})
\begin{equation} \label{Dv}
D v  = (n-1) r^{n-2} \xi_v (r) \, dr + r^{n-1} D \xi_v
= (n-1) \frac{v (r)}{r} dr + r^{n-1} D \xi_v.
\end{equation} 
Let now $\psi \in C^1_c (0, \infty)$.
From the previous identity it follows that 
\begin{align}
&\int_0^{\infty}  \psi (r) \, d D v (r)
= \int_0^{\infty}  \psi (r)  \frac{n-1}{r} \, v (r)  \, dr + \int_0^{\infty}  \psi (r) r^{n-1} d D \xi_v (r) \nonumber \\
&\hspace{.2cm}= \int_0^{\infty}  \psi (r) \frac{n-1}{r} \mathcal{H}^{n-1} (\partial B(r) \cap E) \, dr 
+ \int_0^{\infty}  \psi (r) r^{n-1} d D \xi_v (r) \nonumber \\
&= \int_{\R^n} \psi (|x|) \frac{n-1}{|x|} \chi_E (x) \, dx
+ \int_0^{\infty}  \psi (r) r^{n-1} d D \xi_v (r). \nonumber 
\end{align}
Combining the previous identity and \eqref{distr derivative},
\begin{align*}
\int_0^{\infty}  \psi (r) r^{n-1} d D \xi_v (r)
= \int_{\R^n} \psi (|x|) \, \hat{x} \cdot d D_{\perp} \chi_E, \quad \text{ for every } \psi \in C^1_c (0\, \infty).
\end{align*}
By approximation, the identity above is true also when $\psi$ is a bounded Borel function, 
and this gives \eqref{formula for D xi}. 

If $B \subset (0, \infty)$ is open, thanks to \eqref{formula for D xi} we have that for every $\psi \in C_c (B)$
with $| \psi | \leq 1$ 
\begin{align*}
\int_{B}  \psi (r) r^{n-1} d D \xi_v (r)
= \int_{\Phi (B \times \mathbb{S}^{n-1})} \psi (|x|) \, \hat{x} \cdot d D_{\perp} \chi_E 
\leq |D_{\perp} \chi_E | (\Phi (B \times \mathbb{S}^{n-1})).
\end{align*}
Taking the supremum over all such $\psi$ gives
$$
| r^{n-1}  D \xi_v| (B) \leq |D_{\perp} \chi_E | (\Phi (B \times \mathbb{S}^{n-1})) \quad \text{ for every open set } 
B \subset (0, \infty).
$$
By approximation, the inequality above holds true for every Borel set, and this shows inequality \eqref{bound for D xi}.
\end{proof}
The next lemma gives an important property of the measure $r^{n-1}D\xi_v$.
\begin{lemma}\label{lem:4.14dominik}
Let $v$ be as in Theorem~\ref{fv locally finite perimeter}, 
and let $E \subset \R^n$ be a spherically $v$-distributed set of finite perimeter and finite volume.
Then
\begin{align}
( r^{n-1}D\xi_v )(B) &=\int_{\partial^* E \cap \Phi(B\times \mathbb{S}^{n-1})\cap\{\nu^{E}_{\|}=0 \}}
\hat{x} \cdot \nu^E(x)  \, d\mathcal{H}^{n-1}(x) \label{D xi} \\
&\hspace{.4cm}+\int_{B}dr \int_{(\partial^* E)_r \cap \{\nu^{E}_{\|} \neq 0 \} } 
\frac{\hat{x} \cdot \nu^E (x)}{|\nu_{\|}^E(x)|}d\mathcal{H}^{n-2}(x). \nonumber
\end{align}
for every Borel set $B \subset (0,+\infty)$. 

Moreover, $r^{n-1}D\xi_v \mres G_{F_v} = r^{n-1} \xi_v' dr$
and for $\mathcal{H}^1$-a.e. $r \in G_{F_v} \cap \{ 0 < \alpha_v < \pi \}$  
$$
r^{n-1}\xi'_v(r) = \mathcal{H}^{n-2} (\mathbf{S}_{\alpha_v (r)} (r e_1))  \frac{\hat{x} \cdot \nu^{F_v}(x)}{|\nu_{\|}^{F_v}(x)|},
\qquad \qquad \text{ for every } x \in \mathbf{S}_{\alpha_v (r)} (r e_1). 
$$
\end{lemma}

\begin{proof}
Let $B\subset (0,+\infty)$ be a Borel set.
Then, choosing $\psi=\chi_{B}$ in \eqref{formula for D xi}, and 
recalling \eqref{de giorgi structure and decomposition},
\begin{align*}
&( r^{n-1}D\xi_v )(B)
= \int_{0}^{+\infty} \chi_{B} (r) r^{n-1}d D \xi_v(r) \\
&= \int_{\Phi (B \times \mathbb{S}^{n-1})} \hat{x} \cdot d D_{\perp} \chi_E (x) 
= \int_{\partial^* E\cap \Phi(B\times \mathbb{S}^{n-1})} \hat{x} \cdot \nu^E(x) \, d\mathcal{H}^{n-1}(x)\\ 
&= \int_{\partial^* E\cap \Phi(B\times \mathbb{S}^{n-1})\cap\{\nu^{E}_{\|}=0 \}}
\hat{x} \cdot \nu^E(x) \, d\mathcal{H}^{n-1}(x)
+ \int_{\partial^* E\cap \Phi(B\times \mathbb{S}^{n-1})\cap\{\nu^{E}_{\|}\neq 0 \}}
\hat{x} \cdot \nu^E(x) \, d\mathcal{H}^{n-1}(x)\\
&=\int_{\partial^* E\cap \Phi(B\times \mathbb{S}^{n-1})\cap\{\nu^{E}_{\|}=0 \}}
\hat{x} \cdot \nu^E(x) \, d\mathcal{H}^{n-1}(x)
+\int_{B}dr \int_{(\partial^* E)_r \cap \{\nu^{E}_{\|} \neq 0 \} } 
\frac{\hat{x} \cdot \nu^E(x)}{|\nu_{\|}^E(x)|}d\mathcal{H}^{n-2}(x),
\end{align*}
where in the last equality we have used the Coarea formula. 

Let us now prove the second part of the statement.
If one chooses $E = F_v$, thanks to Remark~\ref{rem 2 after volpert} we have 
\begin{align*}
r^{n-1}D \xi_v \mres G_{F_v} 
&= \left( \int_{(\partial^* F_v)_r \cap \{\nu^{F_v}_{\|} \neq 0 \} } 
\frac{\hat{x} \cdot \nu^{F_v}(x)}{|\nu_{\|}^{F_v}(x)|}d\mathcal{H}^{n-2}(x) \right) \, dr \mres G_{F_v} \\
&=  \mathcal{H}^{n-2} (\mathbf{S}_{\alpha_v (r)} (r e_1))  \frac{\hat{x} \cdot \nu^{F_v}(x)}{|\nu_{\|}^{F_v}(x)|}.
\end{align*}
In particular,
$$
r^{n-1}D \xi_v \mres G_{F_v}  =r^{n-1} \xi'_v (r) \, dr\mres G_{F_v}.
$$
Moreover, since $\xi'_v (r) = 0$ $\mathcal{H}^1$-a.e. in $\{ \alpha = 0 \} \cup \{ \alpha = \pi \}$, 
we obtain that for $\mathcal{H}^1$-a.e. $r \in (0, \infty)$
$$
r^{n-1}\xi'(r) = \mathcal{H}^{n-2} (\mathbf{S}_{\alpha_v (r)} (r e_1))  \frac{\hat{x} \cdot \nu^{F_v}(x)}{|\nu_{\|}^{F_v}(x)|},
\qquad \qquad \text{ for every } x \in \mathbf{S}_{\alpha_v (r)} (r e_1). 
$$
\end{proof}
We now prove an auxiliary inequality that will be useful later.

\begin{proposition} \label{lem:4.17dominik}
Let $v$ be as in Theorem~\ref{fv locally finite perimeter}, 
and suppose that there exists a spherically $v$-distributed set $E \subset \R^n$ of finite perimeter and finite volume.
Then, $F_v$ is a set of finite perimeter in $\R^n$. 
Moreover, for every Borel set $B\subset (0,+\infty)$ 
\begin{equation} \label{smhts}
P(F_v;\Phi(B\times \mathbb{S}^{n-1})) \leq \left| r^{n-1} D\xi_v \right|(B) 
+ \left|D_{\|}\chi_{F_v}  \right|(\Phi(B\times \mathbb{S}^{n-1})).
\end{equation}
\end{proposition}
\begin{proof}
The proof is based on the arguments of \cite[Lemma 3.5]{ChlebikCianchiFuscoAnnals05} and \cite[Lemma 3.3]{barchiesicagnettifusco}.
Thanks to Lemma~\ref{lem:4.12Dominik}, $v \in BV(0, \infty)$.
Let $\{ v_j \}_{j \in \mathbb{N}} \subset C^1_c (0, \infty)$ be a sequence of non-negative functions 
such that $v_j \to v$ $\mathcal{H}^1$-a.e. in $(0,\infty)$ 
and $|D v_j| \overset{*}{\rightharpoonup} |D v|$.
For every $j \in \mathbb{N}$, we denote by $F_{v_j} \subset \mathbb{R}^n$
the set defined by \eqref{def F}, with $v_j$ in place of $v$.
Let now $\Omega \subset (0,\infty)$ be open, 
and let $ \varphi \in C^{1}_{c} (\Phi(\Omega\times \mathbb{S}^{n-1});\R^n)$ with 
$\| \varphi \|_{L^{\infty} (\Phi(\Omega\times \mathbb{S}^{n-1});\R^n)} \leq 1$. 
Thanks to Remark~\ref{rem divergence}, we have
\begin{align}
&\int_{\Phi(\Omega\times \mathbb{S}^{n-1})} \chi_{F_{v_j}}(x) \, \diver \varphi (x) dx 
= \int_{\Phi(\Omega\times \mathbb{S}^{n-1})} \chi_{F_{v_j}} (x) 
\, \diver_{\parallel} \varphi_{\parallel}(x) dx \label{cap2: divestimate} \\
&+\int_{\Phi(\Omega\times \mathbb{S}^{n-1})} \chi_{F_{v_j}}(x) \,  
\left( \nabla \varphi (x) \, \hat{x} \right) \cdot \hat{x} \, dx
+ \int_{\Phi(\Omega\times \mathbb{S}^{n-1})} \chi_{F_{v_j}}(x) \,  
\frac{n-1}{|x|} \left( \varphi (x) \cdot \hat{x} \right) \, dx. \nonumber
\end{align} 
In the following, it will be convenient to introduce the function $V_j : (0,\infty) \to \R$
given by
$$
V_j (r):= \int_{\mathbf{B}_{\alpha_{v_j} (r)} (r e_1)} 
\varphi (x) \cdot \hat{x} \, d\mathcal{H}^{n-1}(x)
=  r^{n-1} \int_{\mathbf{B}_{\alpha_{v_j} (r)} (e_1)} 
\varphi (r \omega) \cdot \omega \, d\mathcal{H}^{n-1}(\omega),
$$
where $\alpha_{v_j}: (0, r) \to [0, \pi]$ is defined by \eqref{alpha}, with $v_j$ in place of $v$.
We divide the proof into several steps.

\vspace{.2cm}

\noindent
\textbf{Step 1:} We show that $V_j$ is Lipschitz continuous with compact support.
Indeed, 
$$
\text{supp} \, V_j \subset \Lambda(\supp \, \varphi) 
:= \left\{r \in (0,+\infty): (\supp \, \varphi)\cap \partial B (r) \neq \emptyset   \right\}.
$$
Moreover, for every $r_1, r_2 \in (0,\infty)$,
\begin{align*}
&| V_j (r_1) - V_j (r_2) |
\leq \int_{\mathbf{B}_{\alpha_{v_j} (r_1)} (e_1)} 
| r^{n-1}_1  \varphi (r_1 \omega) \cdot \omega 
- r^{n-1}_2 \varphi (r_2 \omega) \cdot \omega | \, d\mathcal{H}^{n-1}(\omega) \\
&+ r^{n-1}_2 \left| \int_{\mathbf{B}_{\alpha_{v_j} (r_1)} (e_1)} 
 \varphi (r_2 \omega) \cdot \omega \, d\mathcal{H}^{n-1}(\omega)
- \int_{\mathbf{B}_{\alpha_{v_j} (r_2)} (e_1)} 
 \varphi (r_2 \omega) \cdot \omega \, d\mathcal{H}^{n-1}(\omega) \right| \\
&\leq c | r_1-  r_2| + r^{n-1}_2  
\int_{ \mathbf{B}_{ \alpha_{v_j} (\widetilde{r}_1)} (e_1)
\setminus 
\mathbf{B}_{ \alpha_{v_j} (\widetilde{r}_2)} (e_1) }
| \varphi (r_2 \omega) \cdot \omega | \, d\mathcal{H}^{n-1}(\omega) \\
&\leq c | r_1-  r_2| + r^{n-1}_2 | \xi_{v_j} (r_1) - \xi_{v_j} (r_2) | \leq c | r_1-  r_2|,
\end{align*}
where we used the fact that $\xi_{v_j}$
is compactly supported in $(0, \infty)$ (since $v_j$ is), 
and $\widetilde{r}_1$ and $\widetilde{r}_2$ are such that 
$\alpha_{v_j} (\widetilde{r}_1) = \max\{ \alpha_{v_j} (r_1), \alpha_{v_j} (r_2)\}$
and $\alpha_{v_j} (\widetilde{r}_2) := \min\{ \alpha_{v_j} (r_1), \alpha_{v_j} (r_2)\}$.

\vspace{.2cm}

\noindent
\textbf{Step 2:} We show that $\alpha_{v_j}$ is $\mathcal{H}^1$-a.e. differentiable and 
that 
\begin{align}
V'_j (r) &= (n-1)r^{n-2} \int_{\mathbf{B}_{\alpha_{v_j} (r)} (e_1)} 
\varphi (r \omega) \cdot \omega \, d\mathcal{H}^{n-1}(\omega) \nonumber \\
&\hspace{.5cm}+ r^{n-1} \bigg( \alpha_{v_j}' (r) \int_{\mathbf{S}_{\alpha_{v_j} (r)}(e_1)} 
 \varphi (r \omega) \cdot \omega \, d\mathcal{H}^{n-2}(\omega) \bigg) \label{formula V'}\\
 &\hspace{.5cm}+ r^{n-1}\int_{\mathbf{B}_{\alpha_{v_j} (r)}(e_1)}  
 \left( \nabla \varphi (r \omega) \, \omega \right) 
 \cdot \omega \, d\mathcal{H}^{n-1}(\omega), \nonumber 
\end{align}
for $\mathcal{H}^1$-a.e. $r > 0$.
Let us set $A_j:= \{ 0 <  \alpha_{v_j} < \pi \}$.
Since $v_j \in C^1_c (0, \infty)$, from \eqref{def mathcal F}
it follows that $\alpha_{v_j} \in C^1 (A_j)$. Moreover, for every $r \in A_j$
\begin{align*}
&V'_j (r)
= \frac{d}{d r} \bigg( r^{n-1} \int_{0}^{\alpha_{v_j} (r)}d\beta 
 \int_{\mathbf{S}_\beta (e_1)} \varphi (r \omega) \cdot \omega \, d\mathcal{H}^{n-2}(\omega) \bigg) \\
&= (n-1)r^{n-2} \int_{\mathbf{B}_{\alpha_{v_j} (r)} (e_1)} 
\varphi (r \omega) \cdot \omega \, d\mathcal{H}^{n-1}(\omega) 
+ r^{n-1} \bigg( \alpha_{v_j}' (r) \int_{\mathbf{S}_{\alpha_{v_j} (r)}(e_1)} 
 \varphi (r \omega) \cdot \omega \, d\mathcal{H}^{n-2}(\omega) \bigg) \\
 &\hspace{.5cm}+ r^{n-1} \int_{0}^{\alpha_{v_j} (r)}d\beta 
 \int_{\mathbf{S}_\beta (e_1)} \left( \nabla \varphi (r \omega) \, \omega \right) 
 \cdot \omega \, d\mathcal{H}^{n-2}(\omega) \\
 &= (n-1)r^{n-2} \int_{\mathbf{B}_{\alpha_{v_j} (r)} (e_1)} 
\varphi (r \omega) \cdot \omega \, d\mathcal{H}^{n-1}(\omega) 
+ r^{n-1} \bigg( \alpha_{v_j}' (r) \int_{\mathbf{S}_{\alpha_{v_j} (r)} (e_1)} 
 \varphi (r \omega) \cdot \omega \, d\mathcal{H}^{n-2}(\omega) \bigg) \\
 &\hspace{.5cm}+ r^{n-1}\int_{\mathbf{B}_{\alpha_{v_j} (r)} (e_1)}  
 \left( \nabla \varphi (r \omega) \, \omega \right) 
 \cdot \omega \, d\mathcal{H}^{n-1}(\omega).
\end{align*} 
This shows \eqref{formula V'} whenever $r \in A_j$.
Note now that 
\begin{align*}
V_j (r) & = 0 \quad & \text{ for every } r \in \text{Int} ( \{ \alpha_{v_j} = 0 \} ), \\
V_j (r) & =
r^{n-1} \int_{\mathbb{S}^{n-1}} 
\varphi (r \omega) \cdot \omega \, d\mathcal{H}^{n-1}(\omega)
\quad & \text{ for every } r \in \text{Int} ( \{ \alpha_{v_j} = \pi \} ),
\end{align*}
where $\text{Int} ( \cdot )$ stands for the interior of a set.
Since 
$\alpha_{v_j}' (r)= 0$ for every $ r \in~\text{Int} ( \{ \alpha_{v_j} = 0 \} ) \cup \text{Int} ( \{ \alpha_{v_j} = \pi \} )$,
using the identities above one can see that \eqref{formula V'} holds true 
for $\mathcal{H}^1$-a.e. $r > 0$.

\vspace{.2cm}

\noindent
\textbf{Step 3:} We show that
\begin{align*} 
&\int_{\Phi(\Omega\times \mathbb{S}^{n-1})} \chi_{F_{v_j}}(x) \,  
\left( \nabla \varphi (x) \, \hat{x} \right) \cdot \hat{x} \, dx
+ \int_{\Phi(\Omega\times \mathbb{S}^{n-1})} \chi_{F_{v_j}}(x) \,  
\frac{n-1}{|x|} \left( \varphi (x) \cdot \hat{x} \right) \, dx \nonumber \\
&= - \int_{\Omega} \, dr \, r^{n-1} \bigg( \alpha_{v_j}' (r) \int_{\mathbf{S}_{\alpha_{v_j} (r)} (e_1)} 
 \varphi (r \omega) \cdot \omega \, d\mathcal{H}^{n-2}(\omega) \bigg). 
\end{align*}
Integrating \eqref{formula V'}, thanks to the classical divergence theorem applied in $\Omega$, 
and recalling that $V_j$ has compact support, we obtain 
\begin{align*}
0&= (n-1) \int_{\Omega} \, dr \, r^{n-2} \int_{\mathbf{B}_{\alpha_{v_j} (r)} (e_1)} 
\varphi (r \omega) \cdot \omega \, d\mathcal{H}^{n-1}(\omega) \\
&\hspace{.5cm}+ \int_{\Omega} \, dr \, r^{n-1} \bigg( \alpha_{v_j}' (r) 
\int_{\mathbf{S}_{\alpha_{v_j} (r)}(e_1)} 
 \varphi (r \omega) \cdot \omega \, d\mathcal{H}^{n-2}(\omega) \bigg) \\
 &\hspace{.5cm}+ \int_{\Omega} \, dr \, r^{n-1}\int_{\mathbf{B}_{\alpha_{v_j} (r)} (e_1) }  
 \left( \nabla \varphi (r \omega) \, \omega \right) 
 \cdot \omega \, d\mathcal{H}^{n-1}(\omega)\\
 &= \int_{\Phi(\Omega\times \mathbb{S}^{n-1})} \chi_{F_{v_j}}(x) \,  
\frac{n-1}{|x|} \left( \varphi (x) \cdot \hat{x} \right) \, dx \\
&\hspace{.5cm}+ \int_{\Omega} \, dr \, r^{n-1} 
\bigg( \alpha_{v_j}' (r) \int_{\mathbf{S}_{\alpha_{v_j} (r)} (e_1)} 
 \varphi (r \omega) \cdot \omega \, d\mathcal{H}^{n-2}(\omega) \bigg) \\
 &\hspace{.5cm}+ \int_{\Phi(\Omega\times \mathbb{S}^{n-1})} \chi_{F_{v_j}}(x) \,  
\left( \nabla \varphi (x) \, \hat{x} \right) \cdot \hat{x} \, dx,
\end{align*}
which gives the claim.

\vspace{.2cm}

\textbf{Step 4:} we prove that
\begin{align} 
\int_{\Phi(\Omega\times \mathbb{S}^{n-1})}\chi_{F_{v_j}}(x)  \, \diver \varphi(x) dx 
\leq \left|r^{n-1}D \xi_{v_j} \right|(\Lambda(\supp \, \varphi)) \label{qwedfg33} 
+ \int_{\Omega} \, \mathcal{H}^{n-2} (\mathbf{S}_{\alpha_{v_j} (r)})dr,
\end{align}
where $\Lambda(\supp \, \varphi) \subset (0, \infty)$ is the compact set defined in Step 1.
Thanks to \eqref{cap2: divestimate} and Step~3
\begin{align}
&\int_{\Phi(\Omega\times \mathbb{S}^{n-1})}\chi_{F_{v_j}}(x)  \, \diver \varphi (x) \, dx 
=  \int_{\Phi(\Omega\times \mathbb{S}^{n-1})}\chi_{F_{v_j}}(x)  
\, \diver_{\parallel} \varphi_{\parallel} (x) \, dx \nonumber \\
&\hspace{.4cm}- \int_{\Omega} \, dr \, r^{n-1} 
\bigg( \alpha_{v_j}' (r) \int_{\mathbf{S}_{\alpha_{v_j} (r)} (e_1)} 
 \varphi (r \omega) \cdot \omega \, d\mathcal{H}^{n-2}(\omega) \bigg).\label{eq:2.4}
\end{align}
We now estimate the right hand side of the expression above.
Thanks to \eqref{alpha} and arguing as in Step 2 we have that 
$$
\xi_{v_j}' (r) = \alpha_{v_j}' (r) \mathcal{H}^{n-2}( \mathbf{S}_{\alpha_{v_j (r)} } (e_1) )
\qquad \text{ for } \mathcal{H}^1\text{-a.e. } r \in (0, \infty).
$$
Therefore, 
\begin{align}
&- \int_{\Omega} \, dr \, r^{n-1} 
\bigg( \alpha_{v_j}' (r) \int_{\mathbf{S}_{\alpha_{v_j} (r)} (e_1)} 
 \varphi (r \omega) \cdot \omega \, d\mathcal{H}^{n-2}(\omega) \bigg) \nonumber \\
 &\hspace{.5cm}\leq \int_{\Lambda(\supp \,\varphi)}r^{n-1}\left| \alpha_{v_j}' (r) \right| 
\mathcal{H}^{n-2}( \mathbf{S}_{\alpha_{v_j} (r)} (e_1)) dr  \label{first piece} \\
&\hspace{.5cm}= \int_{\Lambda(\supp \,\varphi)}r^{n-1}\left| \xi'_{v_j} (r) \right| dr= 
\left|r^{n-1}D \xi_{v_j} \right|(\Lambda(\supp \, \varphi)). \nonumber
\end{align}
Let us now focus on the second integral in the right hand side of \eqref{eq:2.4}.
Applying the divergence theorem \eqref{div theorem hypersurfaces}
with $A = \mathbf{B}_{\alpha_{v_j} (r)} (r e_1)$, 
and denoting by $\nu_* (x)$ the exterior unit normal to $\mathbf{S}_{\alpha_{v_j} (r)} (r e_1)$, we have 
\begin{align}
&\int_{\Phi(\Omega\times \mathbb{S}^{n-1})} \chi_{F_{v_j}}(x)  
\, \diver_{\parallel} \varphi_{\parallel} (x) \, dx 
= \int_{\Omega}dr \int_{\mathbf{B}_{\alpha_{v_j} (r)}(r e_1)} 
\diver_{\parallel} \varphi_{\parallel} (x)  \, d \mathcal{H}^{n-1} (x) \nonumber \\ 
&= \int_{\Omega}dr \int_{\mathbf{S}_{\alpha_{v_j} (r)} (r e_1)} 
\varphi_{\parallel} (x) \cdot \nu_{*} (x) d\mathcal{H}^{n-2}(x)  
\leq \int_{\Omega}dr \, \mathcal{H}^{n-2} (\mathbf{S}_{\alpha_{v_j} (r)} (r e_1)). \label{second piece}
\end{align}
Combining \eqref{eq:2.4}, \eqref{first piece}, and \eqref{second piece}, we obtain \eqref{qwedfg33}.

\vspace{.2cm}

\textbf{Step 5:} We show that $F_v$ is a set of finite perimeter.
Note that $\chi_{F_{v_j}} \to \chi_{F_{v}}$ $\mathcal{H}^n$-a.e. in $\R^n$, and 
$\alpha_{v_j} \to \alpha$ $\mathcal{H}^1$-a.e. in $(0, \infty)$. 
Note also that, from our choice of the sequence $\{ v_j \}_{j \in \mathbb{N}}$
and thanks to \eqref{Dv}, it follows that 
$$
| r^{n-1} D \xi_{v_j}| 
\overset{*}{\rightharpoonup} | r^{n-1} D \xi_{v}| \quad \quad \text{ as } j \to \infty.
$$
Therefore, taking the limsup as $j \to \infty$ in \eqref{qwedfg33}, and using the fact that 
$\Lambda(\supp \, \varphi)$ is compact,
\begin{align*}
&\int_{\Phi(\Omega\times \mathbb{S}^{n-1})}\chi_{F_v}(x) \, \diver \, \varphi(x) dx
= \limsup_{j \to \infty}  \int_{\Phi(\Omega\times \mathbb{S}^{n-1})}\chi_{F_{v_j}}(x)  \, \diver \, \varphi(x) dx \\
&\leq \limsup_{j \to \infty} \left|r^{n-1}D \xi_{v_j} \right|(\Lambda(\supp \, \varphi))
+ \limsup_{j \to \infty} \int_{\Omega} \mathcal{H}^{n-2} (\mathbf{S}_{\alpha_{v_j} (r)} (r e_1))\, dr\\
&\leq \left|r^{n-1}D \xi_v \right|(\Lambda(\supp \, \varphi))
+ \int_{\Omega}  \mathcal{H}^{n-2} (\mathbf{S}_{\alpha_v (r)} (r e_1))\, dr 
\leq \left|r^{n-1}D \xi_v \right|(\Omega)
+ \int_{\Omega} \mathcal{H}^{n-2} (\partial^* E_r)\, dr \\
&\leq \left|r^{n-1}D \xi_v \right|(\Omega) + P(E;\Phi(\Omega\times \mathbb{S}^{n-1})), 
\end{align*}
where we also used the isoperimetric inequality in the sphere (see \eqref{isop ineq}) and the Coarea formula.
Taking the supremum of the above inequality 
over all functions $ \varphi \in C^{1}_{c} (\Phi(\Omega\times \mathbb{S}^{n-1});\R^n)$ with 
$\| \varphi \|_{L^{\infty} (\Phi(\Omega\times \mathbb{S}^{n-1});\R^n)} \leq 1$, we obtain 
$$
P(F_v;\Phi(\Omega\times \mathbb{S}^{n-1}))\leq \left|r^{n-1}D  \xi_v \right|(\Omega) 
+ P(E;\Phi(\Omega\times \mathbb{S}^{n-1})).
$$
Thanks to \eqref{bound for D xi} we have
$$
P(F_v;\Phi(\Omega\times \mathbb{S}^{n-1})) 
\leq 2 P(E; P(F_v;\Phi(\Omega\times \mathbb{S}^{n-1}))) < \infty,
$$ 
since $E$ is a set of finite perimeter by assummption.
Since $\Omega$ was arbitrary, this shows that $F_v$ is a set of locally finite perimeter.

\vspace{.2cm}

\noindent
\textbf{Step 6:} We conclude. Let $\Omega \subset (0,\infty)$ be open,
and let $\varphi \in C^{1}_{c} (\Phi(\Omega\times \mathbb{S}^{n-1});\R^n)$ with 
$\| \varphi \|_{L^{\infty} (\Phi(\Omega\times \mathbb{S}^{n-1});\R^n)} \leq 1$.
Combining \eqref{cap2: divestimate}, Step 3, and \eqref{first piece}, 
we have that for every $j \in \mathbb{N}$
$$
\int_{\Phi(\Omega\times \mathbb{S}^{n-1})}\chi_{F_{v_j}}(x)  \, \diver \, \varphi (x) dx 
\leq \left|r^{n-1}D \xi_{v_j} \right|(\Lambda(\supp \, \varphi)) \label{qwedfg} 
+ \int_{\Phi(\Omega\times \mathbb{S}^{n-1})}\chi_{F_{v_j}}(x)  
\, \diver_{\parallel} \varphi_{\parallel} (x) \, dx.  
$$
Taking the limsup as $j \to \infty$ and thanks to Corollary~\ref{parallel total variation}, 
\begin{align*}
\int_{\Phi(\Omega\times \mathbb{S}^{n-1})}\chi_{F_v}(x) \, \diver \, \varphi (x) dx 
&\leq \left|r^{n-1}D \xi_v \right|(\Lambda(\supp \, \varphi)) 
+ \int_{\Phi(\Omega\times \mathbb{S}^{n-1})}\chi_{F_v}(x) 
\, \diver_{\parallel} \varphi_{\parallel} (x) \, dx \\
&\leq \left|r^{n-1}D \xi_v \right|(\Lambda(\supp \, \varphi)) 
+ | D_{\parallel} \chi_{F_v} | (\Phi(\Omega\times \mathbb{S}^{n-1})),
\end{align*}
where we also used the fact that $\Lambda(\supp \, \varphi)$ is compact.

Taking the supremum over all $\varphi \in C^{1}_{c} (\Phi(\Omega\times \mathbb{S}^{n-1});\R^n)$ with 
$\| \varphi \|_{L^{\infty} (\Phi(\Omega\times \mathbb{S}^{n-1});\R^n)} \leq 1$, 
\begin{equation} \label{for open sets}
P(F_v;\Phi(\Omega\times \mathbb{S}^{n-1}))\leq \left|r^{n-1}D \xi_v \right|(\Omega) 
+  | D_{\parallel} \chi_{F_v} | (\Phi(\Omega\times \mathbb{S}^{n-1})),
\end{equation}
which shows \eqref{smhts} when $B$ is an open set. 
Let now $B \subset (0, \infty)$ be a Borel set.
From \eqref{for open sets} it follows that 
$$
P(F_v;\Phi(B \times \mathbb{S}^{n-1}))\leq \left|r^{n-1}D  \xi_v \right|(\Omega) + P(E;\Phi(\Omega\times \mathbb{S}^{n-1})),
$$
for any open set $\Omega \subset (0, \infty)$ with $B \subset \Omega$.
Taking the infimum of the above inequality over all
 open sets $\Omega \subset (0, \infty)$ with $B \subset \Omega$,
 we obtain inequality \eqref{smhts} when $B$ is a Borel set.
 \end{proof}

\section{Proof of Theorem~\ref{fv locally finite perimeter}}\label{section spherical proof perimeter inequality}

In this section we prove Theorem~\ref{fv locally finite perimeter}, and state 
some important auxiliary results. 
The proof of Lemma~\ref{lem:5.4dominik} 
is postponed to Section~\ref{cylindrical section}, 
since it requires some results related to the circular symmetrisation. 
We start by proving Theorem~\ref{fv locally finite perimeter}. 
\begin{proof}[Proof of Theorem~\ref{fv locally finite perimeter}]
We will adapt the arguments of the proof of \cite[Theorem 1.1]{barchiesicagnettifusco}.
Let $G_{F_v}$ be the set associated with $F_v$ given by Theorem~\ref{thm:volpert}.
We start by proving \eqref{per ineq}. 
We will first prove the inequality when $B \subset (0, \infty) \setminus G_{F_v}$, and then 
in the case $B \subset G_{F_v}$. 
The case of a general Borel set $B \subset (0, \infty)$  
then follows by decomposing $B$ as $B = (B \setminus G_{F_v}) \cup (B \cap G_{F_v})$.

\vspace{.2cm}

\noindent
\textbf{Step 1:} We prove inequality \eqref{per ineq} when $B \subset (0, \infty) \setminus G_{F_v}$.
First observe that, thanks to Proposition~\ref{coarea} 
and \eqref{de giorgi structure and decomposition}, 
\begin{align}
&\left|D_{\parallel}\chi_{F_v}  \right|(\Phi(B\times \mathbb{S}^{n-1}))
= \int_{\partial^* F_v \cap \Phi(B\times \mathbb{S}^{n-1})}|\nu_{\parallel}^{F_v}(x)|d\mathcal{H}^{n-1}(x)
=\int_{B}\mathcal{H}^{n-2}((\partial^* F_v)_r) dr \nonumber \\
&= \int_{B \cap \{ 0 < \alpha_v \} }\mathcal{H}^{n-2}((\partial^* F_v)_r) dr
= \int_{B \cap ( \{ 0 < \alpha_v \} \setminus G_{F_v})}\mathcal{H}^{n-2}((\partial^* F_v)_r) dr = 0, \label{asd}
\end{align} 
where we used the fact that $B \subset (0, \infty) \setminus G_{F_v}$
and $\mathcal{H}^{1}(\{ 0 < \alpha_v \}\setminus G_{F_v})=0$. 
Therefore, thanks to Proposition~\ref{lem:4.17dominik}
\begin{align}
P(F_v;\Phi(B\times \mathbb{S}^{n-1})) 
&\leq r^{n-1}\left| D \xi_v \right|(B) + \left|D_{\|}\chi_{F_v}  \right|(\Phi(B\times \mathbb{S}^{n-1})) \nonumber \\
&= r^{n-1}\left| D \xi_v \right|(B)
\leq P(E;\Phi(B\times \mathbb{S}^{n-1})),
\label{eq:4.39dominik}
\end{align}
where in the last inequality we used \eqref{bound for D xi}.

\vspace{.2cm}

\noindent
\textbf{Step 2:} We prove inequality \eqref{per ineq} when $B \subset G_{F_v}$. 
We divide this part of the proof into further substeps.

\vspace{.2cm}

\noindent
\textbf{Step 2a:} we prove that 
\begin{equation}\label{eq:cagnetti 3.13}
P(E;\Phi(B\times \mathbb{S}^{n-1}))\geq P(E;\Phi(B\times \mathbb{S}^{n-1})\cap \{\nu_{\parallel}^{E}=0  \})
+\int_{B}\sqrt{p_E^2 (r)+g^{2} (r)}dr,
\end{equation}
where $g: (0,\infty) \to \R$ and $p_E : (0,\infty) \to [0, \infty)$ are defined as
$$ 
g(r) :=\int_{\partial^* E\cap \partial B(r)} \frac{\hat{x} \cdot \nu^{E}(x)}{|\nu_{\parallel}^{E}(x)|}d\mathcal{H}^{n-2}(x)
\quad \text{ and } \quad p_E (r):= \mathcal{H}^{n-2} (\partial^* E \cap \partial B(r)),
$$
for $\mathcal{H}^1$-a.e. $r \in (0, \infty)$, respectively.
We have
\begin{align*}
&P(E;\Phi(B\times \mathbb{S}^{n-1})) \\
&=P(E;\Phi(B\times \mathbb{S}^{n-1})\cap \{\nu_{\parallel}^{E}=0  \})+P(E;\Phi(B\times \mathbb{S}^{n-1})\cap \{\nu_{\parallel}^{E}\neq0  \})\\
&=P(E;\Phi(B\times \mathbb{S}^{n-1})\cap \{\nu_{\parallel}^{E}=0  \})+ \int_{\partial^* E \cap \Phi(B\times \mathbb{S}^{n-1})\cap \{\nu_{\parallel}^{E}\neq0  \}}d\mathcal{H}^{n-1}(x)\\
&=P(E;\Phi(B\times \mathbb{S}^{n-1})\cap \{\nu_{\parallel}^{E}=0  \})+ \int_{B}dr\int_{\partial^* E \cap \partial B(r)}\frac{1}{|\nu_{\parallel}^{E}(x)|}d\mathcal{H}^{n-2}(x)\\
&=P(E;\Phi(B\times \mathbb{S}^{n-1})\cap \{\nu_{\parallel}^{E}=0  \})
+ \int_{B}dr\int_{\partial^* E \cap \partial B(r)}
\sqrt{1+\left( \frac{\hat{x} \cdot \nu^{E}(x)}{| \nu_{\parallel}^{E}(x)|} \right)^2}d\mathcal{H}^{n-2}(x),
\end{align*}
where in the last equality we used the fact that 
$$
1 = | \nu_{\perp}^E |^2 + | \nu_{\parallel}^E |^2 = ( \hat{x} \cdot \nu^E )^2 + | \nu_{\parallel}^E |^2.
$$
Defining the function $f: \mathbb{R} \to [0, \infty)$ as
$$ 
f(t):= \sqrt{1 + t^{2}}, 
$$
we obtain
\begin{align*}
&P(E;\Phi(B\times \mathbb{S}^{n-1})) \\
&= P(E;\Phi(B\times \mathbb{S}^{n-1})\cap \{\nu_{\parallel}^{E}=0  \})
+ \int_{B}  dr \int_{\partial^* E \cap \partial B(r)} f \left( \frac{\hat{x} \cdot \nu^{E}(x)}{|\nu_{\parallel}^{E}(x)|} \right) \, d\mathcal{H}^{n-2}(x).
\end{align*}
%
%
%
%
Observing that $f$ is strictly convex, \eqref{eq:cagnetti 3.13} follows applying Jensen's inequality. 

\vspace{.2cm}

\noindent
\textbf{Step 2b:} We show that
\begin{align}
&\int_{B} \sqrt{p_E^{2} (r)+(r^{n-1}\xi_v'(r))^{2}} \, dr \nonumber \\
&\hspace{.8cm} \leq P(E;\Phi(B\times \mathbb{S}^{n-1})\cap \{\nu^{E}_{\|}=0  \}) 
+ \int_{B}\sqrt{p_E^{2} (r)+g^{2}(r)}\, dr. \label{eq:cagnetti 3.14}  
\end{align}
Let $H\subset \mathbb{N}$ be a finite set, and let $\{A_h \}_{h \in H}$ be a finite partition of Borel sets of $B$.
Note that, for each $h \in H$, we have 
$A_h \subset B \subset G_{F_v}$.
Therefore, thanks to Lemma~\ref{lem:4.14dominik}, 
for every $h \in H$
we have $r^{n-1}D\xi_v \mres A_h = r^{n-1} \xi_v' dr \mres A_h$ and
\begin{align}
& \int_{A_h} w_h r^{n-1}\xi_v'(r) \, dr = \int_{A_h} w_h r^{n-1} d D \xi_v(r) \nonumber \\
&= \int_{\partial^* E \cap \Phi(A_h \times \mathbb{S}^{n-1})\cap\{\nu^{E}_{\|}=0 \}}
w_h \, \hat{x} \cdot \nu^E(x)  \, d\mathcal{H}^{n-1}(x) \nonumber \\ 
&\hspace{.4cm}
+\int_{A_h }dr \int_{(\partial^* E)_r \cap \{\nu^{E}_{\|} \neq 0 \} } w_h
\frac{\hat{x} \cdot \nu^E (x)}{|\nu_{\|}^E(x)|}d\mathcal{H}^{n-2}(x) \nonumber \\
&= \int_{\partial^* E \cap \Phi(A_h \times \mathbb{S}^{n-1})\cap\{\nu^{E}_{\|}=0 \}}
w_h \, \hat{x} \cdot \nu^E(x)  \, d\mathcal{H}^{n-1}(x)  
+\int_{A_h } w_h \, g (r) \, dr. \label{here Dx is abs cont}
\end{align}
We will now use the fact that, by duality, we can write
\begin{equation} \label{duality}
\sqrt{1 + t^{2}}= \sup_{h \in \mathbb{N}}\left\{w_h t +\sqrt{1-w_h^2}  \right\} \quad \text{ for every } t \in\R,  
\end{equation}
where $\{w_h \}_{h \in \mathbb{N}}$ is a countable dense set in $(-1,1)$. 
Then, thanks to \eqref{here Dx is abs cont}
\begin{align*}
&\sum_{h\in H} \int_{A_h}\left(w_h r^{n-1}\xi_v'(r) + p_E (r)\sqrt{1- w_h^2} \right)dr \\
&=\sum_{h\in H}
\int_{\partial^* E\cap \Phi(A_h \times \mathbb{S}^{n-1}) \cap \{\nu^E_{\|}=0 \}}
w_h \, \hat{x} \cdot \nu^E(x) d\mathcal{H}^{n-1}(x) \\
&\hspace{.5cm}
+ \sum_{h\in H} \int_{A_h} \Big( w_h \, g (r) + p_E (r)\sqrt{1-w_h^2} \Big) dr  
\\ 
&\leq \sum_{h\in H} 
\int_{\partial^* E\cap \Phi(A_h\times \mathbb{S}^{n-1})\cap \{\nu^E_{\|}=0 \}} 
| \hat{x} \cdot \nu^E(x) |d\mathcal{H}^{n-1}(x) \\
&\hspace{.5cm} + \sum_{h\in H}  \int_{A_h} p_E (r) \left( w_h  \frac{g(r)}{p_E (r)}+\sqrt{1-w_h^2} \right) dr 
\\
&\leq \sum_{h\in H} \left(P(E;\Phi( A_h\times \mathbb{S}^{n-1}) \cap \{ \nu^E_{\|}=0\} )  \right)
+ \int_{A_h} p_E (r) \sqrt{1+\frac{g^2(r)}{p_E^2 (r)}} dr \\
&=P(E;\Phi(B\times \mathbb{S}^{n-1})\cap \{\nu^E_{\|}=0  \}) + 
\int_B\sqrt{p_E^2 (r)+g^2(r)}dr,
\end{align*}
where we applied identity \eqref{duality} with $t = g(r)/p_E (r)$,
and we also used the fact that $p_E (r) = 0$
for $\mathcal{H}^1$-a.e. $r \notin \{ 0 < \alpha_v < \pi\}$, 
thanks to Volper't theorem.   
Applying Lemma~\ref{lem:cagnetti 2.6} to the functions 
$$ 
\varphi_h(r)= p_E (r) \left( w_h \frac{r^{n-1}\xi'_v(r)}{p_E (r)} +\sqrt{1-w_h^2}\right),
$$  
we obtain \eqref{eq:cagnetti 3.14}.

\vspace{.2cm}

\noindent
\textbf{Step 2c:} We conclude the proof of Step 2.
In the special case $E=F_v$, thanks to Vol'pert Theorem and Lemma~\ref{lem:4.14dominik} we have 
\begin{align}
&P(F_v;\Phi(B\times \mathbb{S}^{n-1})) = \mathcal{H}^{n-1}(\partial^* F_v \cap \Phi(B\times \mathbb{S}^{n-1})) \nonumber \\
&= \int_{B \cap \{ 0 < \alpha_v < \pi\}}\int_{\partial^* (F_v)_r}\frac{1}{|\nu_{\parallel}^{F_v}(x)|}d\mathcal{H}^{n-2}(x) dr \nonumber \\
&=\int_{B \cap \{ 0 < \alpha_v < \pi\}}\int_{\partial^*  (F_v)_r  } 
\sqrt{1+\left( \frac{\nu^{F_v}(x)}{|\nu_{\parallel}^{F_v}(x)|} \right)^2}d\mathcal{H}^{n-2}(x) dr \nonumber \\
&=\int_{B \cap \{ 0 < \alpha_v < \pi\}}\sqrt{p_{F_v}^2(r)+(r^{n-1}\xi'_v(r))^2}dr. \label{useful for Fv} 
\end{align}
Using the isoperimetric inequality \eqref{isop ineq} 
together with \eqref{eq:cagnetti 3.14} and \eqref{eq:cagnetti 3.13} we then have, 
\begin{align*}
&P(F_v;\Phi(B\times \mathbb{S}^{n-1})) 
\leq \int_{B\cap \{ 0 < \alpha_v < \pi\}}\sqrt{p_E^2(r)+(r^{n-1}\xi'_v(r))^2}dr\\
&\leq  P(E;\Phi(B\times \mathbb{S}^{n-1})\cap \{\nu^E_{\|}=0  \}) + \int_B\sqrt{p_E^2(r)+g^2(r)}dr\\
&\leq P(E;\Phi(B\times \mathbb{S}^{n-1})),
\end{align*}
from which we conclude.

\vspace{.2cm}

\noindent
\textbf{Step 3:} We conclude the proof of the theorem.
Suppose $P (E) = P(F_v)$.
Then, in particular, all the inequalities in Step 2 hold true as equalities.
At the end of Step 2c we used the fact that, by the isoperimetric inequality \eqref{isop ineq}, we have 
$$
p_{F_v} (r) \leq p_E (r) \qquad \text{ for $\mathcal{H}^1$-a.e. $r \in \{ 0 < \alpha_v < \pi\}$}.
$$
If the above becomes an equality, this means that for $\mathcal{H}^1$-a.e. $r \in \{ 0 < \alpha_v < \pi\}$
the slice $E_r$ is a spherical cap. Finally, the fact that for $\mathcal{H}^1$-a.e. $r \in \{ 0 < \alpha_v < \pi\}$
we have 
$$
\mathcal{H}^{n-2} (\partial^* (E_r) \Delta (\partial^* E)_r) = 0
$$
follows from Vol'pert Theorem~\ref{thm:volpert}, and this shows (a).

Let us now prove (b). 
If $P (E) = P(F_v)$, the Jensen's inequality at the end of Step 2b, 
for the strictly convex function 
$$ 
f(t):= \sqrt{1 + t^{2}}, 
$$
becomes an equality.
This implies that for $\mathcal{H}^1$-a.e. $r \in \{ 0 < \alpha_v < \pi\}$ the function 
$$
x \longmapsto \frac{\hat{x} \cdot \nu^{E}(x)}{|\nu_{\parallel}^{E}(x)|} 
$$
is $\mathcal{H}^{n-2}$-a.e. constant in $\partial^* E_r$.
Since, for $\mathcal{H}^{n-2}$-a.e.  $x \in \partial^* E_r$, we have 
$$
1 = |\nu_{\parallel}^{E}(x)|^2 + (\hat{x} \cdot \nu^{E}(x))^2,
$$
this implies that 
$$
x \longmapsto \frac{(\hat{x} \cdot \nu^{E}(x))^2}{|\nu_{\parallel}^{E}(x)|^2} 
= 1 - \frac{1}{|\nu_{\parallel}^{E}(x)|^2}
$$
is $\mathcal{H}^{n-2}$-a.e. constant in $\partial^* E_r$.
Therefore, the two functions 
$$
x \longmapsto \nu^E (x) \cdot \hat{x} \qquad \text{ and } \qquad 
x \longmapsto | \nu^E_{\parallel}| (x)
$$
are constant $\mathcal{H}^{n-2}$-a.e. in $(\partial^* E )_r$.

\end{proof}
The previous result allows us to prove a useful proposition
(see also \cite[Proposition~3.4]{barchiesicagnettifusco}).
\begin{proposition}\label{proposition that later will given the perimeter}
Let $v: (0,\infty) \to [0, \infty)$ be a measurable function 
satisfying \eqref{bound on v} such that $F_v$ 
is a set of finite perimeter and finite volume, 
let $E$ be a spherically $v$-distributed set of finite perimeter, and let $f:(0, \infty) \to [0, \infty]$ be a Borel function.
Then, 
\begin{align}
&\int_{\partial^* E} f (|x|) \, d \mathcal{H}^{n-1} (x) \nonumber \\
&\geq \int_0^{\infty} f (r) \sqrt{p_E^2(r)+(r^{n-1}\xi'_v(r))^2} \, dr
+ \int_0^{\infty} f (r) r^{n-1} d | D^s \xi_v | (r). \label{more than perimeter}
\end{align}
Moreover, in the special case $E = F_v$, equality holds true. 
\end{proposition}

\begin{proof}
To prove the proposition it is enough to consider the case in which 
$f = \chi_B$, with $B \subset (0,\infty)$ Borel set. 

First, suppose $B \subset (0,\infty) \setminus G_{F_v}$. Thanks to Lemma~\ref{lem:4.14dominik}, 
in this case we have $\xi'_v = 0$ in $B$ and $|r^{n-1} D \xi_v| (B) = |r^{n-1} D^s \xi_v| (B)$.
Then, from \eqref{bound for D xi} it follows that 
\begin{align*}
&\int_{\partial^* E} \chi_B (|x|) \, d \mathcal{H}^{n-1} (x) 
 = P (E ; \Phi (B \times \mathbb{S}^{n-1}))  
\geq | D_{\perp} \chi_E | (\Phi (B \times \mathbb{S}^{n-1})) \\
& \geq |r^{n-1} D \xi_v| (B) = |r^{n-1} D^s \xi_v| (B) 
= \int_0^{\infty} \chi_B (r) r^{n-1} d | D^s \xi_v | (r) \\
&= \int_0^{\infty} \chi_B (r) \sqrt{p_E^2(r)+(r^{n-1}\xi'_v(r))^2} \, dr
+ \int_0^{\infty} \chi_B (r) r^{n-1} d | D^s \xi_v | (r),
\end{align*}
where we also used the fact that
$p_E = 0$ $\mathcal{H}^1$-a.e. in $B$, since
$$
\mathcal{H}^n (E \cap \Phi (B \times \mathbb{S}^{n-1})) \leq 
\int_{ \{ v = 0 \} } \, dr \int_{E_r} \, d \mathcal{H}^{n-1} (x)
= \int_{\{ v = 0 \}} v (r)\, dr = 0.
$$

Let us now assume $B \subset G_{F_v}$. In this case, 
by Lemma~\ref{lem:4.14dominik} we have 
$|r^{n-1} D^s \xi_v| (B) = 0$. Then, thanks to \eqref{eq:cagnetti 3.13}
and \eqref{eq:cagnetti 3.14} we obtain
\begin{align*}
&\int_{\partial^* E} \chi_B (|x|) \, d \mathcal{H}^{n-1} (x) 
 = P (E ; \Phi (B \times \mathbb{S}^{n-1}))  \\
 & \geq P(E;\Phi(B\times \mathbb{S}^{n-1})\cap \{\nu_{\parallel}^{E}=0  \})
+\int_{B}\sqrt{p_E^2 (r)+g^{2} (r)}dr \\
&\geq \int_{B} \sqrt{p_E^{2} (r)+(r^{n-1}\xi_v'(r))^{2}} \, dr \\
&= \int_0^{\infty} \chi_B (r) \sqrt{p_E^2(r)+(r^{n-1}\xi'_v(r))^2} \, dr
+ \int_0^{\infty} \chi_B (r) r^{n-1} d | D^s \xi_v | (r),
\end{align*}
so that \eqref{more than perimeter} follows.

Consider now the case $E = F_v$. If $B \subset G_{F_v}$, 
recalling again that by Lemma~\ref{lem:4.14dominik} we have 
$|r^{n-1} D^s \xi_v| (B) = 0$, thanks to \eqref{useful for Fv} we obtain
\begin{align*}
&\int_{\partial^* F_v} \chi_B (|x|) \, d \mathcal{H}^{n-1} (x)  = P(F_v;\Phi(B\times \mathbb{S}^{n-1})) 
=\int_{B}\sqrt{p_{F_v}^2(r)+(r^{n-1}\xi'_v(r))^2} \, dr \\
&= \int_0^{\infty} \chi_B (r) \sqrt{p_{F_v}^2(r)+(r^{n-1}\xi'_v(r))^2} \, dr
+ \int_0^{\infty} \chi_B (r) r^{n-1} d | D^s \xi_v | (r).
\end{align*}
If, instead, $B \subset (0,\infty) \setminus G_{F_v}$, then 
$\xi'_v = 0$ in $B$ and $|r^{n-1} D \xi_v| (B) = |r^{n-1} D^s \xi_v| (B)$.
Therefore, thanks to \eqref{eq:4.39dominik}, 
\begin{align*}
&\int_{\partial^* F_v} \chi_B (|x|) \, d \mathcal{H}^{n-1} (x)  = P(F_v;\Phi(B\times \mathbb{S}^{n-1})) 
\leq r^{n-1}\left| D \xi_v \right|(B) = |r^{n-1} D^s \xi_v| (B) \\
&= \int_0^{\infty} \chi_B (r) \sqrt{p_{F_v}^2(r)+(r^{n-1}\xi'_v(r))^2} \, dr
+ \int_0^{\infty} \chi_B (r) r^{n-1} d | D^s \xi_v | (r).
\end{align*}
\end{proof}

\noindent
An important consequence of the above proposition is a formula 
for the perimeter of $F_v$.
\begin{corollary} \label{corollary perimeter Fv}
Let $v: (0,\infty) \to [0, \infty)$ be a measurable function 
satisfying \eqref{bound on v} such that $F_v$ 
is a set of finite perimeter and finite volume.
Then
\begin{equation} \label{perimeter Fv}
P (F_v; \Phi (B \times \mathbb{S}^{n-1}))
= \int_B \sqrt{p_{F_v}^2(r)+(r^{n-1}\xi'_v(r))^2} \, dr
+ \int_B r^{n-1} d | D^s \xi_v | (r). 
\end{equation}
\end{corollary}
We conclude this section with  an important result,  that will be used later.
\begin{proposition}\label{prop:5.3dominik}
Let $v: (0,\infty) \to [0, \infty)$ be a measurable function 
satisfying \eqref{bound on v} such that $F_v$ 
is a set of finite perimeter and finite volume, and let $I\subset(0,+\infty)$ be an open set. 
Then the following three statements are equivalent:
\begin{itemize}
\item[(i)] $\mathcal{H}^{n-1}\left(\Big\{x\in \partial^* F_v \cap \Phi(I\times \mathbb{S}^{n-1}): \nu^{F_v}_{\|}(x)=0  \Big\}  \right)=0$;

\vspace{.2cm}

\item[(ii)] $\xi_v \in W^{1,1}_{\textnormal{loc}} (I)$;

\vspace{.2cm}

\item[(iii)] $P(F_v;\Phi(B\times \mathbb{S}^{n-1}))=0$ for every Borel set $B\subset I$, such that $\mathcal{H}^1(B)=0$.

\end{itemize}
\end{proposition}

\begin{remark} \label{implication iii to i if I Borel}
Note that the equivalence  \textnormal{(iii)} $\Longleftrightarrow$ \textnormal{(i)}
holds true also if $I$ is a Borel set.
To show this, we only need to prove that 
 \textnormal{(i)} $\Longrightarrow$ \textnormal{(iii)}, 
 since the opposite implication is given by repeating  
 Step 3 of the proof of Proposition~\ref{prop:5.3dominik}. 
 Suppose \textnormal{(i)} is satisfied. 
 Then from \eqref{D xi} we have $r^{n-1} D \xi_v \mres I = r^{n-1} \xi'_v \mres I$.
 Therefore, thanks to \eqref{perimeter Fv}
 $$
P (F_v; \Phi (B \times \mathbb{S}^{n-1}))
= \int_B \sqrt{p_{F_v}^2(r)+(r^{n-1}\xi'_v(r))^2} \, dr \quad \text{ for every Borel set }
B \subset I,
$$
which implies (iii).
\end{remark}
\begin{proof}
We divide the proof into three steps.

\vspace{.2cm}

\noindent
\textbf{Step 1:} (i) $\Longrightarrow$ (ii). 
Recall that, by Lemma \ref{lem:4.12Dominik}, $\xi_v \in BV_{\textnormal{loc}}(I)$.
If (i) is satisfied, from \eqref{D xi} we have 
$r^{n-1} D \xi_v \mres I = r^{n-1} \xi'_v \mres I$, 
which implies (ii). 

\vspace{.2cm}

\noindent
\textbf{Step 2:} (ii) $\Longrightarrow$ (iii). 
This implication follows from formula \eqref{perimeter Fv}.

\vspace{.2cm}

\noindent
\textbf{Step 3:} (iii) $\Longrightarrow$ (i) (note that we will not use the fact that $I$
is open). 
Assume (iii) holds true. Then,
$$
\mathcal{H}^{n-1} 
\left(\Big\{x\in \partial^* F_v \cap \Phi(I\times \mathbb{S}^{n-1}): 
\nu^{\partial^* F_v}_{\|}(x)=0  \Big\}  \right)
\leq P(\partial^* F_v;\Phi((B_0\cap I) \times \mathbb{S}^{n-1}))=0,
$$
where we used the fact that $\mathcal{H}^1 (B_0) = 0$, 
thanks to \eqref{measure of B0 is 0}.  
\end{proof}

\section{Circular symmetrisation  and proof of Lemma~\ref{lem:5.4dominik}}\label{cylindrical section} 
In this section we  show  Theorem~\ref{theorem ineq cylindrical},
Lemma~\ref{lem:5.4dominik cyl}, and finally Lemma~\ref{lem:5.4dominik}. 
We will only  sketch the proofs,  since in most cases 
 the arguments  follow the lines of the proofs  in   Section~\ref{preliminary spherical},
Section~\ref{section properties v and xi}, 
and Section~\ref{section spherical proof perimeter inequality}.

We start with some notation which, together with 
that one already given in the Introduction, 
will be extensively used in this section.
Let $(r, x') \in (0, \infty) \times \R^{n-2}$, $\beta \in [0, \pi]$, and let $p \in \mathbb{S}^1$.
The circular arc of centre $(r p, x' )$ and radius $\beta$ is the set
$$
\mathcal{B}_{\beta} (r p, x' ) := 
\{ x \in \partial  B ((0,x'), r) \cap \Pi_{x'} : 
\text{dist}_{\mathbb{S}^{1}} (\hat{x}_{12} , r p) < \beta \},
$$
If $\ell  : (0, \infty) \times \R^{n-2} \to [0, \infty)$ is a measurable function
satisfying \eqref{bound on ell}, we define $\alpha^{\ell} : (0, \infty) \times \R^{n-2} \to [0, \pi]$
and $\xi^{\ell} : (0, \infty) \times \R^{n-2} \to [0, 2 \pi]$ as 
$$
\alpha^{\ell} := \frac{1}{2 r} \ell (r, x') \qquad \text{ and  } \qquad
\xi^{\ell} (r, x') = \frac{1}{r} \ell (r, x') = 2 \alpha^{\ell} (r, x').
$$
Note that in this case the relation between $\alpha^{\ell}$ and $\xi^{\ell}$ is linear.
%
If $\mu$ is an $\R^n$-valued Radon measure on $\R^n \setminus \{ x_{12} = 0 \}$, 
we will write $\mu = \mu_{12 \perp} + \mu_{12 \parallel}$, 
where $\mu_{12 \perp}$ and $\mu_{12 \parallel}$
are the $\R^n$-valued Radon measures on $\R^n \setminus \{ x_{12} = 0 \}$ such that
\begin{align*}
\int_{\R^n \setminus \{ x_{12} = 0 \}} \varphi  \cdot  d \mu_{12 \perp} 
= \int_{\R^n \setminus \{ x_{12} = 0 \}} \varphi_{12 \perp}  \cdot d \mu, 
\end{align*}
and
\begin{align*}
\int_{\R^n \setminus \{ x_{12} = 0 \}} \varphi  \cdot d \mu_{12 \parallel} 
= \int_{\R^n \setminus \{ x_{12} = 0 \}} \varphi_{12 \parallel}  \cdot d \mu,
\end{align*}
for every $\varphi \in C_c (\R^n \setminus \{ x_{12} = 0 \}; \R^n)$.
The next two results  play the role  of Proposition~\ref{coarea}
and Vol'pert Theorem~\ref{thm:volpert},  in the context of circular symmetrisation. 
\begin{proposition} 
Let E be a set of finite perimeter in $\R^n$ 
and let $g:\R^n \rightarrow [0,\infty]$ be a Borel function. Then, 
$$ 
\int_{\partial^* E} g(x) |\nu^{E}_{1 2 \parallel}(x)| d\mathcal{H}^{n-1}(x) 
= \int_{(0, \infty) \times \R^{n-2}} dr \, d x' \int_{(\partial^* E)_{(r, x')} } g(x) \, d\mathcal{H}^{0}(x).
$$
\end{proposition}

\begin{proof}
In this case, the result follows applying \cite[Remark~2.94]{AFP} 
with $N = n-1$, 
$M = n$, $k = n-1$, and $f (x) = (|x_{12}|, x')$. 
\end{proof}

\begin{theorem} \label{thm:volpert circular}
Let $\ell  : (0, \infty) \times \R^{n-2} \to [0, \infty)$ be a measurable function
satisfying \eqref{bound on ell}, and let $E \subset \R^n$ be an 
circularly $\ell$-distributed set of finite perimeter and finite volume.
Then, there exists a Borel set $G^{\ell}_E \subset \{ \alpha^{\ell} > 0 \}$
with $\mathcal{H}^{n-1} (\{ \alpha^{\ell} > 0 \} \setminus G^\ell_E) = 0$,
such that 

\begin{itemize}

\item[(i)] for every $(r, x') \in G^\ell_E$:

\vspace{.1cm}

\begin{itemize}

\item[(ia)] $E_{(r, x')}$ is a set of finite perimeter in $\partial B_r (0,x') \cap \Pi_{x'}$;

\vspace{.1cm}

\item [(ib)] $\partial^{*} (E_{(r, x')})     =   (\partial^* E)_{(r, x')} $;

\end{itemize}

\vspace{.1cm}

\item[(ii)] for every $(r, x') \in G^\ell_E \cap \{ 0 < \alpha^{\ell}  < \pi \}$:
 
 \vspace{.1cm}

 \begin{itemize}
\item[(iia)] $| \nu^E_{12 \parallel} (r \omega, x') |  > 0$;

\vspace{.1cm}

\item[(iib)]$\nu^E_{12 \parallel} (r \omega, x') = \nu^{E_{(r, x')}} (r \omega, x')
| \nu^E_{12 \parallel} (r \omega, x') |$,   

\vspace{.1cm}

\end{itemize}
for \textbf{every} $\omega \in \mathbb{S}^{1}$ 
such that $( r \omega , x') \in \partial^{*} (E_{(r, x')})     =   (\partial^* E)_{(r, x')}$.
\vspace{.1cm}

\end{itemize}
%
\end{theorem}

\begin{proof}
The statement follows applying the results of \cite[Section~2.5]{GMSbook1}, 
where  the slicing of codimension higher than $1$  for currents is  defined.
\end{proof}
\begin{remark}
Note that, if $(r, x') \in G^{\ell}_E$, 
conditions (iia) and (iib)
are satisfied for \textbf{every} $\omega \in \mathbb{S}^{1}$ 
such that $( r \omega , x') \in \partial^{*} (E_{(r, x')})     =   (\partial^* E)_{(r, x')}$.
This is due to the fact that the circular symmetrisation has codimension $1$.
Such property fails, in general, for the spherical symmetrisation 
(see Remark~\ref{remark gmt}).
\end{remark}

%
%


\begin{remark} \label{remark gmt circular}
An argument similar to that one used in Remark~\ref{remark gmt}
shows that 
$$
\mathcal{H}^{n-1}
( \partial^* E \cap \Phi_{12} (G^{\ell}_E \times \mathbb{S}^{1})
\cap
\{ \nu^E_{12 \parallel}   =0 \}  ) = 0.
$$
As a consequence, the measure
$\lambda^{\ell}_E$ defined as:
$$
\lambda^{\ell}_E (B) :=\int_{\partial^* E \cap \Phi_{12} (B \times \mathbb{S}^{1})
\cap
\{ \nu^E_{12 \parallel}   =0 \} }
\hat{x}_{12} \cdot \nu^E(x)  \, d\mathcal{H}^{1}(x), 
$$
for every Borel set $B \subset (0, \infty) \times \R^{n-2}$,
is singular with respect to the Lebesgue measure in $(0, \infty) \times \R^{n-2}$.
\end{remark}
 The following result plays the role of Lemma~\ref{lem:4.12Dominik}
in the context of circular symmetrisation. 
\begin{lemma}  
Let $\ell  : (0, \infty) \times \R^{n-2} \to [0, \infty)$ be a measurable function
satisfying \eqref{bound on ell}, and let $E \subset \R^n$ be an 
circularly $\ell$-distributed set of finite perimeter and finite volume.
Then, $\ell \in BV_{\textnormal{loc}} ((0, \infty)\times \R^{n-2})$.
Moreover,  $\xi^{\ell} \in BV_{\textnormal{loc}} ((0, \infty)\times \R^{n-2})$ and 
\begin{equation*} 
\int_{(0, \infty)\times \R^{n-2}}  \psi (r, x') \, r \, d D_r \xi^{\ell} (r, x')
= \int_{\R^n \setminus \{ x_{12} = 0 \}} \psi (|x_{12}|, x') \, \hat{x}_{12} 
\cdot d D_{12 \perp} \chi_E (x), 
\end{equation*}
for every bounded Borel function $\psi: (0, \infty)\times \R^{n-2} \to \R$, 
where $D_r \xi^{\ell}$ denotes the $r$-component of the $\R^{n-1}$-valued 
Radon measure $D \xi^{\ell}$.
As a consequence, 
\begin{equation*} 
| r D_r \xi^{\ell} | (B) \leq | D_{12 \perp} \chi_E | (\Phi_{12} (B \times \mathbb{S}^{1})), 
\end{equation*}
for every Borel set $B \subset (0, \infty)\times \R^{n-2}$. 
In particular, $r D_r \xi^{\ell}$ is a bounded Radon measure on $(0, \infty)\times \R^{n-2}$.
Finally, 
$$
D_{x'} \ell (B) = \int_{\partial^* E \cap \Phi_{12} (B \times \mathbb{S}^1)} \nu^E_{x'} (x) \, d \mathcal{H}^{n-1}  (x), 
$$
for every Borel set $B \subset (0, \infty)\times \R^{n-2}$. 
\end{lemma}

\begin{remark}
Unlike what happened when we were considering the spherical symmetrisation, 
now the function $\ell$ might fail to be in $BV ((0, \infty)\times \R^{n-2})$.
Indeed, in Step 1 of the proof of Lemma~\ref{lem:4.12Dominik}
we used the fact that for $r$ bounded we are in a bounded set.
This is not true in the context of circular symmetrisation. 
\end{remark}
The next lemma, which is related to Lemma~\ref{lem:4.14dominik}, 
will show the advantage of considering a symmetrisation of codimension $1$.

\begin{lemma}
Let $\ell  : (0, \infty) \times \R^{n-2} \to [0, \infty)$ be a measurable function
satisfying \eqref{bound on ell}, and let $E \subset \R^n$ be an 
circularly $\ell$-distributed set of finite perimeter and finite volume.
Then
\begin{align*}
( r \, d D_r \xi^{\ell} )(B) 
&=\int_{\partial^* E \cap \Phi_{12} (B \times \mathbb{S}^{1})\cap\{\nu^{E}_{12 \|}=0 \}}
\hat{x}_{12} \cdot \nu^E(x)  \, d\mathcal{H}^{n-1}(x)  \\
&\hspace{.4cm}+\int_{B}dr \, d x' \int_{(\partial^* E)_{(r, x')} \cap \{\nu^{E}_{12 \|} \neq 0 \} } 
\frac{\hat{x}_{12} \cdot \nu^E (x)}{| \nu_{12  \|}^E(x)|}d\mathcal{H}^{0}(x). 
\end{align*}
for every Borel set $B \subset (0, \infty)\times \R^{n-2}$. 
Moreover, 
$$
r (\xi^{\ell})' (r, x') = \int_{(\partial^* E)_{(r, x')} \cap \{\nu^{E}_{12 \|} \neq 0 \} } 
\frac{\hat{x}_{12} \cdot \nu^E (x)}{| \nu_{12  \|}^E(x)|}d\mathcal{H}^{0}(x),
$$
for $\mathcal{H}^{n-1}$-a.e. $(r, x') \in (0, \infty)\times \R^{n-2}$, 
where $(\xi^{\ell})'$ denotes the approximate differential 
of $\xi^{\ell}$ with respect to $r$. Similarly, 
\begin{align*}
D_{x'} \ell (B) 
&= \int_{\partial^* E \cap \Phi_{12} (B \times \mathbb{S}^{1})\cap\{\nu^{E}_{12 \|}=0 \}} \nu^E_{x'} (x) \, d \mathcal{H}^{n-1}  (x)
  \\
&\hspace{.4cm}+\int_{B}dr \, d x' \int_{(\partial^* E)_{(r, x')} \cap \{\nu^{E}_{12 \|} \neq 0 \} } 
\frac{\nu^E_{x'} (x)}{| \nu_{12  \|}^E(x)|} d\mathcal{H}^{0}(x). 
\end{align*}
for every Borel set $B \subset (0, \infty)\times \R^{n-2}$, and 
$$
\nabla_{x'} \ell (r, x') = \int_{(\partial^* E)_{(r, x')} \cap \{\nu^{E}_{12 \|} \neq 0 \} } 
\frac{\nu^E_{x'} (x)}{| \nu_{12  \|}^E(x)|} d\mathcal{H}^{0}(x),
$$
for $\mathcal{H}^{n-1}$-a.e. $(r, x') \in (0, \infty)\times \R^{n-2}$, 
where $\nabla_{x'} \ell$ denotes the approximate gradient 
of $\ell$ with respect to $x'$.
\end{lemma}
The  next result should be compared to Proposition~\ref{lem:4.17dominik}. 
\begin{proposition} 
Let $\ell  : (0, \infty) \times \R^{n-2} \to [0, \infty)$ be a measurable function
satisfying \eqref{bound on ell}, and suppose that there exists
an circularly $\ell$-distributed set $E \subset \R^n$ be  of finite perimeter and finite volume.
Then, $F^{\ell}$ is a set of finite perimeter in $\R^n$. 
Moreover, for every Borel set $B\subset (0,+\infty) \times \R^{n-2}$ 
\begin{equation*} 
P(F^{\ell};\Phi_{12} (B\times \mathbb{S}^{1})) 
\leq | D_{x'} \ell | (B) +  \big| r  D_r \xi^{\ell} \big| (B) 
+ \left|D_{12 \parallel} \chi_{F_v}  \right|(\Phi_{12} (B\times \mathbb{S}^{1})).
\end{equation*}
\end{proposition}
We  are now ready to prove Theorem~\ref{theorem ineq cylindrical}. 
\begin{proof}[Proof of Theorem~\ref{theorem ineq cylindrical}]
Using the results shown above, Theorem~\ref{theorem ineq cylindrical} can be proved 
by following the lines of the proof of Theorem~\ref{fv locally finite perimeter}.
\end{proof}
We  will now state the results that are need to prove Lemma~\ref{lem:5.4dominik cyl}.
The next proposition should be compared to Proposition~\ref{proposition that later will given the perimeter}. 
\begin{proposition} 
Let $\ell  : (0, \infty) \times \R^{n-2} \to [0, \infty)$ be a measurable function
satisfying \eqref{bound on ell} such that $F^{\ell}$ 
is a set of finite perimeter and finite volume, 
let $E \subset \R^n$ be an circularly $\ell$-distributed set of finite perimeter, 
and let $f:(0, \infty) \times \R^{n-2} \to [0, \infty]$ be a Borel function.
Then, 
\begin{align*}
&\int_{\partial^* E} f (|x_{12}|, x') \, d \mathcal{H}^{n-1} (x) \\
&\geq \int_{(0, \infty) \times \R^{n-2}} f (r, x') \sqrt{p_E^2(r, x') +(r (\xi^{\ell})' (r, x'))^2
+ | \nabla_{x'} \ell (r, x')|^2} \, dr \, d x' \\
&+ \int_{(0, \infty) \times \R^{n-2}} f (r, x') \, r \, d | D_r^s \xi^{\ell} | (r, x')
+ \int_{(0, \infty) \times \R^{n-2}} f (r, x')  d | D_{x'}^s \ell | (r, x'). 
\end{align*}
Moreover, in the special case $E = F^{\ell}$, equality holds true. 
\end{proposition}
A straightforward consequence of the previous result is the following formula
for the perimeter of $F^{\ell}$.

\begin{corollary} 
Let $\ell  : (0, \infty) \times \R^{n-2} \to [0, \infty)$ be a measurable function
satisfying \eqref{bound on ell} such that $F^{\ell}$ 
is a set of finite perimeter and finite volume.
Then
\begin{align*} 
&P (F^{\ell}; \Phi_{12} (B \times \mathbb{S}^{1})) \\
&= \int_{B}  \sqrt{p_E^2(r, x') +(r (\xi^{\ell})' (r, x'))^2
+ | \nabla_{x'} \ell (r, x')|^2} \, dr \, d x' 
+  | r D_r^s \xi^{\ell} | (B)
+  | D_{x'}^s \ell | (B). 
\end{align*}
\end{corollary}
Next lemma relies on the fact that the circular symmetrisation has codimension $1$.
The proof can be obtained by repeating the  arguments  
used in the proof of \cite[Lemma~4.1]{ChlebikCianchiFuscoAnnals05}.
\begin{lemma} 
Let $\ell  : (0, \infty) \times \R^{n-2} \to [0, \infty)$ be a measurable function
satisfying \eqref{bound on ell}, let $E \subset \R^n$ be
an circularly $\ell$-distributed set of finite perimeter and finite volume, 
and let $A \subset (0,+\infty) \times \R^{n-2}$ be a Borel set. 
Then, 
\begin{equation*} 
\mathcal{H}^{n-1} \Big( \{x\in \partial^* E : \nu^E_{12 \|}(x)=0 \} \cap \Phi_{12} ( A \times \mathbb{S}^{1}) \Big)=0.
\end{equation*}
if and only if 
\begin{equation*} \label{no lateral surface for E}
P(E;\Phi_{12} (B \times \mathbb{S}^{1}))=0 
\quad \text{ for every Borel set } B \subset A \text{ with } \mathcal{H}^{n-1} (B) = 0.
\end{equation*}
\end{lemma}
The  next proposition  can be proved with the same arguments used to show Proposition~\ref{prop:5.3dominik}. 
\begin{proposition}
Let $\ell  : (0, \infty) \times \R^{n-2} \to [0, \infty)$ be a measurable function
satisfying \eqref{bound on ell} such that $F^{\ell}$ 
is a set of finite perimeter and finite volume,
and let $\Omega \subset (0,+\infty) \times \R^{n-2}$ be an open set. 
Then the following three statements are equivalent:
\begin{itemize}
\item[(i)] $\mathcal{H}^{n-1}\left(\Big\{x\in \partial^* F^{\ell} 
\cap \Phi_{12}(\Omega\times \mathbb{S}^{1}): \nu^{F^{\ell} }_{1 2 \|}(x)=0  \Big\}  \right)=0$;

\vspace{.2cm}

\item[(ii)] $\xi^{\ell} \in W^{1,1}_{\textnormal{loc}} (\Omega)$ and $\ell \in W^{1,1}_{\textnormal{loc}} (\Omega)$;

\vspace{.2cm}

\item[(iii)] $P(F^{\ell} ;\Phi_{12}(B \times \mathbb{S}^{1}))=0$ for every Borel set $B \subset \Omega$, 
such that $\mathcal{H}^{n-1}(B)=0$.

\end{itemize}
\end{proposition}

\begin{proof}[Proof of Lemma~\ref{lem:5.4dominik cyl}]
Once all the results above are established, Lemma~\ref{lem:5.4dominik cyl}
can be shown by adapting the arguments used in the proof of \cite[Proposition~4.2]{ChlebikCianchiFuscoAnnals05}.
\end{proof}

We can now prove Lemma~\ref{lem:5.4dominik}.
 As already mentioned in the Introduction, the proof relies  on Theorem~\ref{theorem ineq cylindrical} 
and Lemma~\ref{lem:5.4dominik cyl}. 
\begin{proof}[Proof of Lemma~\ref{lem:5.4dominik}]
We divide the proof into steps.

\vspace{.2cm}

\noindent
\textbf{Step 1:} We show that \eqref{eq:5.6dominik} $\Longrightarrow$ \eqref{eq:5.7dominik}.
Suppose (\ref{eq:5.6dominik}) is satisfied. 
Then, from \eqref{D xi} we have $r^{n-1} D \xi_v \mres I = r^{n-1} \xi'_v \mres I$.
Thanks to \eqref{perimeter Fv}, this implies that
\begin{equation*} 
P (F_v; \Phi (B \times \mathbb{S}^{n-1}))
= \int_B \sqrt{p_{F_v}^2(r)+(r^{n-1}\xi'_v(r))^2} \, dr.  
\qquad \text{ for every Borel set } B \subset I. 
\end{equation*}
In particular, condition (iii) of Proposition~\ref{prop:5.3dominik}
is satisfied. Then, \eqref{eq:5.7dominik} follows from Remark~\ref{implication iii to i if I Borel}.

\vspace{.2cm}

\noindent
\textbf{Step 2:} We show that if $P(E;\Phi(I\times \mathbb{S}^{n-1}))=P(F_v;\Phi(I\times \mathbb{S}^{n-1}))$, then \eqref{eq:5.7dominik} implies \eqref{eq:5.6dominik}.
To this aim, we first prove an auxiliary result.

\vspace{.2cm}

\noindent
\textbf{Step 2a:} We show that if $\overline{F} \subset \R^n$
is a set of finite perimeter such that $(\overline{F})_r$
is a spherical cap for $\mathcal{H}^1$-a.e. $r > 0$, and
\begin{equation} \label{assumption on Fbar}
\mathcal{H}^{n-1}\left(\left\{x\in \partial^* \overline{F} \cap \Phi(I\times \mathbb{S}^{n-1}): 
\nu^{\overline{F}}_{\parallel} (x)=0  \right\}  \right)=0,
\end{equation}
then $\mathcal{H}^{n-1} (B^j) = 0$ for every $j = 2, \ldots, n$, where
$$
B^j:= \left\{x\in \partial^* \overline{F} \cap \Phi(I\times \mathbb{S}^{n-1}): 
\nu^{\overline{F}}_{1j \parallel} (x)=0  \right\}.
$$
Here, the vector $\nu^{\overline{F}}_{1j \parallel}$ is defined in the following way. 
Let $j \in \{ 2, \ldots, n \}$, and let $\nu^{\overline{F}}_{1j}$
be the orthogonal projection of $\nu^{\overline{F}}$ on the 
bi-dimensional plane generated by $e_1$ and $e_j$.
In this plane, we consider the following orthonormal basis $\{  \widehat{x}_{1j},  \widetilde{x}_{1j}\}$:
$$
 \widehat{x}_{1j} = \frac{1}{\sqrt{x_1^2 + x_j^2}}
 (x_1,  \overbrace{ 0, \ldots, 0}^{j-2 \textnormal{ times}} , 
  x_j, 
 \overbrace{ 0, \ldots, 0}^{n - j \textnormal{ times}} ), 
 $$
 and 
 $$
 \widetilde{x}_{1j} = \frac{1}{\sqrt{x_1^2 + x_j^2}}
 (- x_j,  \overbrace{ 0, \ldots, 0}^{j-2 \textnormal{ times}} , 
  x_1, 
 \overbrace{ 0, \ldots, 0}^{n - j \textnormal{ times}} ),
  $$ 
where $\widehat{x}_{1j}$ is directed along the radial direction, and 
$\widetilde{x}_{1j}$ is parallel to the tangential direction.
To show the claim, first of all note that, by Vol'pert Theorem~\ref{thm:volpert}, 
for $\mathcal{H}^1$-a.e. $r > 0$ we have
$$
(B^j)_r = \left\{x\in \partial^* \overline{F}_r \cap \Phi(I\times \mathbb{S}^{n-1}): 
\nu^{\overline{F}_r}_{\parallel} (x) \cdot \widetilde{x}_{1j} =0  \right\}.
$$
up to an $\mathcal{H}^{n-2}$-negligible set.
Since $(B^j)_r$ is a spherical cap, we have $\mathcal{H}^{n-2} ((B^j)_r) = 0$.
Then, thanks to \eqref{assumption on Fbar}, 
\begin{align*}
\mathcal{H}^{n-1} (B^j) 
&= \mathcal{H}^{n-1} \left( B^j \cap \left\{x\in \partial^* \overline{F} \cap \Phi(I\times \mathbb{S}^{n-1}): 
\nu^{\overline{F}}_{\parallel} (x) \neq 0  \right\} \right)  \\
&= \int_{I} dr \int_{\partial^* \overline{F}_r \cap (B^j)_r} 
\chi_{ \{ \nu^{\overline{F}}_{\parallel} \neq 0 \} } (x) \frac{1}{|\nu^{\overline{F}}_{\parallel} (x)|} 
\, d \mathcal{H}^{n-2} (x) = 0.
\end{align*}

\vspace{.2cm}

\noindent
\textbf{Step 2b:} We conclude. 
Let $E^1:= E$, and let $E^2$ be set obtained by applying to $E$ 
the circular symmetrisation with respect to $(e_1, e_2)$.
Then, for $j = 3, \ldots, n$, we define iteratively the set $E^j$ 
as the circular symmetral of $E^{j-1}$ with respect to $(e_1, e_j)$.
Note that, since $\mathcal{H}^1$-a.e. spherical section of $E$
is a spherical cap, we have $E^n= F_v$.
Therefore, thanks to the perimeter inequality \eqref{per ineq cyl}
under circular symmetrisation  (see Theorem~\ref{theorem ineq cylindrical}),  
we have 
$$
P(F_v;\Phi(I\times \mathbb{S}^{n-1}))
= P(E^{n-1};\Phi(I\times \mathbb{S}^{n-1}))
= \ldots
= P(E;\Phi(I\times \mathbb{S}^{n-1})).
$$
Moreover, for $j = 3, \ldots, n$, we define $
I_j:=\Phi(I\times \mathbb{S}^{n-1}) \cap \{x_j=0\} \cap \{x_1>0\}$. It is not difficult to check that $$\Phi(I\times \mathbb{S}^{n-1})=\Phi_{1j}(I_j\times \mathbb{S}^{1})\quad \textit{for $j = 3, \ldots, n$}.$$ 
Then, applying Lemma~\ref{lem:5.4dominik cyl} to $F_v$ and $E^{n-1}$, we obtain that
$$
\mathcal{H}^{n-1} \left(\left\{x\in \partial^* E^{n-1} \cap \Phi_{1\,n-1}(I_{n-1}\times \mathbb{S}^{1}): 
\nu^{E^{n-1}}_{1 (n-1) \parallel} (x)=0  \right\} \right) = 0,
$$
which, in turns, implies
$$
\mathcal{H}^{n-1} \left(\left\{x\in \partial^* E^{n-1} \cap \Phi_{1\,n-1}(I_{n-1}\times \mathbb{S}^{1}): 
\nu^{E^{n-1}}_{\parallel} (x)=0  \right\} \right) = 0.
$$
Applying iteratively this argument to $E^{n-2}, \ldots, E$, we conclude.
\end{proof}

\section{Proof of Theorem~\ref{rigidity theorem}: (ii) $\Longrightarrow$ (i)} \label{section ii implies i}

Before giving the proof of the implication (ii) $\Longrightarrow$ (i)
of Theorem~\ref{rigidity theorem}, it will be convenient to introduce
some useful notation. Let $v$ and $\mathcal{I} = \{ 0 < \alpha^{\wedge}_v \leq \alpha^{\vee}_v < \pi \}$
be as in the statement of Theorem~\ref{rigidity theorem}. 
By assumption, $\mathcal{I}$ is an interval and 
$\alpha_v \in W^{1, 1}_{\textnormal{loc}} (I)$ where, to ease the notation, 
we set $I:= \mathcal{{\mathring I}}$.
Let now $E$ be a spherically $v$-distributed set of finite perimeter.
We define the \textit{average direction of $E$} as the map
$d_E: I \to \mathbb{S}^{n-1}$ given by 
\begin{equation} \label{def dE}
d_E (r)
:= 
\begin{cases}
\displaystyle \frac{1}{ \omega_{n-1} (\sin \alpha_v (r))^{n-1}r^{n-1}}
\int_{E_r} \hat{x} \, d \mathcal{H}^{n-1} (x), & \text{ if } r \in I \cap G_E, \\
e_1 & \text{ otherwise in }I,
\end{cases}
\end{equation}
where $G_E \subset (0, \infty)$ is the set given by Theorem~\ref{thm:volpert}.
To ease our calculations, it will also be convenient to introduce the 
\textit{barycentre function $b_E: I \to \mathbb{R}^{n}$ of $E$} as 
$$
b_E (r)
:= 
\begin{cases}
\displaystyle \frac{1}{r^{n-1}}
\int_{E_r} \hat{x} \, d \mathcal{H}^{n-1} (x), & \text{ if } r \in I \cap G_E, \\
e_1 & \text{ otherwise in }I.
\end{cases}
$$ 
The importance of the functions $d_E$ and $b_E$ is given by the following lemma.
\begin{lemma} \label{lemma about dE}
Let $v$ be as in Theorem~\ref{rigidity theorem},
let $I \subset (0, \infty)$ be an open interval, 
and let $E$ be a spherically $v$-distributed set of finite perimeter
such that $E_r$ is $\mathcal{H}^{n-1}$-equivalent to a spherical cap
for $\mathcal{H}^1$-a.e. $r \in I$. Then, 
$$
E \cap \Phi (I \times \mathbb{S}^{n-1}) =_{\mathcal{H}^n} \{ x \in \Phi (I \times \mathbb{S}^{n-1}) :  
\textnormal{dist}_{\mathbb{S}^{n-1}} (\hat{x}, d_E (|x|)) < \alpha_v (|x|)  \}.
$$
Moreover, 
\begin{equation} \label{identity with d_E and b_E}
b_E (r) = \omega_{n-1} (\sin \alpha_v (r))^{n-1} d_E (r) 
\qquad \textnormal{ for $\mathcal{H}^1$-a.e. } r \in I.
\end{equation}
\end{lemma}

\begin{proof}
Let us immediately observe that (\ref{identity with d_E and b_E}) follows by the definitions of $d_E$ and $b_E$. By assumption, for $\mathcal{H}^1$-a.e. $r \in I$, there exists $\omega (r) \in \mathbb{S}^{n-1}$
such that $E_r = \mathbf{B}_{\alpha_v (r)} (r \omega (r))$.
We are left to show that 
\begin{equation} \label{identity with d_E}
\omega (r) = d_E (r) \qquad \text{ for $\mathcal{H}^1$-a.e. } r \in I.
\end{equation}
Note that for $\mathcal{H}^1$-a.e. $r \in I$ we have $ E_r = \mathbf{B}_{\alpha_v (r)} (r \omega (r))$ and 
$\partial^* E_r = \mathbf{S}_{\alpha_v (r)} (r \omega (r))$. Therefore, for $\mathcal{H}^1$-a.e. $r \in I$
\begin{align}\label{a useful calculation 0}
\int_{E_r}\hat{x} \, d\mathcal{H}^{n-1}(x)
= \int_0^{\alpha_v(r)}\,d\beta\int_{\mathbf{S}_{\beta} (r \omega (r))}x\, d \mathcal{H}^{n-2} (x).
\end{align}
Observe now that, thanks to the symmetry of the geodesic sphere
and recalling \eqref{measure geodesic sphere}, for every $\beta \in (0, \alpha_v(r))$ we have
\begin{align}\label{a useful calculation}
&\int_{\mathbf{S}_{\beta} (r \omega (r))} x \, d \mathcal{H}^{n-2} (x)
=  \left( \int_{\mathbf{S}_{\beta} (r \omega (r))} (x \cdot \omega (r)) \, d \mathcal{H}^{n-2} (x) \right) \omega (r)  \\
&=  r \cos \beta \, \mathcal{H}^{n-2} (\mathbf{S}_{\beta} (r \omega (r))) \, \omega (r)
= (n-1) \, \omega_{n-1} r^{n-1} \cos \beta \,  
(\sin \beta)^{n-2}  \, \omega (r) . \nonumber
\end{align}
Combining \eqref{a useful calculation 0} and \eqref{a useful calculation}
we obtain that for $\mathcal{H}^1$-a.e. $r \in I$
\begin{align*}
\int_{E_r}\hat{x} \, d\mathcal{H}^{n-1}(x)
&=  (n-1) \, \omega_{n-1} r^{n-1} \left( \int_0^{\alpha_v(r)} \cos \beta \,  
(\sin \beta)^{n-2} \,d\beta \right) \omega (r) \\
&= \omega_{n-1} r^{n-1} (\sin \alpha_v (r))^{n-1} \omega (r).
\end{align*}

Recalling the definition of $d_E$, identity \eqref{identity with d_E} follows.

%
\end{proof}

\begin{remark}
Let us point out that here we are using the term barycentre
in a slightly imprecise way.
Indeed, for a given $r \in I \cap G_E$, the geometric barycentre 
of $E_r$ is given by 
\begin{align*}
&\frac{1}{\mathcal{H}^{n-1} (E_r)}
\int_{E_r} x \, d \mathcal{H}^{n-1} (x)
= \frac{1}{\xi_v (r) r^{n-1}}
\int_{E_r} x \, d \mathcal{H}^{n-1} (x) \\
&= \frac{r}{\xi_v (r)} \frac{1}{r^{n-1} }
\int_{E_r} \hat{x} \, d \mathcal{H}^{n-1} (x) 
= \frac{r}{\xi_v (r)} b_E (r).
\end{align*} 
Nevertheless, we will still keep this terminology, since 
$b_E$ turns out to be very useful for our analysis.  
\end{remark}

We are now ready to prove the implication (ii) $\Longrightarrow$ (i) 
of Theorem~\ref{rigidity theorem}.
%
%
%
%
%
%
%

\begin{proof}[Proof of Theorem~\ref{rigidity theorem}: (ii) $\Longrightarrow$ (i)]
Suppose (ii) is satisfied, and let $E \in \mathcal{N} (v)$.
We are going to show that there exists an orthogonal transformation $R \in SO (n)$
such that $\mathcal{H}^n(E \Delta (R F_v)) = 0$.
We now divide the proof into steps.

\vspace{.2cm}

\noindent
\textbf{Step 1:} First of all, we observe that
$$
\mathcal{H}^{n-1}\left(\Big\{x\in \partial^* E 
\cap \Phi(I\times \mathbb{S}^{n-1}): \nu^{E}_{\|}(x)=0  \Big\}  \right)=0.
$$
Indeed, since $\alpha_v \in W^{1,1}_{\textnormal{loc}} (I)$, thanks to 
Proposition~\ref{prop:5.3dominik} we have 
$$
\mathcal{H}^{n-1}\left(\Big\{x\in \partial^* F_v 
\cap \Phi(I\times \mathbb{S}^{n-1}): \nu^{F_v}_{\|}(x)=0  \Big\}  \right)=0.
$$
Since $E \in \mathcal{N} (v)$, applying Lemma~\ref{lem:5.4dominik}
the claim follows.

\vspace{.2cm}

\noindent
\textbf{Step 2:} We show that $b_E \in W^{1, 1}_{\textnormal{loc}} (I;\R^n)$
and 
\begin{equation} \label{first version of b'}
b'_E (r) = \frac{1}{r^n} 
 \int_{(\partial^* E)_r \cap \{ \nu^E_{\parallel } \neq 0\}} x
\frac{\hat{x} \cdot \nu^E (x)}{|\nu^E_{\parallel} (x)|} \,  
d \mathcal{H}^{n-2} (x).
\end{equation}
Indeed, let $\psi\in C^1_c(I)$ be arbitrary, and let $i \in \{ 1, \ldots, n\}$. 
By definition of $b_E$
\begin{align*}
\int_{I}(b_E)_i (r)\psi'(r)dr&= 
\int_{I}\int_{E\cap\partial B(r)}\frac{1}{r^{n-1}}\frac{x_i}{|x|}d\mathcal{H}^{n-1}(x)\psi'(r)dr \\
&=\int_{\Phi(I\times \mathbb{S}^{n-1})}\frac{x_i}{|x|^{n}} \psi'(| x |) \chi_{E}(x) \, dx.
\end{align*}
Note now that 
$$
\textnormal{div} \left( \frac{x_i}{|x|^n}  \psi (|x|) \hat{x} \right)
= \frac{x_i}{|x|^n} \psi'(| x |).
$$
Indeed, recalling \eqref{formula div},
\begin{align*}
&\textnormal{div} \left( \frac{x_i}{|x|^n}  \psi (|x|) \hat{x} \right)
= \psi (|x|) \nabla \left( \frac{x_i}{|x|^n} \right) \cdot \hat{x}
+ \frac{x_i}{|x|^n}  \, \textnormal{div} (  \psi (|x|) \hat{x} ) \\
&= \psi (|x|)  \left( \frac{e_i}{|x|^n} - \frac{n \, x_i}{|x|^{n+1}} \hat{x}  \right) \cdot \hat{x}
+ \frac{x_i}{|x|^n}  \,  \left(  \psi' (|x|) + \psi (|x|) \frac{n-1}{|x|} \right)
= \frac{x_i}{|x|^n} \psi'(| x |).
\end{align*}
Therefore,
\begin{align*}
&\int_{I}(b_E)_i (r)\psi'(r)dr
= \int_{\Phi(I\times \mathbb{S}^{n-1})} 
\textnormal{div} \left( \frac{x_i}{|x|^n}  \psi (|x|) \hat{x} \right) \chi_{E}(x) \, dx \\
&= - \int_{\Phi(I\times \mathbb{S}^{n-1})} 
 \frac{x_i}{|x|^n}  \psi (|x|) \hat{x} \cdot  d D \chi_{E}(x) \\
& =  \int_{\partial^* E \cap \Phi(I\times \mathbb{S}^{n-1})}  
 \frac{x_i}{|x|^n}  \psi (|x|) \, \hat{x} \cdot \nu^E (x) d \mathcal{H}^{n-1} (x).
 \end{align*}
Thanks to Step 1 we then obtain 
 \begin{align*}
 &\int_{I}(b_E)_i (r)\psi'(r)dr
 =  \int_{\partial^* E \cap \{ \nu^E_{\parallel } \neq 0\} \cap \Phi(I\times \mathbb{S}^{n-1})}  
 \frac{x_i}{|x|^n}  \psi (|x|) \, \hat{x} \cdot \nu^E (x) d \mathcal{H}^{n-1} (x) \\
 &= 
\int_{I} \psi (r) \frac{1}{r^n} 
\left[  \int_{(\partial^* E)_r \cap \{ \nu^E_{\parallel } \neq 0\}} x_i 
\frac{\hat{x} \cdot \nu^E (x)}{|\nu^E_{\parallel} (x)|} \,  
d \mathcal{H}^{n-2} (x) \right] \, dr,
\end{align*}
so that \eqref{first version of b'} follows.

\vspace{.2cm}

\noindent
\textbf{Step 3:} We show that 
\begin{equation} \label{second version of b'}
b'_E (r) = (n-1) \alpha'_v (r) \,\frac{\cos \alpha_v (r)}{\sin \alpha_v (r)} \, b_E (r)
\qquad \text{ for $\mathcal{H}^1$-a.e. } r \in I.
\end{equation}
Since $E \in \mathcal{N} (v)$, from  Theorem~\ref{fv locally finite perimeter}  
we know that for $\mathcal{H}^1$-a.e. $r \in I$
the spherical slice $E_r$ is a spherical cap.
Then, thanks to Lemma~\ref{lemma about dE}
$$
E_r = \mathbf{B}_{\alpha_v (r)} (r d_E (r)) 
\quad \text{ and }\quad 
(\partial^* E )_r = \mathbf{S}_{\alpha_v (r)} (r d_E (r))
\qquad \text{ for $\mathcal{H}^1$-a.e. } r \in I.
$$
Still thanks to  Theorem~\ref{fv locally finite perimeter},  we know that 
for $\mathcal{H}^1$-a.e. $r \in I$
the functions $x \mapsto \nu^E (x) \cdot \hat{x}$ and $x \mapsto | \nu^E_{\parallel}| (x)$
are constant $\mathcal{H}^{n-2}$-a.e. in $(\partial^* E )_r$, say
$$
\nu^E (x) \cdot \hat{x} = a (r) \quad \text{ and } \quad 
| \nu^E_{\parallel}| (x) = c (r), \qquad \text{ for $\mathcal{H}^1$-a.e. }r \in I,
$$
for some measurable functions $a: I \to (-1, 1)$ and $c: I \to (0, 1]$.
 Therefore, recalling the definition of $d_E$ together with (\ref{a useful calculation 0})-(\ref{a useful calculation}) we obtain
\begin{align}
b'_E (r) &= \frac{1}{r^n} 
 \int_{(\partial^* E)_r \cap \{ \nu^E_{\parallel } \neq 0\}} x
\frac{\hat{x} \cdot \nu^E (x)}{|\nu^E_{\parallel} (x)|} \,  
d \mathcal{H}^{n-2} (x) \nonumber \\
&= \frac{1}{r^n} \frac{a (r)}{c (r)} 
 \int_{\mathbf{S}_{\alpha_v (r)} (r d_E (r))} x \,  
d \mathcal{H}^{n-2} (x) \nonumber \\
&= \frac{1}{r^n} \frac{a (r)}{c (r)} 
r \cos ( \alpha_v (r))  \mathcal{H}^{n-2} (\mathbf{S}_{\alpha_v (r)} (r d_E (r))) d_E (r) \nonumber \\
&= \frac{1}{r^{n-1}} \frac{a (r)}{c (r)} \mathcal{H}^{n-2} (\mathbf{S}_{\alpha_v (r)} (r d_E (r)))
\cos ( \alpha_v (r)) d_E (r).
\label{intermediate we are almost there}
\end{align}
Note now that from Step 1 and \eqref{D xi} it follows that 
for $\mathcal{H}^1$-a.e. $r \in I$
\begin{align*}
r^{n-1} \xi'_v (r) &= \int_{(\partial^* E)_r \cap \{\nu^{E}_{\|} \neq 0 \} } 
\frac{\hat{x} \cdot \nu^E (x)}{|\nu_{\|}^E(x)|}d\mathcal{H}^{n-2}(x) \\
&= \frac{a (r)}{c (r)}  \mathcal{H}^{n-2} (\mathbf{S}_{\alpha_v (r)} (r d_E (r))).
\end{align*}
Plugging last identity into \eqref{intermediate we are almost there}
and using \eqref{identity with d_E and b_E}, we obtain 
\begin{align*}
b'_E (r) &= \xi'_v (r) \cos ( \alpha_v (r)) d_E (r)
= \xi'_v (r) \cos ( \alpha_v (r)) \frac{b_E (r)}{\omega_{n-1} (\sin \alpha_v (r))^{n-1}} \\
&= (n-1) \alpha'_v (r) \,\frac{\cos \alpha_v (r)}{\sin \alpha_v (r)} \, b_E (r),
\end{align*}
where we used the fact that, thanks to \eqref{measure geodesic ball} and \eqref{xi def},
\begin{equation*} 
\xi'_v (r) = (n-1) \omega_{n-1} (\sin \alpha_v (r))^{n-2} \alpha'_v (r)
\qquad \text{ for $\mathcal{H}^1$-a.e. } r \in I.
\end{equation*}

\vspace{.2cm}

\noindent
\textbf{Step 4:} We conclude.
First of all, note that from \eqref{identity with d_E and b_E} and Step 2
it follows that $d_E \in W^{1, 1}_{\text{loc}} (I;\mathbb{S}^{n-1})$.
Then, thanks to Step 3, for $\mathcal{H}^1$-a.e. $r \in I$
\begin{align*}
&\omega_{n-1} d_E' (r) = \frac{d}{d r} \left[ \frac{b_E (r)}{ (\sin \alpha_v (r))^{n-1}} \right]
=   \frac{b'_E (r)}{ (\sin \alpha_v (r))^{n-1}} 
+ b_E (r) \frac{d}{d r} \left[ \frac{1}{(\sin \alpha_v (r))^{n-1}} \right] \\
&= (n-1) \alpha'_v (r) \,  \frac{\cos \alpha_v (r)}{ (\sin \alpha_v (r))^{n}} \, b_E (r)
+ b_E (r)  \left[ - \frac{n-1}{(\sin \alpha_v (r))^{n}} (\cos \alpha_v (r)) \alpha'_v (r)  \right]  = 0,
\end{align*}
for $\mathcal{H}^1$-a.e. $r \in I$.
This shows that $d_E$ is $\mathcal{H}^1$-a.e. constant in $I$.
Therefore, $E \cap \Phi (I \times \mathbb{S}^{n-1})$
can be obtained by applying an orthogonal transformation to $F_v \cap \Phi (I \times \mathbb{S}^{n-1})$.
\end{proof}

\section{Proof of Theorem~\ref{rigidity theorem}: (i) $\Longrightarrow$ (ii)} \label{section i implies ii}

We start by showing that the fact that $\{ 0 < \alpha^{\wedge} \leq \alpha^{\vee} < \pi \}$
is an interval is a necessary condition for rigidity.
\begin{proposition} \label{no interval}
Let $v: (0,\infty) \to [0, \infty)$ be a measurable function 
satisfying \eqref{bound on v}, such that $F_v$ 
is a set of finite perimeter and finite volume, 
and let $\alpha_v$ be defined by \eqref{this is the def of alphav}.
Suppose that the set $\{ 0 < \alpha^{\wedge} \leq \alpha^{\vee} < \pi \}$
is not an interval. 
That is, suppose that there exists $\overline{r} \in \{ \alpha^{\wedge} = 0 \} \cup \{ \alpha^{\vee} = \pi\}$
such that
$$
(0, \overline{r}) \cap \{ 0 < \alpha^{\wedge} \leq \alpha^{\vee} < \pi\} \neq \emptyset
\qquad \text{ and } \qquad 
(\overline{r}, \infty) \cap \{ 0 < \alpha^{\wedge} \leq \alpha^{\vee} < \pi\} \neq \emptyset.
$$
Then, rigidity fails. More precisely, setting $E_1 := F_v \cap B (\overline{r}) $
and $E_2 :=  F_v \setminus B (\overline{r})$, we have 
$$
E_1 \cup  (R E_2) \in \mathcal{N} (v) \qquad \text{ for every } R \in O(n).
$$ 
\end{proposition}
Before giving the proof of Proposition~\ref{no interval}
we need the following lemma.
\begin{lemma}  \label{slice of reduced boundary F}
Let $v: (0,\infty) \to [0, \infty)$ be a measurable function 
satisfying \eqref{bound on v}, such that $F_v$ 
is a set of finite perimeter and finite volume. 
Let $\alpha_v$ be defined by \eqref{this is the def of alphav}, 
and let $\overline{r} > 0$. 
Then, 
$$
(\partial^* F_v)_{\overline{r}} 
=_{\mathcal{H}^{n-1}} 
\mathbf{B}_{\alpha^{\vee}_v (\overline{r})} (\overline{r} e_1)
\setminus  \mathbf{B}_{\alpha^{\wedge}_v (\overline{r})} (\overline{r} e_1).
$$
\end{lemma}

\begin{proof}
We divide the proof in two steps.

\vspace{.2cm}

\noindent
\textbf{Step 1:} We show that 
$$
(\partial^* F_v)_{\overline{r}} 
\subset \overline{\mathbf{B}_{\alpha^{\vee}_v (\overline{r})} (\overline{r} e_1)}
\setminus  \mathbf{B}_{\alpha^{\wedge}_v (\overline{r})} (\overline{r} e_1).
$$ 
To this aim, it will be enough to show that 
\begin{equation} \label{reduced boundary at r bar}
\alpha^{\wedge}_v (\overline{r}) 
\leq \text{dist}_{\mathbb{S}^{n-1}} (\hat{x}, e_1) \leq \alpha^{\vee}_v (\overline{r})
\qquad \text{ for every } x \in (\partial^* F_v)_{\overline{r}}.
\end{equation}
Let us first prove that
\begin{equation} \label{reduced boundary at r bar 2}
\text{dist}_{\mathbb{S}^{n-1}} (\hat{x}, e_1) \leq \alpha^{\vee}_v (\overline{r})
\qquad \text{ for every } x \in (\partial^* F_v)_{\overline{r}}
\end{equation}
Note that \eqref{reduced boundary at r bar 2} is trivial if 
$\alpha^{\vee}_v (\overline{r}) = \pi$.
For this reason, we will assume $\alpha^{\vee}_v (\overline{r}) < \pi$.
Note now that \eqref{reduced boundary at r bar 2} follows if we prove that
\begin{equation} \label{claim 1}
x \in \partial B (\overline{r}) \quad \text{ and } \quad 
\text{dist}_{\mathbb{S}^{n-1}} (\hat{x}, e_1) > \alpha^{\vee}_v (\overline{r}) 
\quad \Longrightarrow \quad x \in F_v^{(0)}.
\end{equation}
Let now $x \in \partial B (\overline{r})$, and suppose that there exists $\delta > 0$ such that
$$
 \text{dist}_{\mathbb{S}^{n-1}} (\hat{x}, e_1) 
 =  \alpha^{\vee}_v (\overline{r}) + \delta.
$$
Let now $\overline{\rho} > 0$ be so small that 
$$
 \text{dist}_{\mathbb{S}^{n-1}} (\hat{y}, \hat{x}) < \frac{\delta}{2} \qquad \text{ for every }
 y \in  B (x, \overline{\rho}).
$$
By triangle inequality for the geodesic distance we have, in particular, that 
$$
\alpha^{\vee}_v (\overline{r}) + \delta
= \text{dist}_{\mathbb{S}^{n-1}} (\hat{x}, e_1) 
\leq \text{dist}_{\mathbb{S}^{n-1}} (\hat{x}, \hat{y})
+  \text{dist}_{\mathbb{S}^{n-1}} (\hat{y}, e_1)
< \frac{\delta}{2} +  \text{dist}_{\mathbb{S}^{n-1}} (\hat{y}, e_1),  
$$ 
so that 
\begin{equation} \label{very small}
\text{dist}_{\mathbb{S}^{n-1}} (\hat{y}, e_1) > \alpha^{\vee}_v (\overline{r}) + \frac{\delta}{2}
\qquad \text{ for every }
 y \in  B (x, \overline{\rho}).
\end{equation}
Thanks to the inequality above, by definition of $F_v$ we have
$$
F_v \cap B (x, \overline{\rho}) 
\subset \left\{ y \in \R^n :  
\alpha^{\vee}_v (\overline{r}) + \frac{\delta}{2} <
\text{dist}_{\mathbb{S}^{n-1}} (\hat{y}, e_1)
< \alpha_v (|y|)  \right\} \cap B (x, \overline{\rho}).
$$  
Therefore, for every $\rho \in (0, \overline{\rho})$
\begin{align*}
&\mathcal{H}^{n} (F_v \cap B (x, \rho))
= \int_{\overline{r}- \rho}^{\overline{r} + \rho} \mathcal{H}^{n-1}
(F_v \cap B (x, \rho) \cap \partial B (r)) \, dr \\
&\leq \int_{\overline{r}- \rho}^{\overline{r} + \rho} 
\chi_{\{ \alpha_v > \alpha^{\vee}_v (\overline{r}) + \delta/2 \}} (r) 
\mathcal{H}^{n-1}
(F_v \cap B (x, \rho) \cap \partial B (r)) \, dr \\
&= \int_{(\overline{r}- \rho, \overline{r} + \rho) 
\cap \{ \alpha_v > \alpha^{\vee}_v (\overline{r}) + \delta/2 \} } 
\mathcal{H}^{n-1}
(F_v \cap B (x, \rho) \cap \partial B (r)) \, dr.
\end{align*}
Note now that, for $\rho$ small enough, there exists $C = C (\overline{r}) > 0$
such that 
\begin{equation*} 
B (x, \rho) \cap \partial B (r)
\subset \mathbf{B}_{C \rho} (r \hat{x}) \qquad \text{ for every } r \in (\overline{r} - \rho, \overline{r} + \rho).
\end{equation*} 
Therefore, 
\begin{align*}
&\mathcal{H}^{n} (F_v \cap B (x, \rho))
\leq \int_{(\overline{r}- \rho, \overline{r} + \rho) 
\cap \{ \alpha_v > \alpha^{\vee}_v (\overline{r}) + \delta/2 \} }
\mathcal{H}^{n-1} (\mathbf{B}_{C \rho} (r \hat{x})) \, dr \\
&= (n-1) \omega_{n-1}  \int_{(\overline{r}- \rho, \overline{r} + \rho) 
\cap \{ \alpha_v > \alpha^{\vee}_v (\overline{r}) + \delta/2 \} }
r^{n-1}
\int_0^{C \rho} (\sin \tau)^{n-2} \, d \tau \, dr \\
&\leq (n-1) \omega_{n-1}  \int_{(\overline{r}- \rho, \overline{r} + \rho) 
\cap \{ \alpha_v > \alpha^{\vee}_v (\overline{r}) + \delta/2 \} }
r^{n-1} \int_0^{C \rho} \tau^{n-2} \, d \tau \, dr \\
&=  \omega_{n-1} C^{n-1} (\overline{r} + \overline{\rho})^{n-1} \rho^{n-1}
\mathcal{H}^1 ((\overline{r}- \rho, \overline{r} + \rho) 
\cap \{ \alpha_v > \alpha^{\vee}_v (\overline{r}) + \delta/2 \}).
\end{align*}
Thus, recalling the definition of $\alpha^{\vee}_v (\overline{r})$, 
\begin{align*}
&\lim_{\rho \to 0^+} \frac{\mathcal{H}^{n} (F_v \cap B (x, \rho))}{\omega_n \rho^n} \\
&\leq \frac{\omega_{n-1} C^{n-1}}{\omega_n} (\overline{r} + \overline{\rho})^{n-1} 
\lim_{\rho \to 0^+} \frac{\mathcal{H}^1 ((\overline{r}- \rho, \overline{r} + \rho) 
\cap \{ \alpha_v > \alpha^{\vee}_v (\overline{r}) + \delta/2 \})}{ \rho} = 0, 
\end{align*}
which gives \eqref{claim 1} and, in turn, \eqref{reduced boundary at r bar 2}.
By similar arguments, one can prove that 
$$
x \in \partial B (\overline{r}) \quad \text{ and } \quad 
\text{dist}_{\mathbb{S}^{n-1}} (\hat{x}, e_1) < \alpha^{\wedge}_v (\overline{r}) 
\quad \Longrightarrow \quad x \in F_v^{(1)},
$$
which implies that 
\begin{equation*} 
\alpha^{\wedge}_v (\overline{r}) 
\leq \text{dist}_{\mathbb{S}^{n-1}} (\hat{x}, e_1) 
\qquad \text{ for every } x \in (\partial^* F_v)_{\overline{r}}.
\end{equation*}
The above inequality, together with \eqref{reduced boundary at r bar 2}, 
shows \eqref{reduced boundary at r bar}.

\vspace{.2cm}

\noindent
\textbf{Step 2:} We conclude.
Thanks to Corollary \ref{corollary perimeter Fv}, 
\begin{align*}
&\mathcal{H}^{n-1} ((\partial^* F_v)_{\overline{r}})
= \mathcal{H}^{n-1} (\partial^* F_v \cap \partial B (\overline{r}) ) 
=  P(F_v ; \partial B (\overline{r})) =  \overline{r}^{n-1} ( \xi_v^{\vee} (\overline{r}) - \xi_v^{\wedge} (\overline{r})  ) \\
&= v^{\vee} (\overline{r}) - v^{\wedge} (\overline{r})
= \mathcal{H}^{n-1} (\overline{\mathbf{B}_{\alpha^{\vee}_v (\overline{r})} (\overline{r} e_1)} )
- \mathcal{H}^{n-1} (  \mathbf{B}_{\alpha^{\wedge}_v (\overline{r})} (\overline{r} e_1) ) \\
&= \mathcal{H}^{n-1} \left(
 \overline{\mathbf{B}_{\alpha^{\vee}_v (\overline{r})} (\overline{r} e_1)}
\setminus  \mathbf{B}_{\alpha^{\wedge}_v (\overline{r})} (\overline{r} e_1) \right)
\end{align*}
Since, by Step 1, 
$$
(\partial^* F_v)_{\overline{r}} 
\subset \overline{\mathbf{B}_{\alpha^{\vee}_v (\overline{r})} (\overline{r} e_1)}
\setminus  \mathbf{B}_{\alpha^{\wedge}_v (\overline{r})} (\overline{r} e_1), 
$$ 
we have 
$$
(\partial^* F_v)_{\overline{r}} =_{\mathcal{H}^{n-1}} 
\overline{\mathbf{B}_{\alpha^{\vee}_v (\overline{r})} (\overline{r} e_1)}
\setminus  \mathbf{B}_{\alpha^{\wedge}_v (\overline{r})} (\overline{r} e_1)
=_{\mathcal{H}^{n-1}} 
\mathbf{B}_{\alpha^{\vee}_v (\overline{r})} (\overline{r} e_1)
\setminus  \mathbf{B}_{\alpha^{\wedge}_v (\overline{r})} (\overline{r} e_1).
$$
\end{proof}
We can now give the proof of Proposition~\ref{no interval}.
\begin{proof}[Proof of Proposition~\ref{no interval}]
Note that, since $B (\overline{r})$ 
is open and $E \cap B (\overline{r}) = F_v \cap B (\overline{r})$, 
we have 
$$
E^{(t)} \cap B (\overline{r}) 
= ( E \cap B (\overline{r}) )^{(t)} 
= ( F_v \cap B (\overline{r}) )^{(t)}
= F_v^{(t)} \cap B (\overline{r}) \quad \text{ for every } t \in [0, 1].
$$
From this, it follows that 
\begin{equation} \label{first term}
\partial^* E \cap B (\overline{r})  = \partial^* F_v \cap B (\overline{r}).
\end{equation}
Similarly, we obtain 
\begin{equation} \label{second term}
\partial^* E \setminus \overline{B (\overline{r})}
= \partial^* (R F_v)   \setminus \overline{B (\overline{r})}
= ( R \, \partial^* F_v)   \setminus 
(R \overline{B (\overline{r})})
= R ( \partial^* F_v  \setminus \overline{B (\overline{r})} ).
\end{equation}
Thus, thanks to \eqref{first term} and \eqref{second term}
\begin{align*}
P (E) &= \mathcal{H}^{n-1} (\partial^* E \cap B (\overline{r}))
+ \mathcal{H}^{n-1} (\partial^* E \cap \partial B (\overline{r}))
+ \mathcal{H}^{n-1} (\partial^* E \setminus \overline{B (\overline{r})}) \\
&= \mathcal{H}^{n-1} (\partial^* F_v \cap B (\overline{r}))
+ \mathcal{H}^{n-1} (\partial^* E \cap \partial B (\overline{r}))
+ \mathcal{H}^{n-1} \left( R ( \partial^* F_v  \setminus \overline{B (\overline{r})} ) \right) \\
&= \mathcal{H}^{n-1} (\partial^* F_v \cap B (\overline{r}))
+ \mathcal{H}^{n-1} (\partial^* E \cap \partial B (\overline{r}))
+ \mathcal{H}^{n-1} ( \partial^* F_v  \setminus \overline{B (\overline{r})} ). 
\end{align*}
Therefore, in order to conclude the proof we only need to show that
\begin{equation} \label{last part now}
\mathcal{H}^{n-1} (\partial^* E \cap B (\overline{r}) )
= \mathcal{H}^{n-1} (\partial^* F_v \cap B (\overline{r}) ).
\end{equation}
Without any loss of generality, we will assume that 
\begin{equation} \label{half lines new}
\alpha_v^\vee( \overline{r})
=\aplim(f,(0, \overline{r}),\overline{r})\,,\qquad 
0 = \alpha_v^\wedge(\overline{r})=\aplim(f,(\overline{r}, \infty),\overline{r})\,.
\end{equation}
Let now $E_1, E_2$, and $R$ be as in the statement.
We divide the proof of \eqref{last part now} into steps. 

\vspace{.2cm}

\noindent
\textbf{Step 1:} We show that 
$$
(\partial^* E)_{\overline{r}} 
\subset \overline{\mathbf{B}_{\alpha^{\vee}_v (\overline{r})} (\overline{r} e_1)}
\cup \{ R (\overline{r} e_1) \}.
$$ 
To this aim, it will be enough to prove that
\begin{equation} \label{inequality 1 for E}
\text{dist}_{\mathbb{S}^{n-1}} (\hat{x}, e_1) \leq \alpha^{\vee}_v (\overline{r})
\qquad \text{ for every } x \in (\partial^* E)_{\overline{r}}.
\end{equation}
If $\alpha^{\vee}_v (\overline{r}) = \pi$ inequality \eqref{inequality 1 for E} is obvious, 
so we will assume that $\alpha^{\vee}_v (\overline{r}) < \pi$.

\vspace{.2cm}

\noindent
\textbf{Step 1a:} We show that 
$$
x \in \partial B (\overline{r}) \quad \text{ and } \quad 
\text{dist}_{\mathbb{S}^{n-1}} (\hat{x}, e_1) > \alpha^{\vee}_v (\overline{r}) 
\quad \Longrightarrow \quad x \in E_1^{(0)}.
$$
Indeed, let $x \in \partial B (\overline{r})$, 
and suppose that there exists $\delta > 0$ such that
$$
 \text{dist}_{\mathbb{S}^{n-1}} (\hat{x}, e_1) 
 =  \alpha^{\vee}_v (\overline{r}) + \delta.
$$
By repeating the argument used to show \eqref{very small}, 
we can choose $\overline{\rho} > 0$ so small that 
\begin{equation*} 
\text{dist}_{\mathbb{S}^{n-1}} (\hat{y}, e_1) > \alpha^{\vee}_v (\overline{r}) + \frac{\delta}{2}
\qquad \text{ for every }
 y \in  B (x, \overline{\rho}).
\end{equation*}
By definition of $E_1$, we then have
\begin{align*}
&E_1 \cap B (x, \overline{\rho}) = F_v \cap B (\overline{r}) \cap B (x, \overline{\rho}) \\
&\subset \left\{ y \in \R^n :  |y| < \overline{r} \text{ and }
\alpha^{\vee}_v (\overline{r}) + \frac{\delta}{2} <
\text{dist}_{\mathbb{S}^{n-1}} (\hat{y}, e_1)
< \alpha_v (|y|)  \right\} \cap B (x, \overline{\rho}).
\end{align*}
Therefore, for every $\rho \in (0, \overline{\rho})$, 
by repeating the calculations done in Step 1 of Lemma~\ref{slice of reduced boundary F}, 
we obtain 
\begin{align*}
&\lim_{\rho \to 0^+} \frac{1}{\omega_n \rho^n} \mathcal{H}^{n} (E_1 \cap B (x, \rho)) \\
&= \lim_{\rho \to 0^+} \frac{1}{\omega_n \rho^n} \int_{\overline{r}- \rho}^{\overline{r}} \mathcal{H}^{n-1}
(F_v \cap B (x, \rho) \cap \partial B (r)) \, dr \\
& \leq \frac{\omega_{n-1} C^{n-1}}{\omega_n} (\overline{r} + \overline{\rho})^{n-1} 
\lim_{\rho \to 0^+} \frac{\mathcal{H}^1 ((\overline{r}- \rho, \overline{r}) 
\cap \{ \alpha_v > \alpha^{\vee}_v (\overline{r}) + \delta/2 \})}{ \rho} = 0,
\end{align*}
where we used \eqref{half lines new}.

\vspace{.2cm}

\noindent
\textbf{Step 1b:} We show that 
$$
\partial B (\overline{r}) \setminus 
\{ R (\overline{r} e_1) \} \subset (R E_2)^{(0)}.
$$
Indeed, let $x \in \partial B (\overline{r})$, and suppose that 
$\eta := \text{dist}_{\mathbb{S}^{n-1}} (\hat{x}, R e_1) > 0$.
We are going to prove that $x \in (R E_2)^{(0)}$.
By repeating the argument used to show \eqref{very small}, 
we can choose $\overline{\rho} > 0$ so small that 
$$
\text{dist}_{\mathbb{S}^{n-1}} (\hat{y}, R e_1) > \frac{\eta}{2}
\qquad \text{ for every } y \in B (x, \overline{\rho}).
$$
Then, 
\begin{align*}
&(R E_2) \cap B (x, \overline{\rho}) 
= \left( R (F_v \setminus \overline{B (\overline{r})}) \right) \cap B (x, \overline{\rho}) \\
&\subset_{\mathcal{H}^{n}} \left\{ y \in \R^n :  |y| > \overline{r} \text{ and }
\frac{\eta}{2} < \text{dist}_{\mathbb{S}^{n-1}} (\hat{y}, R e_1)
< \alpha_v (|y|)  \right\} \cap B (x, \overline{\rho}).
\end{align*}
For $\rho$ small enough, there exists $C = C (\overline{r}) > 0$
such that 
\begin{equation*} 
B (x, \rho) \cap \partial B (r)
\subset \mathbf{B}_{C \rho} (r \hat{x}) \qquad \text{ for every } r \in (\overline{r} - \rho, \overline{r} + \rho).
\end{equation*} 
Therefore, for every $\rho \in (0, \overline{\rho})$,
\begin{align*}
&\mathcal{H}^{n} ((RE_2) \cap B (x, \rho))
\leq \int_{(\overline{r}, \overline{r} + \rho) 
\cap \{ \alpha_v >  \eta/2 \} }
\mathcal{H}^{n-1} (\mathbf{B}_{C \rho} (r \hat{x})) \, dr \\
&= (n-1) \omega_{n-1}  \int_{(\overline{r}, \overline{r} + \rho) 
\cap \{ \alpha_v > \eta/2 \} }
r^{n-1}
\int_0^{C \rho} (\sin \tau)^{n-2} \, d \tau \, dr \\
&=  \omega_{n-1} C^{n-1} (\overline{r} + \overline{\rho})^{n-1} \rho^{n-1}
\mathcal{H}^1 ((\overline{r}, \overline{r} + \rho) 
\cap \{ \alpha_v >  \eta/2 \}).
\end{align*}
From this, thanks to \eqref{half lines new}, we obtain 
\begin{align*}
&\lim_{\rho \to 0^+} \frac{ \mathcal{H}^{n} ((RE_2) \cap B (x, \rho)) }{\omega_n \rho^n} \\
&\leq \frac{\omega_{n-1} C^{n-1}}{\omega_n} (\overline{r} + \overline{\rho})^{n-1} 
\lim_{\rho \to 0^+} \frac{\mathcal{H}^1 ((\overline{r}, \overline{r} + \rho) 
\cap \{ \alpha_v > \eta/2 \})}{ \rho} = 0.
\end{align*}

\vspace{.2cm}

\noindent
\textbf{Step 1c:} We conclude the proof of Step 1.
By definition of $E$, from Step 1a and Step 1b it follows that 
$$
\{ x \in \partial B (\overline{r}) :  
\text{dist}_{\mathbb{S}^{n-1}} (\hat{x}, e_1) > \alpha^{\vee}_v (\overline{r}) \}
\setminus \{ R e_1 \} \subset  E_1^{(0)} \cap (R E_2)^{(0)} = E^{(0)}.
$$
Therefore, 
\begin{align*}
&( \partial^* E )_r 
\subset \partial B (\overline{r}) \setminus 
\left( \{ x \in \partial B (\overline{r}) :  
\text{dist}_{\mathbb{S}^{n-1}} (\hat{x}, e_1) > \alpha^{\vee}_v (\overline{r}) \}
\setminus \{ R e_1 \} \right) \\
& = \overline{\mathbf{B}_{\alpha^{\vee}_v (\overline{r})} (\overline{r} e_1)}
\cup \{ R e_1 \}.
\end{align*}


\vspace{.2cm}

\noindent
\textbf{Step 2:} We show \eqref{last part now}, concluding the proof.
Thanks to Step 1 and Lemma~\ref{slice of reduced boundary F} we have 
 \begin{align*}
&P(E ; \partial B (\overline{r}))
= \mathcal{H}^{n-1} (\partial^* E \cap \partial B (\overline{r})) 
= \mathcal{H}^{n-1} ((\partial^* E)_{\overline{r}}) 
\leq \mathcal{H}^{n-1} \left(
\mathbf{B}_{\alpha^{\vee}_v (\overline{r})} (\overline{r} e_1) \right) \\
&= \mathcal{H}^{n-1} (\partial^* F_v \cap \partial B (\overline{r})) 
= P(F_v ; \partial B (\overline{r})) \leq P(E ; \partial B (\overline{r})), 
\end{align*}
where we also used \eqref{per ineq} with $B = \{ \overline{r} \}$. 
\end{proof}
We now show that, if the jump part $D^j \alpha_v$ of $D \alpha_v$ is non zero,  rigidity fails.
\begin{proposition} \label{jump set implies no rigidity}
Let $v: (0,\infty) \to [0, \infty)$ be a measurable function 
satisfying \eqref{bound on v} such that $F_v$ 
is a set of finite perimeter and finite volume, 
and let $\alpha_v$ be defined by \eqref{this is the def of alphav}.
Suppose that $\alpha_v$ has a jump at some point $\overline{r} > 0$.
Then, rigidity fails. More precisely, setting $E_1 := F_v \cap B (\overline{r}) $
and $E_2 :=  F_v \setminus B (\overline{r})$, we have 
$$
E_1 \cup  (R E_2) \in \mathcal{N} (v),
$$ 
for every $R \in O(n)$ such that 
\begin{equation} \label{distance e1 omega}
0 < \textnormal{dist}_{\mathbb{S}^{n-1}} (R e_1, e_1) 
< \lambda ( \alpha^{\vee}_v  (\overline{r}) - \alpha^{\wedge}_v  (\overline{r}) ) \quad 
\text{ for some } \lambda \in (0,1).
\end{equation}
\end{proposition}

\begin{proof}
Let $R \in O(n)$, $\lambda \in (0,1)$, and $E \in \R^n$ 
be as in the statement, and set $\omega:= R e_1$.
Arguing as in the proof of Proposition~\ref{no interval} 
we have: 
\begin{align*}
P (E) = \mathcal{H}^{n-1} (\partial^* F_v \cap B (\overline{r}))
+ \mathcal{H}^{n-1} (\partial^* E \cap \partial B (\overline{r}))
+ \mathcal{H}^{n-1} ( \partial^* F_v  \setminus \overline{B (\overline{r})} ). 
\end{align*}
%
%
%
Therefore, in order to conclude the proof we only need to show that
\begin{equation} \label{last part}
\mathcal{H}^{n-1} (\partial^* E \cap \partial B (\overline{r}))
= \mathcal{H}^{n-1} (\partial^* F_v \cap \partial B (\overline{r})).
\end{equation}
Without any loss of generality, we will assume that 
\begin{equation} \label{half lines}
\alpha_v^\vee( \overline{r})
=\aplim(f,(0, \overline{r}),\overline{r})\,,\qquad 
\alpha_v^\wedge(\overline{r})=\aplim(f,(\overline{r}, \infty),\overline{r})\,.
\end{equation}
We now proceed by steps.

\vspace{.2cm}

\noindent
\textbf{Step 1:} We show that
\begin{equation} \label{let us show this}
(\partial^* E)_{\overline{r}} 
\subset \overline{\mathbf{B}_{\alpha^{\vee}_v (\overline{r})} (\overline{r} e_1)}
\setminus  \mathbf{B}_{\alpha^{\wedge}_v (\overline{r})} (\overline{r} \omega).
\end{equation}
To show \eqref{let us show this}, it is enough to prove that
for every $x \in (\partial^* E)_{\overline{r}}$ we have
\begin{equation} \label{inequality 1 for E new}
\text{dist}_{\mathbb{S}^{n-1}} (\hat{x}, e_1) \leq \alpha^{\vee}_v (\overline{r})
\qquad \text{ for every } x \in (\partial^* E)_{\overline{r}}, 
\end{equation}
and 
\begin{equation} \label{inequality 2 for E}
 \text{dist}_{\mathbb{S}^{n-1}} (\hat{x}, \omega)
\geq \alpha^{\wedge}_v (\overline{r}) \qquad \text{ for every } x \in (\partial^* E)_{\overline{r}}.
\end{equation}
We will only show \eqref{inequality 1 for E new}, since \eqref{inequality 2 for E}
can be obtained in a similar way.
Note that \eqref{inequality 1 for E new} is automatically satisfied 
if $\alpha^{\vee}_v (\overline{r}) = \pi$, so we will assume 
$\alpha^{\vee}_v (\overline{r}) < \pi$.

By arguing as in Step 1a of the proof of Proposition~\ref{no interval}
we obtain 
\begin{equation} \label{first useful implication}
x \in \partial B (\overline{r}) \quad \text{ and } \quad 
\text{dist}_{\mathbb{S}^{n-1}} (\hat{x}, e_1) > \alpha^{\vee}_v (\overline{r}) 
\quad \Longrightarrow \quad x \in E_1^{(0)}.
\end{equation}
Let us now prove that 
\begin{equation} \label{second useful implication}
x \in \partial B (\overline{r}) \quad \text{ and } \quad 
\text{dist}_{\mathbb{S}^{n-1}} (\hat{x}, e_1) > \alpha^{\vee}_v (\overline{r}) 
\quad \Longrightarrow \quad x \in (R \, E_2)^{(0)}.
\end{equation}
Let $x \in \partial B (\overline{r})$, and suppose that there exists $\delta > 0$
such that
$$
\text{dist}_{\mathbb{S}^{n-1}} (\hat{x}, e_1) = \alpha^{\vee}_v (\overline{r}) + \delta.
$$
Thanks to the argument we used to show \eqref{very small}, 
we can choose $\overline{\rho} > 0$ so small that 
\begin{equation*} 
\text{dist}_{\mathbb{S}^{n-1}} (\hat{y}, e_1) > \alpha^{\vee}_v (\overline{r}) + \frac{\delta}{2}
\qquad \text{ for every }
 y \in  B (x, \overline{\rho}).
\end{equation*}
Therefore, for every $y \in B(x , \overline{\rho})$ we have 
\begin{align*}
&\alpha^{\vee}_v (\overline{r}) + \frac{\delta}{2}
< \text{dist}_{\mathbb{S}^{n-1}} (\hat{y}, e_1) 
\leq \text{dist}_{\mathbb{S}^{n-1}} (\hat{y}, \omega)
+  \text{dist}_{\mathbb{S}^{n-1}} (\omega, e_1) \\
&< \text{dist}_{\mathbb{S}^{n-1}} (\hat{y}, \omega) 
+  \lambda ( \alpha^{\vee}_v  (\overline{r}) - \alpha^{\wedge}_v  (\overline{r}) ).  
\end{align*} 
Since $\overline{r}$ is a jump point for $\alpha_v$, 
we have $\alpha^{\vee}_v  (\overline{r}) > \alpha^{\wedge}_v  (\overline{r})$, 
and the above inequality implies that 
\begin{align*}
&\text{dist}_{\mathbb{S}^{n-1}} (\hat{y}, \omega)
> (1 - \lambda ) \alpha^{\vee}_v (\overline{r})
+ \lambda \alpha^{\wedge}_v (\overline{r}) + \frac{\delta}{2} 
> (1 - \lambda ) \alpha^{\wedge}_v (\overline{r})
+ \lambda \alpha^{\wedge}_v (\overline{r}) + \frac{\delta}{2}
= \alpha^{\wedge}_v (\overline{r}) + \frac{\delta}{2},
\end{align*}
for every $y \in B(x , \overline{\rho})$.
Then, by definition of $E_2$,
\begin{align*}
&( R E_2 ) \cap B (x, \overline{\rho}) 
= \left( R (F_v \setminus \overline{B( \overline{r} )}) \right) \cap B (x, \overline{\rho}) \\
&\subset_{\mathcal{H}^{n}} \left\{ y \in \R^n :  |y| > \overline{r} \text{ and }
\alpha^{\wedge}_v (\overline{r}) + \frac{\delta}{2} <
\text{dist}_{\mathbb{S}^{n-1}} (\hat{y}, \omega)
< \alpha_v (|y|)  \right\} \cap B (x, \overline{\rho}).
\end{align*}
As already observed in the previous proofs, 
for $\rho$ small enough there exists $C = C (\overline{r}) > 0$
such that 
\begin{equation*} 
B (x, \rho) \cap \partial B (r)
\subset \mathbf{B}_{C \rho} (r \hat{x}) \qquad \text{ for every } r \in (\overline{r} - \rho, \overline{r} + \rho).
\end{equation*} 
Therefore, for every $\rho \in (0, \overline{\rho})$ sufficiently small
\begin{align*}
&\mathcal{H}^{n} ( ( R E_2 ) \cap B (x, \rho))
\leq \int_{(\overline{r}, \overline{r} + \rho) 
\cap \{ \alpha_v > \alpha^{\wedge}_v (\overline{r}) + \delta/2 \} }
\mathcal{H}^{n-1} (\mathbf{B}_{C \rho} (r \hat{x})) \, dr \\
&= (n-1) \omega_{n-1}  \int_{(\overline{r}, \overline{r} + \rho) 
\cap \{ \alpha_v > \alpha^{\wedge}_v (\overline{r}) + \delta/2 \} }
r^{n-1}
\int_0^{C \rho} (\sin \tau)^{n-2} \, d \tau \, dr \\
&=  \omega_{n-1} C^{n-1} (\overline{r} + \overline{\rho})^{n-1} \rho^{n-1}
\mathcal{H}^1 ((\overline{r}, \overline{r} + \rho) 
\cap \{ \alpha_v > \alpha^{\wedge}_v (\overline{r}) + \delta/2 \}).
\end{align*}
From this, thanks to \eqref{half lines}, we obtain 
\begin{align*}
&\lim_{\rho \to 0^+} \frac{\mathcal{H}^{n} ( (R E_2) \cap B (x, \rho))}{\omega_n \rho^n} \\
&\leq \frac{\omega_{n-1} C^{n-1}}{\omega_n} (\overline{r} + \overline{\rho})^{n-1} 
\lim_{\rho \to 0^+} \frac{\mathcal{H}^1 ((\overline{r}, \overline{r} + \rho) 
\cap \{ \alpha_v > \alpha^{\wedge}_v (\overline{r}) + \delta/2 \})}{ \rho} = 0,
\end{align*}
which shows \eqref{second useful implication}.
This, together with \eqref{first useful implication}, implies 
\eqref{inequality 1 for E new}. 
As already mentioned, \eqref{inequality 2 for E} can be proved in a similar way, 
and therefore \eqref{let us show this} follows.
 
\vspace{.2cm}

\noindent
\textbf{Step 2:} We conclude.
From \eqref{distance e1 omega} it follows that 
$$
\mathbf{B}_{\alpha^{\wedge}_v (\overline{r})} (\overline{r} \omega)
\subset \mathbf{B}_{\alpha^{\vee}_v (\overline{r})} (\overline{r} e_1).
$$
Therefore, thanks to \eqref{let us show this} and Lemma~\ref{slice of reduced boundary F}
 \begin{align*}
&P(E ; \partial B (\overline{r}))
= \mathcal{H}^{n-1} (\partial^* E \cap \partial B (\overline{r})) 
= \mathcal{H}^{n-1} ((\partial^* E)_{\overline{r}}) 
\leq \mathcal{H}^{n-1} \left(
\mathbf{B}_{\alpha^{\vee}_v (\overline{r})} (\overline{r} e_1)
\setminus  \mathbf{B}_{\alpha^{\wedge}_v (\overline{r})} (\overline{r} \omega) \right) \\
&=  v^{\vee} (\overline{r}) - v^{\wedge} (\overline{r}) 
 = P(F_v ; \partial B (\overline{r})) \leq P(E ; \partial B (\overline{r})), 
\end{align*}
where we also used \eqref{per ineq} with $B = \{ \overline{r} \}$. 
Then, \eqref{last part} follows from the last chain of inequalities.
\end{proof}
We conclude this section showing that, if $D^c \alpha_v \neq 0$, rigidity fails.
\begin{proposition} \label{no Cantor}
Let $v: (0,\infty) \to [0, \infty)$ be a measurable function 
satisfying \eqref{bound on v} such that $F_v$ 
is a set of finite perimeter and finite volume, 
and let $\alpha_v$ be defined by \eqref{this is the def of alphav}.
Suppose that $D^c \alpha_v \neq 0$.
Then, rigidity fails. 
\end{proposition}

\begin{proof}
We are going to construct a spherically $v$-distributed set $E \in \mathcal{N} (v)$
that cannot be obtained by applying a single orthogonal transformation to $F_v$ 
(see \eqref{set E counterexample} below).

First of all, let us note that it is not restrictive to assume that 
$\alpha_v$ is purely Cantorian.
Indeed, by \eqref{1D decomposition BV} one can decompose 
$\alpha_v$ into 
\begin{equation} \label{alpha is decomposed}
\alpha_v = \alpha^a_v + \alpha^j_v + \alpha^c_v,
\end{equation}
where $\alpha^a_v \in W^{1,1}_{\textnormal{loc}} (0, \infty)$, 
$\alpha^j_v$ is a purely jump function, and $\alpha^c_v$ is purely Cantorian. 
Thanks to \eqref{alpha is decomposed}, in the general case 
when $\alpha_v \neq \alpha_v^c$, the proof can be repeated by applying 
our argument just to the Cantorian part $\alpha^c_v$ of $\alpha_v$.
Therefore, from now on we will assume that
$$
D \alpha_v = D^c \alpha_v.
$$
Thanks to Proposition~\ref{no interval}, we can also assume that 
$\{ 0 < \alpha^{\wedge}_v \leq \alpha^{\vee}_v < \pi \}$ is an interval 
(otherwise there is nothing to prove, since rigidity fails).
Moreover, since $\alpha_v$ is continuous, there exist
 $a, b >0$, with $a < b$, such that $I:= (a, b) \subset \subset \{ 0 < \alpha^{\wedge}_v \leq \alpha^{\vee}_v < \pi \}$ and
\begin{equation} \label{alpha positive}
0 <  \alpha_v (r) < \pi \qquad \text{ for every } r \in I. 
\end{equation}
Since $D^c \alpha_v \neq 0$, 
it is not restrictive to assume $|D^c \alpha_v| (I) > 0$.
For each $\gamma \in (-\pi, \pi)$, we define $R_{\gamma} \in O(n)$ 
in the following way:
$$
R_{\gamma}
\begin{pmatrix} 
x_1  \\
x_2 \\
x_3 \\
\vdots \\
x_n
\end{pmatrix}
= 
\begin{pmatrix} 
 x_1 \cos \gamma -  x_2 \sin \gamma \\
 x_1 \sin \gamma +  x_2 \cos \gamma \\
x_3 \\
\vdots \\
x_n
\end{pmatrix} .
$$
That is, $R_{\gamma}$ is a counterclockwise rotation 
of the angle $\gamma$ in the plane $(x_1, x_2)$.
Let now fix $\lambda \in (0,1)$, and define $\beta : (0, \infty) \to (- \pi, \pi)$ as
$$
\beta (r):=
\begin{cases}
0  & \text{ if } r \in (0, a), \\
\lambda (\alpha_v (r) - \alpha_v (a)) & \text{ if } r \in [a, b], \\
\lambda (\alpha_v (b) - \alpha_v (a)) & \text{ if } r \in (b, \infty).
\end{cases}
$$
We set 
\begin{equation} \label{set E counterexample}
E:= \{ x \in \R^n : \text{dist}_{\mathbb{S}^{n-1}} (\hat{x}, R_{\beta (| x|)} e_1 ) < \alpha^{\vee}_v ( | x |) \}.
\end{equation}
Clearly, $E$ cannot be obtained by applying a single orthogonal transformation to $F_v$.
Let us show that $E \in \mathcal{N} (v)$, so that rigidity fails.
We proceed by steps.

\vspace{.2cm}

\noindent
\textbf{Step 1:} We construct a sequence of functions 
$v^k: I \to [0, \infty)$ satisfying the following properties:

\begin{itemize}

\item[(a)] $\displaystyle \lim_{k \to \infty} \alpha_{v^k} (r) = \alpha_{v} (r)$ for $\mathcal{H}^1$-a.e. $r \in I$;

\vspace{.1cm}

\item[(b)] $D \xi_{v^k} = D^j \xi_{v^k}$ for every $k \in \mathbb{N}$;

\vspace{.2cm}

\item[(c)] $\displaystyle \lim_{k \to \infty} P (F_{v^k}; \Phi (I \times \mathbb{S}^{n-1}))
= P (F_v; \Phi (I \times \mathbb{S}^{n-1}))$.
 
\end{itemize}
First of all note that, by \eqref{def mathcal F} 
and by the chain rule in $BV$ (see, \cite[Theorem~3.96]{AFP}), 
it follows that $\xi_v$ is purely Cantorian, 
where $\xi_v$ is given by \eqref{xi def}. 
Moreover, from \eqref{formula total variations in 1D} 
and from the fact that $\xi_v$ is continuous, we have 
$$
| D \xi_v | (I) = 
\sup \left\{ \sum_{i=1}^{N-1} | \xi_v (r_{i+1})  -   \xi_v (r_i) | :
a < r_1 < r_2 < \ldots < r_N < b  \right\},
$$ 
where the supremum runs over $N \in \mathbb{N}$ and over 
all $r_1, \ldots, r_N$ with $a < r_1 < r_2 < \ldots < r_N < b$.
Therefore, for every $k \in \mathbb{N}$ there exist
$N_k \in \mathbb{N}$ and $r^k_1, \ldots, r^k_N$ 
with $a < r^k_1 < r^k_2 < \ldots < r^k_N < b$ such that 
$$
| D \xi_v | (I) \leq \sum_{i=1}^{N_k-1} | \xi_v (r^k_{i+1})  -   \xi_v (r^k_i) | +\frac{1}{k}
$$
and 
$$
| r^k_{i+1} - r^k_{i}|< \frac{1}{k} \qquad \text{ for every } i = 1, \ldots, N_k -1.
$$
Without any loss of generality, we can assume that the partitions
are increasing in $k$. That is, we will assume that 
$$
\{ r^k_1, \ldots, r^k_{N_k} \} \subset \{ r^{k+1}_1, \ldots, r^{k+1}_{N_{k+1}} \} \quad \text{ for every }k \in \mathbb{N}.
$$
Define now, for every $k \in \mathbb{N}$, 
\begin{equation} \label{def sequence xik}
\xi^k_v (r) :=  \sum_{i=0}^{N_k} \xi_v (r^k_i) \chi_{\big[ r_i^k,r_{i+1}^k \big)}( r ),
\end{equation}
where we set $r^k_0 := a$ and $r^k_{N_k +1} := b$.
Let us now set 
$$
v^k (r) := \xi^k_v (r)/ r^{n-1} \qquad \text{ for every } r \in I \text{ and for every } k \in \mathbb{N}, 
$$
and note that, by definition, $\xi^k_v = \xi_{v^k}$.
Since $\xi_v$ is continuous, we have that 
\begin{equation} \label{a e convergence}
\lim_{k \to \infty}  \xi^k_v (r) =  \xi_v (r) \qquad \text{ for $\mathcal{H}^1$-a.e. } r \in I.
\end{equation}
Recalling \eqref{def mathcal F} and \eqref{alpha}, last relation implies property (a).
Moreover, from \eqref{def sequence xik} we have (b).

Let us now show (c).
Thanks to \eqref{alpha positive} and \eqref{a e convergence}, we have 
\begin{equation} \label{a e convergence p_F}
\lim_{k \to \infty} p_{F^k_v}(r) = p_{F_v}(r) \qquad \text{ for $\mathcal{H}^1$-a.e. } r \in I.
\end{equation}
Moreover, 
\begin{align}
& | D \xi^k_v | (I) 
= \sum_{i=0}^{N_k} | \xi_v (r^k_{i+1}) - \xi_v (r^k_i) | \label{total variation xik} \\
&= | \xi_v (r^k_{1}) - \xi_v (a) | + | \xi_v (b) - \xi_v (r^k_{N_k}) | 
+ \sum_{i=1}^{N_k - 1} | \xi_v (r^k_{i+1}) - \xi_v (r^k_i) |. \nonumber
\end{align}
Since 
\begin{align*}
| D \xi_v | (I) - \frac{1}{k} \leq \sum_{i=1}^{N_k - 1} | \xi_v (r^k_{i+1}) - \xi_v (r^k_i) | 
\leq | D \xi_v | (I),
\end{align*}
using \eqref{total variation xik} and the fact that $\xi_v$ is continuous we obtain 
\begin{equation} \label{convergence total variation}
| D \xi_v | (I) 
= \lim_{k \to \infty} \sum_{i=1}^{N_k - 1} | \xi_v (r^k_{i+1}) - \xi_v (r^k_i) |
= \lim_{k \to \infty} | D \xi^k_v | (I) .
\end{equation}
Thanks to \cite[Theorem~3.23]{AFP}, up to subsequences 
$\xi^k_v$ weakly* converges in $BV(I)$ to $\xi_v$. 
Since, in addition, \eqref{convergence total variation} holds true, 
 we can apply \cite[Proposition~1.80]{AFP}
to the sequence of measures $\{ | D \xi^k_v| \}_{k \in \mathbb{N}}$. 
Therefore, recalling that
$D \xi^k_v =   D^s \xi^k_v $ and $D \xi_v =   D^s \xi_v $, 
we have 
$$
\lim_{k \to \infty} \int_I \, r^n d | D^s \xi^k_v | (r) 
= \lim_{k \to \infty} \int_I \, r^n d | D \xi^k_v | (r) 
=  \int_I \, r^n d | D \xi_v | (r) = \int_I \, r^n d | D^s \xi_v | (r). 
$$
Then, from Corollary~\ref{corollary perimeter Fv}
\begin{align*}
&\lim_{k \to \infty} P (F_{v^k}; \Phi (I \times \mathbb{S}^{n-1}))
=\lim_{k \to \infty}  \left( \int_I p_{F_{v^k}}(r) \, dr
+ \int_I r^{n-1} d | D^s \xi^k_v | (r) \right) \\
&= \left( \int_I p_{F_v}(r) \, dr
+ \int_I r^{n-1} d | D^s \xi_v | (r) \right)
= P (F_v; \Phi (I \times \mathbb{S}^{n-1})),
\end{align*}
where we also used \eqref{a e convergence p_F}.

\vspace{.2cm}

\noindent
\textbf{Step 2:} For each $k \in \mathbb{N}$, 
we construct a spherically $v^k$-distributed set $E^ k$ such that
\begin{equation*} 
P (E^k ; \Phi (I \times \mathbb{S}^{n-1})) =  P (F_{v^k}; \Phi (I \times \mathbb{S}^{n-1})).
\end{equation*}
From \eqref{def mathcal F} and \eqref{alpha} it follows that 
$\alpha_{v^k}  = \mathcal{F}^{-1} (\xi^k_v) \in BV (I)$, and 
\begin{equation} \label{alphak is a sum}
\alpha_{v^k} (r) =   \sum_{i=0}^{N_k} \alpha_v (r^k_i) \chi_{\big[ r_i^k,r_{i+1}^k \big)}( r ).
\end{equation}
Therefore, for each $k \in \mathbb{N}$ we have that 
$D \alpha_{v^k} = D^j \alpha_{v^k}$, 
and the jump set of $\alpha_{v^k}$ is a finite set.
More precisely, 
$$
D \alpha_{v^k} =  
\sum_{i=1}^{N_k} ( \alpha_v (r^k_{i}) - \alpha_v (r^k_{i-1}) ) \delta_{r^k_i},
$$
where $\delta_r$ denotes the Dirac delta measure concentrated at $r$. 
Let $\lambda \in (0, 1)$ be fixed, and define the set $E^k_1 \subset \Phi (I \times \mathbb{S}^{n-1})$ as
$$
E^k_1 := \left[ F_{v^k} \cap ( B (r^k_1) \setminus \overline{B (a)}) \right] \cup 
\left[ R_{\lambda (\alpha_v (r^k_{1}) - \alpha_v (a)) } ( F_{v^k} \cap ( B (b) \setminus B (r_1^k)) ) \right].
$$
Thanks to Proposition~\ref{jump set implies no rigidity}, we have that 
$$
P (E^k_1 ; \Phi (I \times \mathbb{S}^{n-1})) =  P (F_{v^k}; \Phi (I \times \mathbb{S}^{n-1})).
$$
Define now $E^k_2 \subset \Phi (I \times \mathbb{S}^{n-1})$ as
$$
E^k_2 := ( E^k_1 \cap B (r^k_2) ) \cup 
\left[ R_{\lambda (\alpha_v (r^k_{2}) - \alpha_v (r^k_{1})) } ( E^k_1 \setminus B (r^k_2) ) \right].
$$
Applying again Proposition~\ref{jump set implies no rigidity}, we have 
$$
P (E^k_2 ; \Phi (I \times \mathbb{S}^{n-1}))
= P (E^k_1 ; \Phi (I \times \mathbb{S}^{n-1})) =  P (F_{v^k}; \Phi (I \times \mathbb{S}^{n-1})).
$$
Note that, since $R_{\gamma}$ is associative with respect to $\gamma$ 
(that is, we have $R_{\gamma_1} R_{\gamma_2} = R_{\gamma_1 + \gamma_1}$), 
we can write $E^k_2$ as
\begin{align*}
E^k_2 &= \left[ F_{v^k} \cap ( B (r^k_1) \setminus \overline{B (a)}) \right] \cup 
\left[ R_{\lambda (\alpha_v (r^k_{1}) - \alpha_v (a)) } ( F_{v^k} \cap ( B (r^k_2) \setminus B (r_1^k)) ) \right] \\
&\hspace{.2cm} 
\cup \left[ R_{\lambda (\alpha_v (r^k_{2}) - \alpha_v (a)) } ( F_{v^k} \cap ( B (b) \setminus B (r^k_2)) ) \right].
\end{align*}
Iterating this procedure $N_k$ times, we obtain that 
$$
P (E^k ; \Phi (I \times \mathbb{S}^{n-1}))
 =  P (F_{v^k}; \Phi (I \times \mathbb{S}^{n-1})),
$$
where
\begin{equation} \label{def ek new}
E_k:= E^k_{N_k} 
= \{ x \in \Phi (I \times \mathbb{S}^{n-1}) : 
\textnormal{dist}_{\mathbb{S}^{n-1}} (\hat{x}, R_{\lambda ( \alpha_{v^k}(|x|) - \alpha_{v^k}(a)} ) e_1) 
<  \alpha_{v^k} (|x|) \}. 
\end{equation}

\vspace{.2cm}

\noindent
\textbf{Step 3:} We show that 
$E^k  \longrightarrow \widehat{E}$  in  $\Phi(I \times \mathbb{S}^{n-1})$,
for some spherically $v$-distributed set $\widehat{E}$ such that
$$
P (\widehat{E} ; \Phi (I \times \mathbb{S}^{n-1}))
= P (F_v; \Phi (I \times \mathbb{S}^{n-1})).
$$
From \eqref{alphak is a sum} and \eqref{a e convergence}
it follows that
$$
\lim_{k \to \infty}  \alpha_{v^k} (r) =  \alpha_v (r) \qquad \text{ for $\mathcal{H}^1$-a.e. } r \in I.
$$
Therefore, from \eqref{def ek new} we have 
$E^k  \longrightarrow \widehat{E}$ ( in $(\Phi(I \times \mathbb{S}^{n-1}))$),
where $\widehat{E}$ is the spherically $v$-distributed set in $\Phi(I \times \mathbb{S}^{n-1})$
given by 
\begin{equation} \label{definition Ebar}
\widehat{E}  := \{ x \in \Phi (I \times \mathbb{S}^{n-1}) : 
\textnormal{dist}_{\mathbb{S}^{n-1}} (\hat{x}, R_{\lambda (\alpha_{v}(|x|) - \alpha_{v}(a))} e_1) 
<  \alpha_{v} (|x|) \}.
\end{equation}
Then, by the lower semicontinuity of the perimeter 
with respect to the $L^1$ convergence
(see, for instance, \cite[Proposition~12.15]{maggiBOOK}):
\begin{align*}
&P (\widehat{E} ; \Phi (I \times \mathbb{S}^{n-1}))
\leq \lim_{k \to \infty} P (E^k; \Phi (I \times \mathbb{S}^{n-1})) \\
&\lim_{k \to \infty} P (F_{v^k}; \Phi (I \times \mathbb{S}^{n-1}))
= P (F_v; \Phi (I \times \mathbb{S}^{n-1})) \\
&\leq P (\widehat{E} ; \Phi (I \times \mathbb{S}^{n-1})),
\end{align*}
where we also used \eqref{per ineq}.

\vspace{.2cm}

\noindent
\textbf{Step 4:} We conclude. 
Let $E$ be given by \eqref{set E counterexample}.
Then, $E$ is spherically $v$-distributed and satisfies
$$
E =_{\mathcal{H}^n} ( F_v \cap (B (a)) ) 
\cup \left[  \widehat{E} \cap (B(b) \setminus B (a)) \right]
\cup \left[ R_{\lambda (\alpha_v (b) - \alpha_v (a))} (F_v \setminus (B (b)) ) \right], 
$$
where $\widehat{E}$ is defined in \eqref{definition Ebar}.
By repeating the arguments used in the proof
of Proposition~\ref{no interval}, and using the fact that 
$\Phi (I \times \mathbb{S}^{n-1}) = B (b) \setminus \overline{B (a)}$, 
one can see that 
\begin{align*}
P (E) &= P(E; B (a)) + P(E;  \partial B (a))
+ P(E; B (b) \setminus \overline{B (a)}) \\
&+ P(E;  \partial B (b)) +P(E; \R^n \setminus \overline{B (b)}) \\
&= P(F_v; B (a)) + P(E;  \partial B (a))
+ P(\widehat{E} ; B (b) \setminus \overline{B (a)}) \\
&+ P(E;  \partial B (b)) +P(F_v; \R^n \setminus \overline{B (b)}) \\
&= P(F_v; B (a)) + P(E;  \partial B (a))
+ P( F_v; B (b) \setminus \overline{B (a)}) \\
&+ P(E;  \partial B (b)) +P(F_v; \R^n \setminus \overline{B (b)}),
\end{align*}
where we also used Step 3 and the invariance of the perimeter under orthogonal transformations. 
Since $\alpha_v$ is continuous, an argument similar to the one used to prove \eqref{let us show this} 
shows that 
$$
P(E;  \partial B (a)) = P(E;  \partial B (b)) = 0.
$$ 
Therefore, 
$$
P (E) = P(F_v; B (a)) 
+ P( F_v; B (b) \setminus \overline{B (a)}) +P(F_v; \R^n \setminus \overline{B (b)})
= P(F_v).
$$
\end{proof}
We can now give the proof of the implication (i) $\Longrightarrow$ (ii)
of Theorem~\ref{rigidity theorem}.
\begin{proof}[Proof of Theorem~\ref{rigidity theorem}: (i) $\Longrightarrow$ (ii)]
To show the implication, it suffices to combine 
Proposition~\ref{no interval}, Proposition~\ref{jump set implies no rigidity}, 
and Proposition~\ref{no Cantor}.
\end{proof}

\section{Acknowledgements}

The authors would like to thank Marco Cicalese, Nicola Fusco,
and Emanuele Spadaro for inspiring discussions on the subject.
They would also like to thank Frank Morgan for useful comments on a preliminary version
of the paper.
F. Cagnetti was supported by the EPSRC under the Grant EP/P007287/1 ``Symmetry of Minimisers in Calculus of Variations''.


\def\cprime{$'$}

\end{document}